\documentclass{article}

\usepackage{amsmath} 
\usepackage{amsthm}        
\usepackage{amssymb}       
\usepackage{braket}       
\usepackage{amsfonts}
\usepackage[all]{xy}
\usepackage{xspace}
\usepackage{calc}
\usepackage{comment}
\usepackage{pgfplots}
\usepgflibrary{fpu}
\usepackage{mathtools}
\usepackage{tabularx}
\usepackage{ifthen}             
\usepackage{graphicx}
\usepackage{docmute}
\usepackage{array}
\usepackage{subfloat} 
\usepackage{textcomp}
\usepackage{bbm}            
\usepackage{etoolbox}  
\usepackage{verbatim}
\usepackage{stmaryrd}
\usepackage{cancel}
\usepackage{xcolor}
\colorlet{Black}{black}
\usepackage{tensor}
\usepackage{enumerate}  
\usepackage{thm-restate}
\usepackage[numbers,sort]{natbib}
\usepackage{enumitem}
\usepackage{tocloft}
\usepackage{pdfpages}
\usepackage[export]{adjustbox}

\usepackage[margin=3cm]{geometry}

\newenvironment{tz}[1][]{%
                                \begin{tikzpicture}[baseline={([yshift=-.8ex]current bounding                        box.center)},#1] %
                                }{%
                        \end{tikzpicture} %
                        }

\DeclareRobustCommand{\SkipTocEntry}[5]{}

\usepackage{tikz}
\usetikzlibrary{matrix}
\usetikzlibrary{decorations.pathreplacing}
\usetikzlibrary{decorations.markings}
\usetikzlibrary{arrows}
\usetikzlibrary{calc}
\usetikzlibrary{shapes.misc}
\usetikzlibrary{fit}
\usepgflibrary{decorations.pathmorphing}
\usepgflibrary{shapes.geometric}

\pgfdeclarelayer{edgelayer}
\pgfdeclarelayer{nodelayer}
\pgfdeclarelayer{foreground}
\pgfsetlayers{edgelayer,nodelayer,main,foreground}

\tikzstyle{none}=[inner sep=0pt]

\tikzstyle{rn}=[circle,fill=Red,draw=Black,line width=0.8 pt]
\tikzstyle{gn}=[circle,fill=Lime,draw=Black,line width=0.8 pt]
\tikzstyle{bl}=[circle,fill=Blue,draw=Black,line width=0.8 pt]

\tikzstyle{simple}=[-,draw=Black,thick]
\tikzstyle{arrow}=[-,draw=Black,postaction={decorate},decoration={markings,mark=at position .5 with {\arrow{>}}},thick]
\tikzstyle{tick}=[-,draw=Black,postaction={decorate},decoration={markings,mark=at position .5 with {\draw (0,-0.1) -- (0,0.1);}},line width=2.000]
\makeatother\def\thickness{0.7pt}
\tikzstyle{dot}=[circle, draw=black, fill=black, inner sep=.5ex, line width=\thickness, node on layer=foreground]

\makeatletter
\pgfkeys{%
  /tikz/on layer/.code={
    \pgfonlayer{#1}\begingroup
    \aftergroup\endpgfonlayer
    \aftergroup\endgroup
  },
  /tikz/node on layer/.code={
     \gdef\node@@on@layer{%
      \setbox\tikz@tempbox=\hbox\bgroup\pgfonlayer{#1}\unhbox\tikz@tempbox\endpgfonlayer\egroup}
    \aftergroup\node@on@layer
  },
  /tikz/end node on layer/.code={
    \endpgfonlayer\endgroup\endgroup
  }
}

\def\node@on@layer{\aftergroup\node@@on@layer}

\makeatletter
\def\calign@preamble{%
   &\hfil\strut@
    \setboxz@h{\@lign$\m@th\displaystyle{##}$}%
    \ifmeasuring@\savefieldlength@\fi
    \set@field
    \hfil
    \tabskip\alignsep@
}
\let\cmeasure@\measure@
\patchcmd\cmeasure@{\divide\@tempcntb\tw@}{}{}{}
\patchcmd\cmeasure@{\divide\@tempcntb\tw@}{}{}{}
\patchcmd\cmeasure@{\ifodd\maxfields@
  \global\advance\maxfields@\@ne
  \fi}{}{}{}    
\newenvironment{calign}
{%
  \let\align@preamble\calign@preamble
  \let\measure@\cmeasure@
  \align
}
{%
  \endalign
}  
\makeatother
\tikzset{
    master/.style={
        execute at end picture={
            \coordinate (lower right) at (current bounding box.south east);
            \coordinate (upper left) at (current bounding box.north west);
        }
    },
    slave/.style={
        execute at end picture={
            \pgfresetboundingbox
            \path (upper left) rectangle (lower right);
        }
    }
}

\tikzset{blob/.style={draw, circle, fill=white, inner sep=1pt, minimum width=15pt, font=\scriptsize, line width=0.7pt}}
\tikzset{greenregion/.style={fill=green, fill opacity=0.3, draw=none}}
\tikzset{redregion/.style={fill=red, fill opacity=0.3, draw=none}}
\tikzset{blueregion/.style={fill=blue, fill opacity=0.3, draw=none}}
\tikzset{yellowregion/.style={fill=yellow, fill opacity=0.5, draw=none}}
\tikzset{cyanregion/.style={fill=cyan, fill opacity=0.3, draw=none}}
\tikzset{orangeregion/.style={fill=orange, fill opacity=0.6, draw=none}}
\tikzset{solidgreenregion/.style={fill=green!30, fill opacity=1, draw=none}}
\tikzset{solidredregion/.style={fill=red!30, fill opacity=1, draw=none}}
\tikzset{solidblueregion/.style={fill=blue!30, fill opacity=1, draw=none}}
\tikzset{solidyellowregion/.style={fill=yellow!30, fill opacity=1, draw=none}}
\tikzset{string/.style={line width=0.7pt}}
\tikzset{zig/.style={decoration={zigzag,segment length=3, amplitude=0.5}}}
\tikzset{bnd/.style={draw,string}}   
\tikzset{projector/.style={circle, draw, font=\scriptsize, inner sep=-5pt, minimum width=0.35cm, string, fill=white}}
\tikzset{dimension/.style={font=\scriptsize, inner sep=1pt}}

\tikzset{arrow data/.style 2 args={
      decoration={
         markings,
         mark=at position #1 with \arrow{#2}},
         postaction=decorate}
}

\tikzset{along path/.style={every path/.style={}, sloped, allow upside down}}

\def\zxnormal {
                \def \zxscale{0.55}
                \def\zxnodescale{0.8}
                \def\vertexscale{0.7}
                \def\zxshift{0.075cm}
                \def\hadscale{0.8}
                \def\trianglescale{1}
                \def\boxscale{1}
                }

\def\zxgreen{white}

\def\zxwhite{white}
\def\zxblack{black!50}

\tikzset{front/.style ={node on layer=foreground}}
\tikzset{zx/.style = {string, scale=\zxscale}}
\tikzset{zxnode/.style n args={1}{blob,scale=\zxnodescale,fill=#1,node on layer=foreground}}
\tikzset{box/.style={draw, rectangle, fill=white, inner sep=1pt, minimum width=10pt,minimum height=10pt, font=\scriptsize, line width=0.7pt,scale=\zxnodescale,node on layer=foreground}}
\tikzset{boxvertex/.style={draw, rectangle, fill=white, line width=0.733pt,scale=0.75*\vertexscale}}
\tikzset{bigbox/.style={draw, rectangle, fill=white,  minimum width=\boxscale *18pt,minimum height=\boxscale*8pt, line width=0.7pt,scale=\zxnodescale}}
\newlength{\unitbox}
\tikzset{widebox/.style ={draw,rectangle, fill=white, line width=0.7pt,scale=0.75*\zxnodescale,minimum height=15pt,inner sep=1pt,  minimum width = \unitbox,   anchor=center }}
\tikzset{wideboxm/.style n args={1}{draw,rectangle, fill=white, line width=0.7pt,scale=0.75*\zxnodescale,minimum height=15pt,inner sep=1pt,  minimum width =2\unitbox+#1\unitbox,   anchor=center }}
\tikzset{triangleup/.style n args={1}{draw, shape=isosceles triangle, isosceles triangle stretches, fill=white, line width=0.7pt,scale=0.75*\zxnodescale,minimum height=15pt,inner sep=1pt,  minimum width = #1*\trianglescale cm +0.15*\trianglescale cm,  shape border rotate=90, anchor=south }}
\tikzset{triangledown/.style n args={1}{draw, shape=isosceles triangle, isosceles triangle stretches, fill=white, line width=0.7pt,scale=0.75*\zxnodescale,minimum height=15pt,inner sep=1pt,  minimum width = #1*\trianglescale cm +0.15*\trianglescale cm,  shape border rotate=-90, anchor=north }}
\tikzset{zxvertex/.style n args={1}{draw,fill=#1,circle,line width=0.7pt,scale=0.75*\vertexscale}}
\tikzset{zxdown/.style={yshift=-\zxshift}}
\tikzset{zxup/.style={yshift=\zxshift}}

\newcommand\mult[3]{ 
\draw[string] (#1.center) to [out=up, in=-135] +(0.5*#2,#3) to [out=-45, in=up] +(0.5*#2,-#3);
\node[zxvertex=\zxgreen,zxdown] at ($(#1)+(0.5*#2,#3)$){};
}
\newcommand\comult[3]{ 
\draw[string] (#1.center) to [out=down, in=135] +(0.5*#2,-#3) to [out=45, in=down] +(0.5*#2,#3);
\node[zxvertex=\zxgreen,zxup] at ($(#1) +(0.5*#2,-#3)$){};}

\newcommand\unit[2]{ 
\draw[string] (#1.center) to + (0, -#2);
\node[zxvertex=\zxgreen] at ($(#1) +(0,-#2)$){};
}

\newcommand{\Tr}{\mathrm{Tr}}

\renewcommand{\to}[1][]{\ensuremath{\xrightarrow{#1}}}

\allowdisplaybreaks[1]

\theoremstyle{plain} 

\newtheorem{theorem}{Theorem}[section]
\newtheorem{lemma}[theorem]{Lemma}
\newtheorem{corollary}[theorem]{Corollary}          
\newtheorem{proposition}[theorem]{Proposition}

\theoremstyle{definition} 
\newtheorem{definition}[theorem]{Definition}
\newtheorem{construction}[theorem]{Construction}
\newtheorem{remark}[theorem]{Remark}

\newtheorem{example}[theorem]{Example}

\theoremstyle{remark}  

\newtheoremstyle{special_statement} 
        {\topskip}
        {\topskip}
        {\addtolength{\leftskip}{2.5em} \itshape }
        {}
        {\bfseries}
        {:}
        {.5em}
        {}
\theoremstyle{special_statement}

\DeclareMathOperator{\Hom}{Hom}
\DeclareMathOperator{\End}{End}

\DeclareMathOperator{\Fun}{Fun}
\newcommand{\id}{\mathrm{id}}

\newcommand{\wt}{\mathrm{wt}}
\newcommand{\For}{\mathrm{For}}
\renewcommand{\Vec}{\mathrm{Vec}}

\newcommand{\Rep}{\mathrm{Rep}}

\newcommand{\Hilb}{\ensuremath{\mathrm{Hilb}}}

\newcommand{\Obj}{\ensuremath{\mathrm{Obj}}}

\newcommand{\QGraph}{\ensuremath{\mathrm{QGraph}}}

\newcommand{\Corep}{\ensuremath{\mathrm{Corep}}}

\newcommand{\F}{\ensuremath{\mathrm{SSFA}}}

\newcommand{\Gal}{\ensuremath{\mathrm{Gal}}}
\newcommand{\Fib}{\ensuremath{\mathrm{Fib}}}

\makeatletter
\DeclareFontFamily{OMX}{MnSymbolE}{}
\DeclareSymbolFont{MnLargeSymbols}{OMX}{MnSymbolE}{m}{n}
\SetSymbolFont{MnLargeSymbols}{bold}{OMX}{MnSymbolE}{b}{n}
\DeclareFontShape{OMX}{MnSymbolE}{m}{n}{
    <-6>  MnSymbolE5
   <6-7>  MnSymbolE6
   <7-8>  MnSymbolE7
   <8-9>  MnSymbolE8
   <9-10> MnSymbolE9
  <10-12> MnSymbolE10
  <12->   MnSymbolE12
}{}
\DeclareFontShape{OMX}{MnSymbolE}{b}{n}{
    <-6>  MnSymbolE-Bold5
   <6-7>  MnSymbolE-Bold6
   <7-8>  MnSymbolE-Bold7
   <8-9>  MnSymbolE-Bold8
   <9-10> MnSymbolE-Bold9
  <10-12> MnSymbolE-Bold10
  <12->   MnSymbolE-Bold12
}{}

\let\llangle\@undefined
\let\rrangle\@undefined
\DeclareMathDelimiter{\llangle}{\mathopen}%
                     {MnLargeSymbols}{'164}{MnLargeSymbols}{'164}
\DeclareMathDelimiter{\rrangle}{\mathclose}%
                     {MnLargeSymbols}{'171}{MnLargeSymbols}{'171}
\makeatother

\usepackage[colorlinks=true, linkcolor=black, citecolor=black, urlcolor=black, hypertexnames=false
        ]{hyperref}

\newcounter{DRcomment}
\newcounter{DVcomment}
\newcounter{BMcomment}
\newcounter{JVcomment}
\newcommand\ignore[1]{}

\tikzstyle{blackdot}=[circle, draw=black, fill=black, inner sep=.5ex, line width=\thickness, node on layer=foreground]
\tikzstyle{whitedot}=[circle, draw=black, fill=white, inner sep=.5ex, line width=\thickness, node on layer=foreground]

\tikzset{proofdiagram/.style={scale=1}}
\newlength\morphismheight
\newlength\minimummorphismwidth
\newlength\stateheight
\usepackage[utf8]{inputenc}

\title{Unitary transformations of fibre functors}
\author{Dominic Verdon}
\date{School of Mathematics \\ University of Bristol \\[1ex] \href{mailto:dominic.verdon@bristol.ac.uk}{dominic.verdon@bristol.ac.uk} \\[2ex]
\today}

\begin{document}

\normalsize
\zxnormal
\maketitle

\begin{abstract}
We study unitary pseudonatural transformations (UPTs) between fibre functors $\Rep(G) \to \Hilb$, where $G$ is a compact quantum group. For fibre functors $F_1, F_2$ we show that the category of UPTs $F_1 \to F_2$ and modifications is isomorphic to the category of finite-dimensional $*$-representations of the corresponding bi-Hopf-Galois object. We give a constructive classification of fibre functors accessible by a UPT from the canonical fibre functor, as well as UPTs themselves, in terms of Frobenius algebras in the category $\Rep(A_G)$, where $A_G$ is the Hopf $*$-algebra dual to the compact quantum group. As an example, we show that finite-dimensional quantum isomorphisms from a quantum graph $X$ are UPTs between fibre functors on $\Rep(G_X)$, where $G_X$ is the quantum automorphism group of $X$.
\end{abstract}

\maketitle
\section{Introduction}
\footnotetext{2020 \emph{Mathematics Subject Classification}. Primary 20G42, 18M40, 81R50; Secondary 18M30.}
\subsection{Overview}
Compact symmetry groups play a crucial role in quantum physics. By Tannaka duality, a compact group $G$ is interchangeable with its category of finite dimensional representations $\Rep(G)$, with canonical unitary monoidal \emph{fibre functor} $F: \Rep(G) \to \Hilb$, where $\Hilb$ is the category of finite-dimensional Hilbert spaces and linear maps. One can therefore equivalently say that categories of representations of compact groups with a fibre functor play a crucial role in quantum physics. One way to interpret such categories is as encoding a consistent structure of system types, fusion rules, permissible transitions, etc. The fibre functor then assigns state spaces to all of the systems in the theory, and can be seen as a representation of this compositional structure as part of finite-dimensional quantum theory. Schematically:
\begin{calign}
&\textrm{Compositional category} &\to &\textrm{Representation}
\\
&\Rep(G) & \to & \Hilb
\end{calign}
One can generalise this notion of representation to theories more general than those described by the representation category of a compact group. In particular, we can consider representations of \emph{$C^*$-tensor categories with conjugates}. Such categories and their fibre functors are described by the representation theory of compact quantum groups, and their associated Hopf-Galois objects.

In~\cite{Verdon2020a} we introduced a notion of \emph{unitary pseudonatural transformation} relating two monoidal functors, or more generally two pseudofunctors. This paper is a study of unitary pseudonatural transformations between fibre functors on $C^*$-tensor categories with conjugates, or equivalently (provided a fibre functor exists) representation categories of compact quantum groups. The physical significance of these transformations will be explained in forthcoming work. 
 
\paragraph{Unitary pseudonatural transformations.} Unitary pseudonatural transformations are a generalisation of unitary monoidal natural isomorphisms, defined as follows. Let $\mathcal{C}$ be a $C^*$-tensor category with conjugates, and let $F,F': \mathcal{C} \to \Hilb$ be fibre  functors. Then a unitary pseudonatural transformation specifies:
\begin{itemize}
\item A Hilbert space $H$.
\item For every object $X$ of $\mathcal{C}$, a unitary linear map $F(X) \otimes H \to H \otimes F'(X)$.
\end{itemize}
These unitaries must obey equations generalising the monoidality and naturality conditions for unitary monoidal natural isomorphisms $F \to F'$, which are recovered when $H = \mathbb{C}$.

\paragraph{Hopf-Galois theory.}
Let $\mathcal{C}$ be a $C^*$-tensor category with conjugates. Whenever a fibre functor $F:\mathcal{C} \to \Hilb$ exists, we can construct a monoidal equivalence $\mathcal{C} \simeq \Rep(G)$ for a \emph{compact quantum group} $G$. The category $\mathcal{C}$ can therefore be understood in terms of the compact quantum group $G$, or rather its dual Hopf $*$-algebra $A_G$.

Let $F_1,F_2: \mathcal{C} \to \Hilb$ be fibre functors corresponding to compact quantum groups $G_1,G_2$. Then one can construct an $A_{G_2}$-$A_{G_1}$-bi-Hopf-Galois object $Z$ linking the two fibre functors. This is a $*$-algebra with a compatible left and right coactions of the algebras $A_{G_2},A_{G_1}$. 

Here we show (Theorem~\ref{thm:higherhopfgal}) that there is an isomorphism of categories between:
\begin{itemize}
\item The category $\Rep(Z)$ of finite-dimensional $*$-representations of $Z$ and intertwining linear maps.
\item The category $\Hom(F_1,F_2)$ of unitary pseudonatural transformations $F_1 \to F_2$ and modifications.
\end{itemize}
This generalises the known fact~\cite[Thm 4.4.1]{Bichon1999} that the 1-dimensional $*$-representations of an $A_{G_2}$-$A_{G_1}$-bi-Hopf-Galois object correspond to unitary monoidal natural transformations $F_1 \to F_2$. 
\ignore{
This theorem implies that a UPT exists between two fibre functors precisely when the corresponding bi-Hopf-Galois object admits a finite-dimensional $*$-representation. In particular, for all finite compact quantum groups, any fibre functor is accessible from the canonical fibre functor by a UPT.}

\paragraph{Morita theory.}

In~\cite{Verdon2020a} we showed that the 2-category $\Fun(\mathcal{C},\Hilb)$ of unitary fibre functors, unitary pseudonatural transformations and modifications has certain nice properties; in particular, it is a pivotal dagger 2-category with split dagger idempotents. This makes it an appropriate setting for Morita theory~\cite[Appendix]{Musto2019}.

Let us fix some fibre functor $F: \mathcal{C} \to \Hilb$. By the results just discussed, the endomorphism category $\End(F)$ of UPTs $F \to F$ and modifications is isomorphic to the category $\Rep(A_G)$ of f.d. $*$-representations of the associated compact quantum group algebra. 

We use Morita theory to classify fibre functors $F'$ such that there exists a UPT $F \to F'$, as well as UPTs $F \to F'$, in terms of certain algebraic structures called \emph{simple Frobenius monoids} in the category $\Rep(A_G)$. In particular, we give constructions setting up a correspondence between the following structures (Theorem~\ref{thm:moritaclass}):
\begin{itemize}
\item Unitary monoidal isomorphism classes of unitary fibre functors accessible from $F$ by a UPT; and Morita equivalence classes of simple Frobenius monoids in $\Rep(A_G)$.
\item Equivalence classes of UPTs $\alpha: F \to F'$ for some accessible fibre functor $F'$; and $*$-isomorphism classes of simple Frobenius monoids in $\Rep(A_G)$.
\end{itemize}
As a consequence, we obtain a concrete construction of fibre functors accessible from $F$ by a UPT in terms of idempotent splitting (Theorem~\ref{thm:splitting}). 

\paragraph{Quantum graph isomorphisms.} As an example of UPTs between fibre functors, we show that, for finite quantum graphs $X,Y$, the finite-dimensional quantum graph isomorphisms $X \to Y$ (Definition~\ref{def:qgraphiso}) considered in quantum information theory~\cite{Atserias2019,Brannan2019} are UPTs between accessible fibre functors on the category $\Rep(G_X)$ of representations of the quantum automorphism group of $X$. This sets up an equivalence between the following 2-categories (Theorem~\ref{thm:qgraphequiv}):
\begin{itemize}
\item $\QGraph_X$: Objects --- quantum graphs f.d. quantum isomorphic to $X$. 1-morphisms --- f.d. quantum isomorphisms. 2-morphisms --- intertwiners.
\item $\Fun(\Rep(G_X),\Hilb)_X$: Objects --- Fibre functors accessible by a UPT from the canonical fibre functor on $\Rep(G_X)$. 1-morphisms --- UPTs. 2-morphisms --- modifications.
\end{itemize}
\ignore{
A quantum graph $X = (A,\Gamma)$~\cite{} is a Frobenius monoid $A$ in $\Hilb$ satisfying a certain symmetry condition, together with a linear map $\Gamma: A \to A$ called its adjacency matrix. The category $\Rep(G_X)$ of quantum automorphisms of 
}
\subsection{Acknowledgements} The author thanks Julien Bichon, Amaury Freslon, Laura Man\v{c}inska, Ashley Montanaro, Benjamin Musto, David Reutter, David Roberson, Changpeng Shao, Jamie Vicary and Makoto Yamashita for useful discussions related to this work. The work was supported by EPSRC.

\subsection{Structure}
In Section~\ref{sec:background} we introduce necessary mathematical background material for this paper. In Section~\ref{sec:hopfgal} we discuss the relationship between UPTs and Hopf-Galois theory. In Section~\ref{sec:morita} we discuss the Morita classification/construction of accessible UPTs and fibre functors. In Section~\ref{sec:qgraphiso} we show that finite-dimensional quantum graph isomorphisms are UPTs.

\section{Background}\label{sec:background}
\subsection{Pivotal dagger categories and their diagrammatic calculus}\label{sec:pivdagbackground}
\subsubsection{Monoidal categories}
We assume the reader is familiar with the definition of a monoidal category~\cite{MacLane2013}. We use the standard coherence theorem~\cite{MacLane1963} to assume that all our monoidal categories are strict, allowing the use of the following well-known diagrammatic calculus~\cite{Selinger2010}. 

We read diagrams from bottom to top. Objects are drawn as wires, while morphisms are drawn as boxes whose type corresponds to their input and output wires. Composition of morphisms is represented by vertical juxtaposition, while monoidal product is represented by horizontal juxtaposition. For example, two morphisms $f:X\to Y $ and $g: Y \to Z$ can be composed as follows:
\begin{calign}
\includegraphics{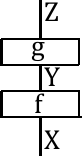}
&
\includegraphics{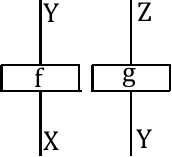} \\
g \circ f: X \to Z 
&
f \otimes g: X \otimes Y \to Y \otimes Z
\end{calign}
The wire for the monoidal unit $\mathbbm{1}$, and the identity morphism $\id_X$ for any object $X$, are invisible in the diagram. Two diagrams which are planar isotopic represent the same morphism~\cite{Selinger2010}. 
\subsubsection{Pivotal categories}
We recall the notion of duality in a monoidal category.
\begin{definition}\label{def:duals}
Let $X$ be an object in a monoidal category. A \emph{right dual} $[X^*,\eta,\epsilon]$ for $X$ is:
\begin{itemize}
\item An object $X^*$.
\item Morphisms $\eta: \mathbbm{1} \to X^* \otimes X$ and $\epsilon: X \otimes X^* \to \mathbbm{1}$ satisfying the following \emph{snake equations}: 
\begin{calign}\label{eq:rightsnakes}
\includegraphics{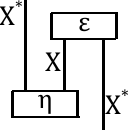}
~~=~~
\includegraphics{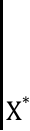}
&
\includegraphics{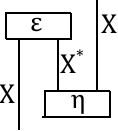}
~~=~~
\includegraphics{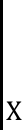}
\end{calign}
\end{itemize}
A \emph{left dual} $[{}^*X,\eta,\epsilon]$ is defined similarly, with morphisms $\eta: \mathbbm{1} \to X \otimes {}^*X$ and $\epsilon: {}^*X \otimes X \to \mathbbm{1}$ satisfying the mirror images of~\eqref{eq:rightsnakes}. 

We say that a monoidal category $\mathcal{C}$ \emph{has right duals} (resp. \emph{has left duals}) if every object $X$ in $\mathcal{C}$ has a chosen right dual $[X^*,\eta,\epsilon]$ (resp. a chosen left dual).
\end{definition}
\noindent
To represent duals in the graphical calculus, we draw an upward-facing arrow on the $X$-wire and a downward-facing arrow on the $X^*$- or ${}^*X$-wire, and write $\eta$ and $\epsilon$ as a cup and a cap, respectively. Then the equations~\eqref{eq:rightsnakes} become purely topological:
\begin{calign}\nonumber
\includegraphics{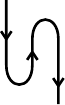}
~~=~~
\includegraphics{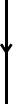}
~~~~~~~~\quad
\includegraphics{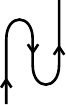}
~~=~~
\includegraphics{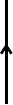}
&&
\includegraphics{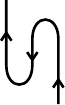}
~~=~~
\includegraphics{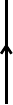}
~~~~~~~~\quad
\includegraphics{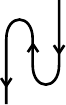}
~~=~~
\includegraphics{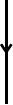}\\\nonumber
\text{right dual} &&\text{left dual}
\end{calign}
\begin{proposition}[{\cite[Lemmas 3.6, 3.7]{Heunen2019}}]\label{prop:nestedduals}
If $[X^*,\eta_X,\epsilon_X]$ and $[Y^*,\eta_Y,\epsilon_Y]$ are right duals for $X$ and $Y$ respectively, then $[Y^* \otimes X^*, \eta_{X \otimes Y},\epsilon_{X \otimes Y}]$ is right dual to $X \otimes Y$, where $\eta_{X \otimes Y}$ and $\epsilon_{X \otimes Y}$ are defined by:
\begin{calign}\nonumber\includegraphics{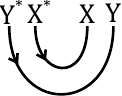}
&&
\includegraphics{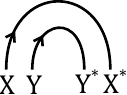}
\\\label{eq:nestedduals}
\eta_{X \otimes Y} && \epsilon_{X \otimes Y}
\end{calign}
Moreover, $[\mathbbm{1},\id_{\mathbbm{1}},\id_{\mathbbm{1}}]$ is right dual to $\mathbbm{1}$. Analogous statements hold for left duals.
\end{proposition}
\noindent
Duals are unique up to isomorphism.
\begin{proposition}[{\cite[Lemma 3.4]{Heunen2019}}]\label{prop:relateduals}
Let $X$ be an object of a monoidal category, and let $[X^*,\eta,\epsilon],[X^*{}',\eta',\epsilon']$ be right duals. Then there is a unique isomorphism $\alpha: X^* \to X^*{}'$ such that 
\begin{calign}\label{eq:relateduals}
\includegraphics{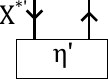}
~~=~~
\includegraphics{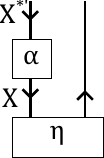}
&
\includegraphics{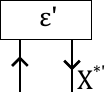}
~~=~~
\includegraphics{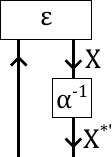}
\end{calign}
An analogous statement holds for left duals.
\end{proposition}
\noindent
In a category with duals, we can define a notion of transposition for morphisms.
\begin{definition}
Let $X,Y$ be objects with chosen right duals $[X^*,\eta_X,\epsilon_X]$ and $[Y^*,\eta_Y,\epsilon_Y]$. For any morphism $f:X \to Y$, we define its \emph{right transpose} $f^T: Y^* \to X^*$ as follows:
\begin{calign}\label{eq:rtranspose}
\includegraphics{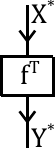}
~~=~~
\includegraphics{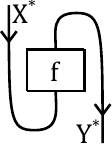}
\end{calign}
For left duals ${}^*X,{}^*Y$, a \emph{left transpose} may be defined analogously.
\end{definition}
\noindent
In this work we are mostly interested in categories with compatible left and right duals. Such categories are called \emph{pivotal}. 
 
Let $\mathcal{C}$ be a monoidal category with right duals. It is straightforward to check that the following defines a monoidal functor $\mathcal{C} \to \mathcal{C}$, which we call the \emph{double duals} functor:
\begin{itemize} 
\item Objects $X$ are taken to the double dual $X^{**}:=(X^*)^*$.
\item Morphisms $f: X \to Y$ are taken to the double transpose $f^{TT}:=(f^T)^T$.
\item The monoidal structure is defined using the isomorphisms of Proposition~\ref{prop:relateduals}.
\end{itemize}
\begin{definition}\label{def:pivcat}
We say that a monoidal category $\mathcal{C}$ with right duals is \emph{pivotal} if the double duals functor is monoidally naturally isomorphic to the identity functor. 
\end{definition}
\noindent
Roughly, the existence of a monoidal natural isomorphism in Definition~\ref{def:pivcat} comes down to the following statement:
\begin{itemize}
\item For every object $X: r \to s$, there is an isomorphism $\iota_X: X^{**} \to X$.
\item These $\{\iota_X\}$ can be chosen compatibly with composition and monoidal product in $\mathcal{C}$.
\end{itemize}
In a pivotal category, for any object $X$ the right dual $X^*$ is also a left dual for $X$ by the following cup and cap (we have drawn a double upwards arrow on the double dual):
\begin{calign}\label{eq:pivldualdef}
\includegraphics{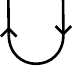}
~~:=~~
\includegraphics{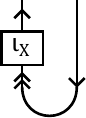}
&
\includegraphics{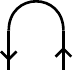}
~~:=~~
\includegraphics{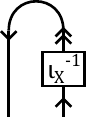}
\end{calign}
With these left duals, the left transpose of a morphism is equal to the right transpose. Whenever we refer to a pivotal category from now on, we suppose that the left duals are chosen in this way. 

There is a very useful graphical calculus for these compatible left and right dualities in a pivotal category.
To represent the transpose, we make our morphism boxes asymmetric by tilting the right vertical edge. We now write the transpose by rotating the boxes, as though we had `yanked' both ends of the wire in the RHS of~\eqref{eq:rtranspose}:
\begin{calign}\nonumber
\includegraphics{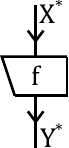}
:=
\includegraphics{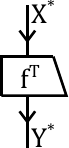}
\end{calign}
Using this notation, morphisms now freely slide around cups and caps.
\begin{proposition}[{\cite[Lemma 3.12, Lemma 3.26]{Heunen2019}}]\label{prop:sliding}
Let $\mathcal{C}$ be a pivotal category and $f:X \to Y$ a morphism. Then:
\begin{calign}\label{eq:sliding}
\includegraphics{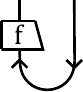}
~~=~~
\includegraphics{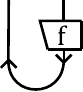}
&
\includegraphics{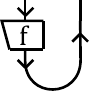}
~~=~~
\includegraphics{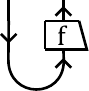}
&
\includegraphics{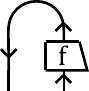}
~~=~~
\includegraphics{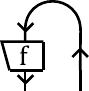}
&
\includegraphics{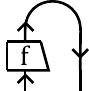}
~~=~~
\includegraphics{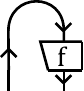}
\end{calign}
\end{proposition}
\noindent
The diagrammatic calculus is summarised by the following theorem.
\begin{theorem}[{\cite[Theorem 4.14]{Selinger2010}}]\label{thm:graphcalcpiv}
Two diagrams for a morphism in a pivotal category represent the same morphism if there is a planar isotopy between them, which may include sliding of morphisms as in Proposition~\ref{prop:sliding}.
\end{theorem}
\noindent
In a pivotal category we can define notions of dimension for objects and trace for morphisms.  
\begin{definition}\label{def:trace}
Let $X$ be an object and let $f: X \to X$ be a morphism in a pivotal category $\mathcal{C}$. We define the \emph{right trace} of $f$ to be the following morphism $\Tr_{R}(f): \mathbbm{1} \to \mathbbm{1}$: 
\begin{calign}\nonumber
\includegraphics{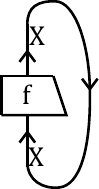}
\end{calign}
We define the \emph{right dimension} $\dim_R(r)$ of an object $X$ of $\mathcal{C}$ to be $\Tr_R(\id_X): \mathbbm{1} \to \mathbbm{1}$. 
The \emph{left traces} $\Tr_L$ and \emph{left dimensions} $\dim_L$ are defined analogously using the right cup and left cap. 
\end{definition}
\begin{definition}
We call a pivotal category $\mathcal{C}$ \emph{spherical} if, for object $X$, and any morphism $f:X\to X$, $\Tr_L(f) = \Tr_R(f) =:\Tr(f)$. In this case we call $\Tr(f)$ and $\dim(f)$ simply the \emph{trace} and the \emph{dimension}.
\end{definition}

\subsubsection{Pivotal dagger categories}
\begin{definition}\label{def:dagcat}
A \emph{dagger structure} on a monoidal category $\mathcal{C}$ is specified by a contravariant identity-on-objects $\dagger:\mathcal{C} \to \mathcal{C}$, written as a power on morphisms, i.e. $\dagger(f) =: f^{\dagger}$, which is:
\begin{itemize}
\item \emph{Involutive}: for any morphism $f: X \to Y$, $\dagger(\dagger(f)) = f$.
\item \emph{Compatible with monoidal product}: for any objects $X,X',Y,Y'$ and morphisms $\alpha: X \to X'$ and $\beta: Y \to Y'$ we have $(\alpha \otimes \beta)^{\dagger} = \alpha^{\dagger}  \otimes \beta^{\dagger}$.
\end{itemize}
\ignore{
\item \emph{Compatible with composition}: For any 2-morphisms $f:X \to Y,g:Y \to Z$ in $\mathcal{C}(r,s)$ and $g$ in $\mathcal{C}(s,t)$, $(f \circ g)^{\dagger_{r,t}} = f^{\dagger_{r,s}} \circ g^{\dagger_{s,t}}$.
\item \emph{Preserves the identity 2-morphisms.} and $(\id_{X})^{\dagger_{r,s}} = \id_X$.
\end{itemize} 
}
We call $f^{\dagger}$ the \emph{dagger} of $f$.
\end{definition}
\noindent
In the graphical calculus, we represent the dagger of a morphism by reflection in a horizontal axis, preserving the direction of any arrows:
\begin{calign}\label{eq:graphcalcdagger}
\includegraphics[scale=1]{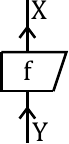}
~~:=~~
\includegraphics[scale=1]{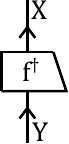}
\end{calign}
\begin{definition}
Let $\mathcal{C}$ be a dagger category.  We say that a morphism $\alpha: X \to Y$ is an \emph{isometry} if $\alpha^{\dagger} \circ \alpha = \id_X$. We say that it is \emph{unitary} if it is an isometry and additionally $ \alpha \circ \alpha^{\dagger}  = \id_Y$.
\end{definition}
\ignore{
\begin{definition}
Let $\mathcal{C}$ be a dagger 2-category and let $r,s$ be objects. We say that a 1-morphism $X: r \to s$ is a \emph{dagger equivalence} if it is an equivalence (Definition~\ref{}) and the invertible 2-morphisms $\alpha: \id_r \to X \circ X^{-1}$ and $\beta: \id_s \to X^{-1} \circ X$ are unitary. 
\end{definition}
}
\noindent
We now give the condition for  compatibility of dagger and pivotal structure.
\begin{definition}\label{def:pivdagcat}
Let $\mathcal{C}$ be a pivotal category which is also a monoidal dagger category. \ignore{Let $X: r \to s$ be a 1-morphism in $\mathcal{C}$, and and let $[X^*,\eta,\epsilon]$ be the chosen right dual; we represent $\eta,\epsilon$ in the graphical calculus by a cup and a cap.} We say that $\mathcal{C}$ is \emph{pivotal dagger} when, for all objects $X$:
\begin{calign}\label{eq:daggerdual}
\includegraphics[scale=1]{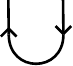}
~~=~~
\left(
\includegraphics[scale=1]{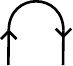}\right)^{\dagger}
&
\includegraphics[scale=1]{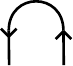}
~~=~~
\left(\includegraphics[scale=1]{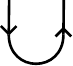}\right)^{\dagger}
\end{calign}
\ignore{
\begin{calign}\label{eq:daggerdual}
\includegraphics[scale=1]{Figures/svg/2cats/daggerdual1.pdf}
~~=~~
\includegraphics[scale=1]{Figures/svg/2cats/daggerdual2.pdf}
\end{calign}
A left dagger dual may be defined analogously.
A \emph{pivotal dagger 2-category} is a 2-category where the chosen right duals are all dagger duals.
} 
\end{definition}
\begin{remark}
For any object $X$ in a monoidal dagger category, a right dual $[X^*,\eta_X,\epsilon_X]$ induces a left dual $[X^*,\epsilon_X^{\dagger},\eta_X^{\dagger}]$. This means that a monoidal dagger category with right duals also has left duals. The pivotal structure gives another way to obtain left duals from right duals~\eqref{eq:pivldualdef}.  The equation~\eqref{eq:daggerdual} implies that the left duals obtained from the dagger structure are the same as those obtained from the pivotal structure.

Practically, when taking the dagger of a cup or a cap in a pivotal dagger category, the equation~\eqref{eq:daggerdual} implies we should reflect the cup or cap in a horizontal axis, preserving the direction of the arrows.
\end{remark}
\noindent
For any morphism $f:X \to Y$, a pivotal dagger structure implies the following \emph{conjugate} morphism $f^*$ is graphically well-defined:
\begin{calign}\label{eq:conjugate}
\includegraphics[scale=1]{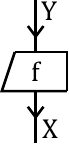}
~~:=~~
\includegraphics[scale=1]{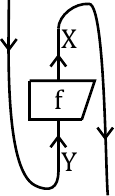}
~~=~~
\includegraphics[scale=1]{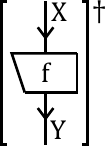}
\end{calign}
\begin{remark}
Following the bra-ket convention, we draw morphisms $f:\mathbbm{1} \to X$ and $f^{\dagger}: X \to \mathbbm{1}$ --- called \emph{states} and \emph{effects} of $X$ respectively --- as triangles rather than as boxes. The morphisms $f$ and $f^{\dagger}$ can be distinguished from $f^T$ and $f^*$ by the direction of the arrows:
\begin{calign}
\includegraphics[scale=1]{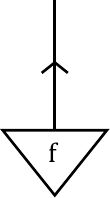}
&
\includegraphics[scale=1]{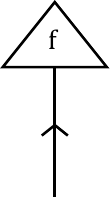}
&
\includegraphics[scale=1]{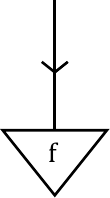}
&
\includegraphics[scale=1]{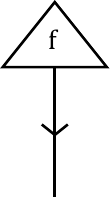}\\
f:\mathbbm{1} \to X
&
f^{\dagger}:X \to \mathbbm{1}
&
f^*:\mathbbm{1} \to X^*
&
f^T: X^* \to \mathbbm{1}
\end{calign}
\end{remark}
\subsubsection{Example: the category $\Hilb$}
A basic example of a pivotal dagger category is the category $\Hilb$. 
The objects of the monoidal category $\Hilb$ are finite-dimensional Hilbert spaces, and the morphisms are linear maps between them; composition of morphisms is composition of linear maps. The monoidal product is given on objects by the tensor product of Hilbert spaces, and on morphisms by the tensor product of linear maps; the unit object is the 1-dimensional Hilbert space $\mathbb{C}$.

For any object $H$, its right dual is defined to be the dual Hilbert space $H^*$. Any basis $\{\ket{v_i}\}$ for $H$ defines a cup and cap:
\def\pv{\vphantom{V^*}}
\begin{calign}\label{eq:cupscapsHilb}
\begin{tz}[zx,xscale=-1]
\draw[arrow data ={0.15}{<}, arrow data={0.89}{<}] (0,0) to [out=up, in=up, looseness=2.5] (2,0) ;
\node[dimension, left] at (2.0,0) {$\pv V$};
\node[dimension, right] at (0,0) {$V^*$};
\end{tz}
&
\begin{tz}[zx,yscale=-1]
\draw[arrow data ={0.15}{>}, arrow data={0.89}{>}] (0,0) to [out=up, in=up, looseness=2.5] (2,0) ;
\node[dimension, right] at (2.05,0) {$\pv V$};
\node[dimension, left] at (0,0) {$V^*$};
\end{tz}\\\nonumber
\ket{v} \otimes \bra{w}\mapsto \braket{w|v}
&
~~1\mapsto \sum_i \bra{v_i} \otimes \ket{v_i}
\end{calign}
It may easily be checked that this cup and cap fulfil the snake equations~\eqref{eq:rightsnakes}. This duality is pivotal; the monoidal natural isomorphism from the identity functor to the double duals functor is given by the standard isomorphism from a Hilbert space to its double dual.

The dagger structure is given by the Hermitian adjoint of a linear map. As long as the basis $\{\ket{v_i}\}$ is orthonormal $\Hilb$ is pivotal dagger. The transpose~\eqref{eq:rtranspose} and conjugate~\eqref{eq:conjugate} are simply the usual transpose and complex conjugate of a linear map with respect to the orthonormal basis defining the duality.

In fact, $\Hilb$ is a \emph{compact closed} category --- it is symmetric monoidal in a way which is compatible with its pivotal dagger structure. Because it is symmetric monoidal,  diagrams in $\Hilb$ should be considered as embedded in four-dimensional space. In particular, for any two Hilbert spaces $V,W$ there is a \emph{swap map} $\sigma_{V,W}: V \otimes W \to W \otimes V$. In four dimensions there is no difference between overcrossings and undercrossings, so we simply draw this as an intersection:
\begin{calign}
\includegraphics[scale=1]{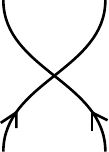} \\
\ket{v} \otimes \ket{w} \mapsto \ket{w} \otimes \ket{v}
\end{calign}
The four-dimensional calculus allows us to untangle arbitrary diagrams and remove any twists, as exemplified by the following equations, which hold regardless of the direction of the arrows on the wires:
\begin{calign}
\includegraphics[scale=1]{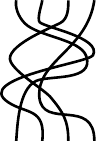}
~~=~~
\includegraphics[scale=1]{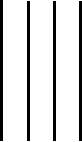}
&
\includegraphics[scale=1]{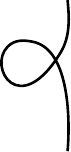}
~~=~~
\includegraphics[scale=1]{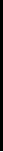}
~~=~~
\includegraphics[scale=1]{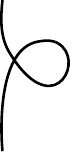}
&
\includegraphics[scale=1]{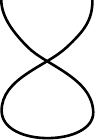}
~~=~~
\includegraphics[scale=1]{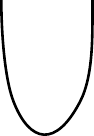}
&
\includegraphics[scale=1]{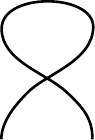}
~~=~~
\includegraphics[scale=1]{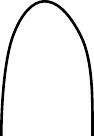}
\end{calign}
It immediately follows that $\Hilb$ is spherical. The trace and dimension of Definition~\ref{def:trace} reduce to the usual notion of trace and dimension of linear maps and Hilbert spaces.

\paragraph{The endomorphism algebra.} The diagrammatic calculus in $\Hilb$ also allows us to conveniently express the endomorphism algebra $B(H)$ of a Hilbert space $H$ using the pivotal dagger structure.
\begin{definition}\label{def:endoalgebra}
Let $H$ be a Hilbert space. We define the following \emph{endomorphism $*$-algebra} on $H \otimes H^*$:
\begin{calign}\label{eq:endoalgebra}
\includegraphics[scale=1]{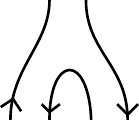}
&
\includegraphics[scale=1]{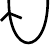}
&
\includegraphics[scale=1]{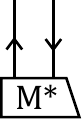}
~~:=~~
\includegraphics[scale=1]{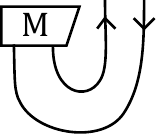}
\\
m: (H \otimes H^*) \otimes (H \otimes H^*) \to H \otimes H^*
&
u: \mathbb{C} \to H \otimes H^*
&
*: H \otimes H^* \to H \otimes H^*
\end{calign}
\end{definition}
\noindent
It is straightforward to check that the endomorphism algebra is indeed a $*$-algebra using the diagrammatic calculus of the pivotal dagger category $\Hilb$. In fact, it is a Frobenius monoid (Definition~\ref{def:Frobenius}).
\begin{proposition}
There is a $*$-isomorphism between the endomorphism algebra $H \otimes H^*$ and the $*$-algebra $B(H)$. 
\end{proposition}
\begin{proof}
Consider the linear bijection $H \otimes H^* \to B(H)$ defined on orthonormal basis elements by $\ket{i} \otimes \bra{j} \mapsto \ket{i}\bra{j}$. It is multiplicative:
\begin{calign}
\includegraphics[scale=1]{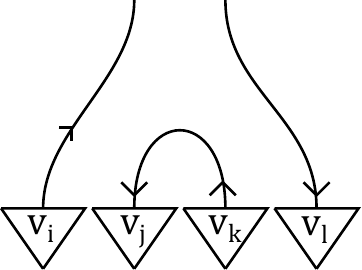}
~~=~~
\includegraphics[scale=1]{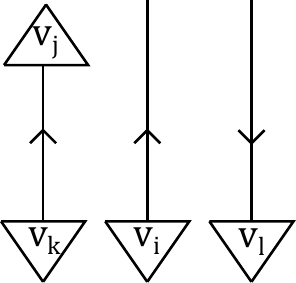}
~~=~~
\delta_{jk}~\includegraphics[scale=1]{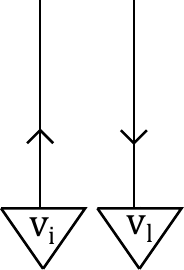}
\end{calign}
It is also unital, since $\sum_i \ket{v_i} \bra{v_i} = \mathbbm{1}$.
Finally, the involution is preserved:
\begin{calign}
\includegraphics[scale=1]{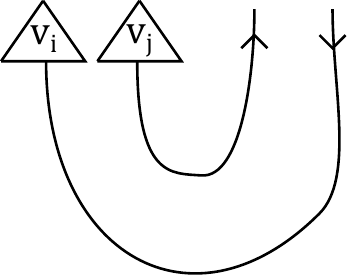}
~~=~~
\includegraphics[scale=1]{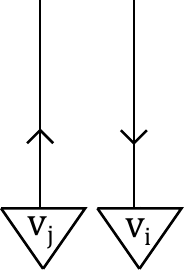}
\end{calign}
\end{proof}
\ignore{
\begin{remark}
In what follows we will occasionally draw diagrams involving infinite-dimensional Hilbert spaces. In this case we will draw arrows only on finite-dimensional Hilbert spaces and their duals. Formally, these diagrams represent morphisms in the category of general (i.e. not just finite-dimensional) Hilbert spaces and linear maps.
\end{remark}
}

\subsection{Monoidal functors}
\subsubsection{Diagrammatic calculus for monoidal functors}
While our monoidal categories are strict, allowing us to use the diagrammatic calculus, we will consider functors between them which are not strict. For this, we use a graphical calculus of \emph{functorial boxes}~\cite{Mellies2006}.
\begin{definition}
Let $\mathcal{C}, \mathcal{D}$, be monoidal categories. A \emph{monoidal functor} $F: \mathcal{C} \to \mathcal{D}$ consists of the following data.
\begin{itemize}
\item A functor $F: \mathcal{C} \to \mathcal{D}$.

In the graphical calculus, we represent the effect of the functor $F$ by drawing a shaded box around objects and morphisms in $\mathcal{C}$. For example, let $X,Y$ be objects and  $f: X \to Y$ a morphism in $\mathcal{C}$. Then the morphism $F(f): F(X) \to F(Y)$ in $\mathcal{D}$ is represented as:
\begin{calign}\nonumber
\includegraphics[scale=0.8]{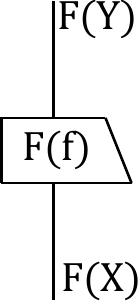}
~~=~~
\includegraphics[scale=0.8]{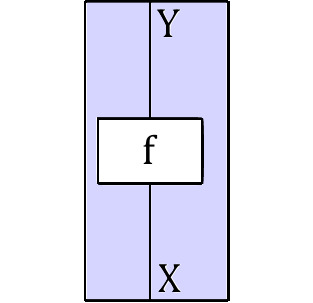}
\end{calign}
\ignore{
From this point forward, we will not write in the labels of the 0-cell regions unless absolutely necessary. This is for two reasons: firstly, to leave the notation uncluttered; and secondly, so that the reader uncomfortable with 2-categories can simply assume that the diagrams are monoidal categories.}

\item For every pair of objects $X,Y$ of $\mathcal{C}$, an invertible \emph{multiplicator} morphism $m_{X,Y}:F(X) \otimes_D F(Y) \to F(X \otimes_C Y)$. In the graphical calculus, these morphisms and their inverses are represented as follows:
\begin{calign}\nonumber
\includegraphics[scale=.8]{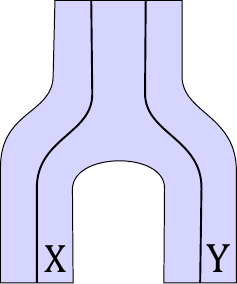}
&
\includegraphics[scale=.7]{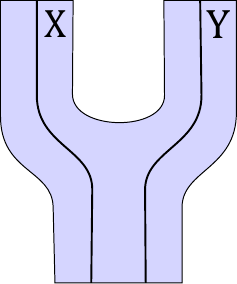}
\\\label{eq:multiplicator}
m_{X,Y}: F(X) \otimes_{\mathcal{D}}  F(Y) \to F(X \otimes_{\mathcal{C}} Y) & m_{X,Y}^{-1}: F(X \otimes_{\mathcal{C}} Y) \to F(X) \otimes_{\mathcal{D}}  F(Y)
\end{calign}
\item An invertible `unitor' morphism $u: \mathbbm{1}_D \to F(\mathbbm{1}_C)$. In the diagrammatic calculus, this morphism and its inverse are represented as follows (recall that the monoidal unit is invisible in the diagrammatic calculus):
\begin{calign}\nonumber
\includegraphics[scale=.8]{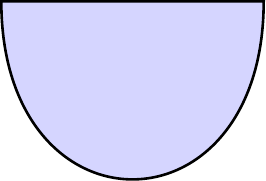}
&
\includegraphics[scale=.8]{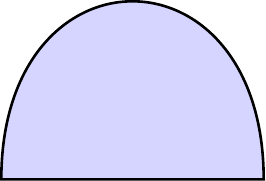} \\\label{eq:unitor}
u: \mathbbm{1}_D \to F(\mathbbm{1}_C) & u^{-1}: F(\mathbbm{1}_C)  \to \mathbbm{1}_D
\end{calign}
\end{itemize}
The multiplicators and unitor obey the following coherence equations:
\begin{itemize}
\item \emph{Naturality}. For any objects $X,X',Y,Y'$ and morphisms $f:X \to X', g: Y \to Y'$ in $\mathcal{C}$:
\begin{calign}\label{eq:psfctnat}
\includegraphics[scale=0.8]{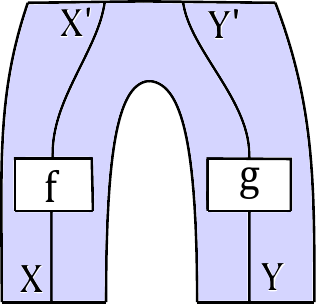}
~~=~~
\includegraphics[scale=0.8]{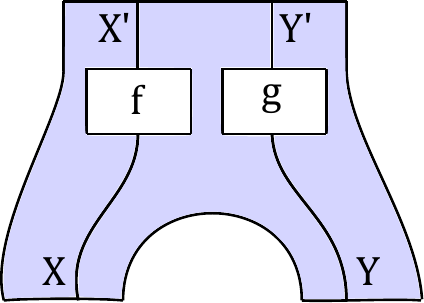}
\end{calign}
\item \emph{Associativity}. For any objects $X,Y,Z$ of $\mathcal{C}$:
\begin{calign}\label{eq:psfctassoc}
\includegraphics[scale=0.8]{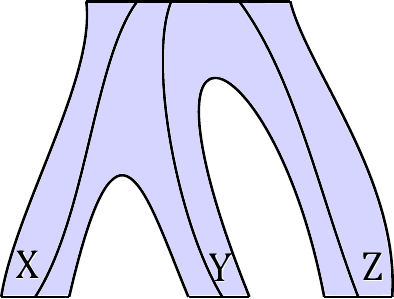}
~~=~~
\includegraphics[scale=0.8]{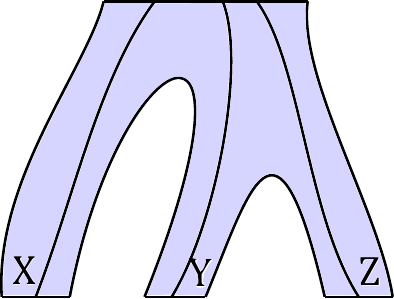}
\end{calign}
\item \emph{Unitality}. For any object $X$ of $\mathcal{C}$:
\begin{calign}\label{eq:psfctunital}
\includegraphics[scale=0.8]{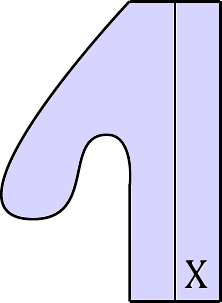}
~~=~~
\includegraphics[scale=0.8]{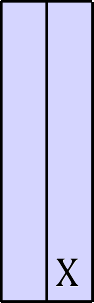}
~~=~~
\includegraphics[scale=0.8]{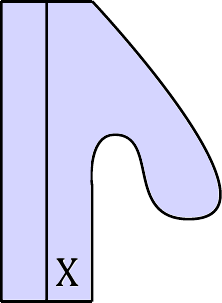}
\end{calign}
\end{itemize}
For monoidal dagger categories $\mathcal{C},\mathcal{D}$, we say that a monoidal functor $F: \mathcal{C} \to \mathcal{D}$ is \emph{unitary} if $F(f^{\dagger}) = F(f)^{\dagger}$ and the multiplicators and unitor are unitary morphisms. 
\ignore{
We say that a monoidal functor $F: \mathcal{C} \to \mathcal{D}$ is an \emph{equivalence} if every object in $\mathcal{D}$ is equivalent to an object in the image of $F$ and the functors $F_{r,s}: \mathcal{C}(r,s) \to \mathcal{D}(r,s)$ are equivalences.
}
\end{definition}
\noindent
We observe that the analogous \emph{conaturality}, \emph{coassociativity} and \emph{counitality} equations for the inverses $\{m_{X,Y}^{-1}\},u^{-1}$, obtained by reflecting~(\ref{eq:psfctnat}-\ref{eq:psfctunital}) in a horizontal axis, are already implied by~(\ref{eq:psfctnat}-\ref{eq:psfctunital}). To give some idea of the calculus of functorial boxes, we explicitly prove the following lemma and proposition. From now on we will unclutter the diagrams by omitting object labels, unless adding the labels seems to significantly aid comprehension.
\begin{lemma}\label{lem:pushpast}
For any objects $X,Y,Z$ of $\mathcal{C}$, the following equations are satisfied:
\begin{calign}\nonumber
\includegraphics[scale=1]{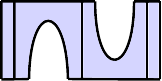}
~~=~~
\includegraphics[scale=1]{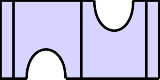}
&
\includegraphics[scale=1]{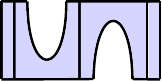}
~~=~~
\includegraphics[scale=1]{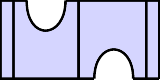}
\end{calign}
\end{lemma}
\begin{proof}
We prove the left equation; the right equation is proved similarly. 
\begin{calign}\nonumber
\includegraphics[scale=1]{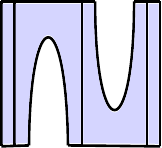}
~~=~~
\includegraphics[scale=1]{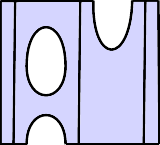}
~~=~~
\includegraphics[scale=1]{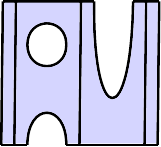}
~~=~~
\includegraphics[scale=1]{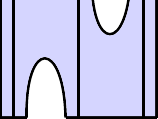}
\end{calign}
Here the first and third equalities are by invertibility of $m_{X,Y}$, and the second is by coassociativity.
\end{proof}
\noindent
With Lemma~\ref{lem:pushpast}, the equations~(\ref{eq:psfctnat}-\ref{eq:psfctunital}) are sufficient to deform functorial boxes topologically as required.  From now on we will do this mostly without comment.

\subsubsection{Induced duals}

We first observe that the duals in $\mathcal{C}$ induce duals in $\mathcal{D}$ under a monoidal functor $F: \mathcal{C} \to \mathcal{D}$. 
\begin{proposition}[Induced duals]\label{prop:indduals}
Let $X$ be an object in $\mathcal{C}$ and $[X^*,\eta,\epsilon]$ a right dual. Then $F(X^*)$ is a right dual of $F(X)$ in $\mathcal{D}$ with the following cup and cap:
\begin{calign}\nonumber
\includegraphics[scale=.6]{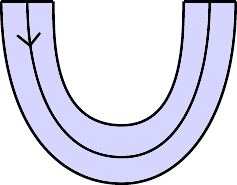}
&
\includegraphics[scale=.6]{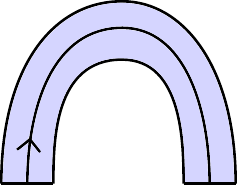}
\\\nonumber
F(\eta) & F(\epsilon)
\end{calign}
The analogous statement holds for left duals.
\end{proposition}
\begin{proof}
We show one of the snake equations~\eqref{eq:rightsnakes} in the case of right duals; the others are all proved similarly.
\begin{calign}\nonumber
\includegraphics[scale=1]{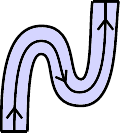}
~~=~~
\includegraphics[scale=1]{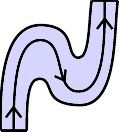}
~~=~~
\includegraphics[scale=1]{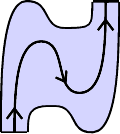}
~~=~~
\includegraphics[scale=1]{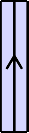}
\end{calign}
Here the first equality is by Lemma~\ref{lem:pushpast}, the second by~\eqref{eq:psfctnat} and the third by~\eqref{eq:psfctunital} and the right snake equation in $\mathcal{C}$.
\end{proof}
\noindent
For any 1-morphism $X$ of $\mathcal{C}$, then, we have two sets of left and right duals on $F(X)$; the first from the pivotal structure in $\mathcal{C}$ by Proposition~\ref{prop:indduals}, and the second from the pivotal structure in $\mathcal{D}$.

In the diagrammatic calculus we distinguish between these two duals by drawing a large downwards arrowhead on the dual in $\mathcal{D}$, like so:
\begin{calign}\nonumber
\includegraphics[scale=1]{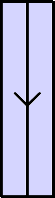}
&
\includegraphics[scale=1]{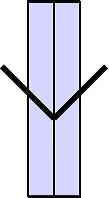}
\\\label{eq:dualscompare}
\textrm{Induced duals $F(X^*)$ from $\mathcal{C}$} & \textrm{Duals $F(X)^*$ in $\mathcal{D}$}
\end{calign}

\subsection{Compact quantum groups }

In this work we will restrict our attention to the specific case of UPTs between unitary $\mathbb{C}$-linear monoidal functors from $C^*$-tensor categories with conjugates $\mathcal{C}$ into the category $\Hilb$ of finite-dimensional Hilbert spaces and linear maps. Provided that such a functor exists, there is a duality theory which identifies $\mathcal{C}$ as the category of corepresentations of a certain algebraic object. 

\subsubsection{$C^*$-tensor categories}
We first recall the definition of a $C^*$-tensor category with conjugates.
\begin{definition}
A dagger category is \emph{$\mathbb{C}$-linear} if: \begin{itemize}
\item For every pair of objects $X,Y$, $\Hom(X,Y)$ is a complex vector space.
\item For every triple of objects $X,Y,Z$, composition $\circ: \Hom(X,Y) \times \Hom(Y,Z) \to \Hom(X,Z)$ is a bilinear map.
\item For every pair of objects $X,Y$, the dagger $\dagger: \Hom(X,Y) \to \Hom(Y,X)$ is an antilinear and positive map, i.e. $\alpha \circ \alpha^{\dagger} = 0$ iff $\alpha=0$.
\end{itemize}
A monoidal dagger category is $\mathbb{C}$-linear if, additionally:
\begin{itemize}
\item For every quadruple of objects $X,X',Y,Y'$, the tensor product $\otimes: \Hom(X,X') \times \Hom(Y,Y') \to \Hom(X \otimes X',Y \otimes Y')$ is a bilinear map.
\end{itemize}
A functor $F: \mathcal{C} \to \mathcal{D}$ between $\mathbb{C}$-linear categories is called \emph{$\mathbb{C}$-linear} if the induced maps on Hom-spaces $F_{X,Y}: \Hom_{\mathcal{C}}(X,Y) \to \Hom_{\mathcal{D}}(F(X),F(Y))$ are $\mathbb{C}$-linear.
\end{definition}
\begin{definition}\label{def:c*tensorcats}
A $\mathbb{C}$-linear monoidal dagger category is called a \emph{$C^*$-tensor category} if 
\begin{itemize}
\item $\Hom(X,Y)$ is a Banach space for all objects $X,Y$, and $\left\Vert fg \right\Vert \leq \left\Vert f \right\Vert \left\Vert g \right\Vert$.
\item The dagger satisfies the following properties for any morphism $f: X \to Y$:
\begin{itemize}
\item $\left\Vert f^{\dagger} \circ f \right\Vert = \left\Vert f \right\Vert^2$; in particular, $\End(X)$ is a $C^*$-algebra for any object $X$.
\item $f^{\dagger} \circ f$ is positive as an element of the $C^*$-algebra $\End(X)$.
\end{itemize}
\end{itemize}
Following~\cite{Neshveyev2013} we also assume that our $C^*$-tensor categories have the following completeness properties:
\begin{itemize}
\item There is an object $\bf{0}$ such that $\dim(\Hom({\bf 0},X))=0$ for every object $X$.
\item There are binary direct sums: for every pair of objects $X_1,X_2$, there is an object $X_1 \oplus X_2$ and morphisms $s_i: X_i \to X_1 \oplus X_2$ (for $i\in \{1,2\}$) such that
\begin{calign}
s_i^{\dagger} s_i = \id_{X_i}
&
s_1 s_1^{\dagger} + s_2 s_2^{\dagger} = \id_{X_1 \oplus X_2} 
\end{calign}
\item Dagger idempotents split: for any morphism $f: X \to X$ such that $f^{\dagger} = f \circ f = f$, there exists an object $Y$ and an isometry $\iota: Y \to X$ such that $\iota \circ \iota^{\dagger} = f$. 
\item The unit object $\mathbbm{1}$ is \emph{irreducible}, i.e. $\End(\mathbbm{1}) = \mathbb{C} \id_{\mathbbm{1}}$.
\end{itemize}
\end{definition}
\noindent
In the setting of $C^*$-tensor categories, one normally speaks of \emph{conjugates}\footnote{Note that these are not the conjugates of~\eqref{eq:conjugate}; they are rather dual objects.} rather than duals.
\begin{definition}\label{def:conjeqns}
Let $X$ be an object of a $C^*$-tensor category $\mathcal{C}$. We say that morphisms $R: \mathbbm{1} \to X^* \otimes X$ and $\bar{R}: \mathbbm{1} \to X \otimes X^{*}$ are \emph{solutions to the conjugate equations} if $[X^*,R,\bar{R}^{\dagger}]$ is right dual to $X$.
\end{definition}
\noindent
We observe that since $[X^*,R,\bar{R}^{\dagger}]$ is right dual to $X$, then $[X^*, \bar{R}, R^{\dagger}]$ is left dual to $X$. We therefore have left and right duals for $X$ satisfying the equations~\eqref{eq:daggerdual}. Suppose that solutions to the conjugate equations are chosen for every object in $\mathcal{C}$. For $\mathcal{C}$ to be a pivotal dagger category with right and left duals $[X^*,R_X,\bar{R}_X^{\dagger}]$ and $[X^*,\bar{R}_X,R_X^{\dagger}]$, by~\eqref{eq:pivldualdef} we require a monoidal natural isomorphism $\iota$ from the  double duals functor to the identity functor relating the left and right duals, i.e.: 
\begin{align}\label{eq:conjugatepivdag}
\bar{R}_X = (\iota_X \otimes \id_X^*) R_{X^*} && R_X^{\dagger} = \bar{R}_{X^*}^{\dagger} (\id_{X^*} \otimes \iota_{X}^{-1}) 
\end{align}
In any $C^*$-tensor category with conjugates each object possesses a distinguished conjugate $[X^*,R,\bar{R}]$, unique up to unitary isomorphism, called a \emph{standard solution}~\cite[Def 2.2.14]{Neshveyev2013}. 
\begin{theorem} \label{thm:c*tenspivdag}
Let $\mathcal{C}$ be a $C^*$-tensor category with conjugates, and let $\{[X^*,R_X,\overline{R}_X]\}$ be a choice of standard solutions for all objects $X$ in $\mathcal{C}$. Then there exists a unitary monoidal natural isomorphism $\iota$ from the double duals functor to the identity functor, such that the equations~\eqref{eq:conjugatepivdag} are obeyed. 
\end{theorem}
\begin{proof}
It is shown in~\cite[Thm. 2.2.21]{Neshveyev2013} that such a unitary natural isomorphism $\iota$ exists. What remains is to demonstrate that $\iota$ is monoidal. We provide a proof of this fact in the Appendix (Section~\ref{sec:app}).
\end{proof}
\noindent
A $C^*$-tensor category equipped with standard solutions to the conjugate equations is therefore a pivotal dagger category (in fact, a spherical dagger category~\cite[Thm 2.2.16]{Neshveyev2013}) and can be treated using the graphical calculus just discussed.
\begin{definition}
Let $\mathcal{C}$ be a $C^*$-tensor category with conjugates. We call a unitary $\mathbb{C}$-linear monoidal functor $F: \mathcal{C} \to \Hilb$ a \emph{fibre functor}.
\end{definition}
\ignore{
\begin{remark}
Semisimple $\mathbb{C}$-linear pivotal dagger categories $\mathcal{C}$ are in particular \emph{$W^*$-tensor categories with conjugates}. Indeed, the semisimplicity assumptions of Definition~\eqref{} imply that $\mathcal{C}$ is a $W^*$-tensor category~\cite[P.523]{Mueger2004}, and the equations~\eqref{} further imply that $\mathcal{C}$ \emph{has conjugates} in the sense of~\cite{Neshveyev2013}. In the other direction, the standard duality in any $W^*$-tensor category with conjugates has an associated spherical pivotal structure relating the left and right duals as in~\eqref{}; this is shown in~\cite{}. 
\end{remark}
}

\subsubsection{Compact quantum groups}

We now introduce the algebraic objects dual to $C^*$-tensor categories with conjugates and a chosen fibre functor. All algebras are taken over $\mathbb{C}$.
\begin{definition}[{\cite[Definition 1.6.1]{Neshveyev2013}}]
A unital $*$-algebra $A$ equipped with a unital $*$-homomorphism $\Delta: A \to A \otimes A$ (the \emph{comultiplication}) is called a \emph{Hopf-$*$-algebra} if $(\Delta \otimes \id_A) \circ \Delta = (\id_A \otimes \Delta) \circ \Delta$ and there exist linear maps $\epsilon: A \to \mathbb{C}$ (the \emph{counit}) and $S: A \to A$ (the \emph{antipode}) such that 
\begin{calign}\nonumber
\includegraphics[scale=1]{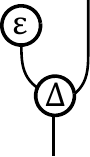}
~~=~~
\includegraphics[scale=1]{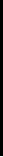}
~~=~~
\includegraphics[scale=1]{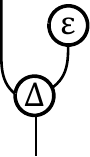}
&&
\includegraphics[scale=1]{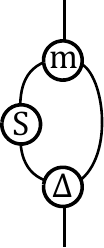}
~~=~~
\includegraphics[scale=1]{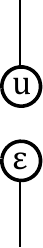}
~~=~~
\includegraphics[scale=1]{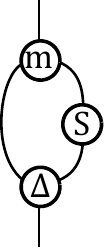}
\\
(\epsilon \otimes \id_A) \circ \Delta = \id_A = (\id_A \otimes \epsilon) \circ \Delta 
&&
m \circ (S \otimes \id_A) \circ \Delta = u \circ \epsilon = m \circ (\id_A \otimes S) \circ \Delta
\end{calign}
where $m: A \otimes A \to A$ is the multiplication and $u: \mathbb{C} \to A$ the unit of the algebra $A$.
\end{definition}
\ignore{
\begin{example}
Let $G$ be a compact group, and let $C(G)$ be the $\mathbb{C}^*$-algebra of continuous complex functions on $G$. For any finite-dimensional $*$-representation $\pi: G \to B(H)$, each matrix entry $\pi_{ij}: G \to \mathbbm{C}$ is an element of $C(G)$. Let $\mathbb{C}[G] \subset C(G)$ be the linear span of matrix coefficients of finite-dimensional representations of $G$. This naturally possesses the structure of a Hopf $*$-algebra (see~\cite{}).
\end{example}
}
\begin{definition}[{\cite[Definition 1.6.5]{Neshveyev2013}}]
A \emph{corepresentation} of a Hopf $*$-algebra $A$ on a vector space $H$ is a linear map $\delta: H \to H \otimes A$ such that 
\begin{calign}
(\delta \otimes \id_A) \circ \delta = (\id_H \otimes \Delta) \circ \delta 
&&
(\id_H \otimes \epsilon) \circ  \delta = \id_H
\end{calign}
The corepresentation is called \emph{unitary} if 
$H$ is a Hilbert space and 
$$
\left\langle \delta(\xi), \delta(\zeta) \right\rangle = (\xi,\zeta) 1_A ~~~~~\textrm{for all } \xi,\zeta \in H
$$
where the $A$-valued inner product $\left\langle \cdot, \cdot \right\rangle$ on $H \otimes A$ is defined by linear extension of $\left\langle \xi \otimes a, \zeta \otimes b \right\rangle = (\xi,\zeta) a^*b$.

For $(H_1,\delta_1), (H_2,\delta_2)$ corepresentations of a Hopf-$*$-algebra $A$, we say that a linear map $f: H_1 \to H_2$ is an \emph{intertwiner} $f:(H_1,\delta_1) \to (H_2,\delta_2)$ if $\delta_2 \circ f = (f \otimes \id_A) \circ \delta_1$.
\end{definition}
\begin{definition}
Let $(H,\delta)$ be a finite-dimensional unitary  corepresentation of $A$, and let $\{\ket{v_i}\}$ be an orthonormal basis of $H$. Then 
$(\bra{v_j} \otimes \id_A) \circ \delta (\ket{v_i})$ defines an $A$-valued matrix $U_{ij}$; we say that the entries of this matrix are the \emph{matrix coefficients} of the representation in the basis $\{\ket{v_i}\}$.
\end{definition}
\noindent
\begin{definition}[{c.f. \cite[Theorem 1.6.7]{Neshveyev2013}}]
We say that a Hopf-$*$-algebra is a \emph{compact quantum group algebra (CQG algebra)} if it is generated as an algebra by matrix coefficients of its finite-dimensional unitary corepresentations. 
\end{definition}
\noindent
For a compact quantum group algebra, a monoidal product of corepresentations can be defined~\cite[Definition 1.3.2]{Neshveyev2013}, as can a notion of conjugate corepresentation~\cite[Def. 1.4.5]{Neshveyev2013}. This yields a $C^*$-tensor category with conjugates $\Corep(A)$ whose objects are finite-dimensional unitary corepresentations of the algebra $(A,\Delta)$ and whose morphisms are intertwiners, with an obvious canonical fibre functor $F: \Corep(A) \to \Hilb$ which forgets the representation. Taking standard solutions to the conjugate equations, $\Corep(A)$ has the structure of a pivotal dagger category. 

Such an algebra is considered as the algebra of matrix coefficents of representations of some `compact quantum group' $G$, such that $\Rep(G) = \Corep(A)$. We will refer to compact quantum groups $G$, and write $\Rep(G)$, in order to emphasise the similarity with representation theory of compact groups. However, the algebra $A_G$ is the concrete object in general.

We now recall the theorem relating $C^*$-tensor categories with conjugates to compact quantum groups. 
\begin{theorem}[{\cite[Theorem 2.3.2]{Neshveyev2013}}]\label{thm:tannakaduality}
Let $\mathcal{C}$ be a $C^*$-tensor category with conjugates, and let $U: \mathcal{C} \to \Hilb$ be a  fibre functor. Then there exists a compact quantum group algebra $A$ (uniquely determined up to isomorphism) and a unitary monoidal equivalence $E_U: \mathcal{C} \to \Rep(G_A)$, such that $U$ is unitarily monoidally naturally isomorphic to $F \circ E_U$.
\end{theorem}
\ignore{
\noindent
The compact quantum group has as its underlying  vector space the predual of the space of natural endomorphisms of the functor $U$. Its algebraic structure naturally arises from the categorical structure of $\mathcal{C}$, $\Hilb$ and $U$ (see e.g.~\cite{}). 
}

\subsection{Unitary pseudonatural transformations}
A pseudonatural transformation between pseudofunctors is a 2-categorical generalisation of a natural isomorphism between functors~\cite{Leinster1998}. Monoidal functors are pseudofunctors, where we consider a monoidal category as a 2-category with a single object. Pseudonatural transformations between monoidal functors can be seen as a generalisation of monoidal natural isomorphisms.

\subsubsection{Definition}
\ignore{
In what follows we will deal mostly with UPTs between monoidal functors. We therefore review the definition of a UPT in this special case.
\begin{remark}
A one-object 2-category is usually referred to as a \emph{monoidal category}. The 1-morphisms are referred to as \emph{objects} and the 2-morphisms as \emph{morphisms}, while the composition of 1-morphisms and horizontal composition of 2-morphisms are referred to as the \emph{monoidal product} and denoted by $\otimes$. The identity 1-morphism is referred to as the \emph{unit object}. Pseudofunctors between such categories are referred to as \emph{monoidal functors}. We will adopt this language when discussing monoidal categories.

In what follows, when we speak of e.g. a pivotal dagger \emph{category}, we mean a one-object pivotal dagger 2-category.  
\end{remark}
\noindent
}
If the categories $\mathcal{C}$ and $\mathcal{D}$ are pivotal dagger categories, there is a notion of unitarity for pseudonatural transformations between monoidal functors $\mathcal{C} \to \mathcal{D}$ which generalises unitary monoidal natural isomorphism~\cite{Verdon2020a}. We here deal only with unitary pseudonatural transformations.
\begin{definition}\label{def:pntmon}
Let $\mathcal{C},\mathcal{D}$ be pivotal dagger  categories, and let $F_1,F_2: \mathcal{C} \to \mathcal{D}$ be unitary monoidal functors. (We colour the functorial boxes blue and red respectively.) A \emph{unitary pseudonatural transformation} $(\alpha,H): F_1 \to F_2$ is defined by the following data:
\begin{itemize}
\item An object $H$ of $\mathcal{D}$ (drawn as a green wire).
\item For every object $X$ of $\mathcal{C}$, a unitary morphism $\alpha_X: F_1(X) \otimes H \to H \otimes F_2(X)$ (drawn as a white vertex):
\begin{calign}
\includegraphics[scale=1]{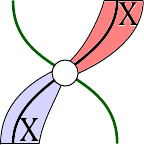}
\end{calign}
\end{itemize}
The unitary morphisms $\alpha_X$ must satisfy the following conditions:
\begin{itemize}
\item \emph{Naturality.} For every morphism $f:X \to Y$ in $\mathcal{C}$:
\begin{calign}\label{eq:pntmonnat}
\includegraphics[scale=0.7]{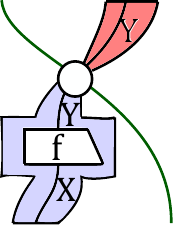}
~~=~~
\includegraphics[scale=0.7]{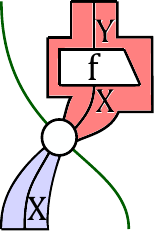}
\end{calign}
\item \emph{Monoidality.} 
\begin{itemize}
\item For every pair of objects $X,Y$ of $\mathcal{C}$:
\begin{calign}\label{eq:pntmonmon}
\includegraphics[scale=0.7]{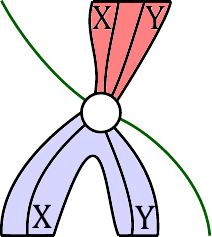}
~~=~~
\includegraphics[scale=0.7]{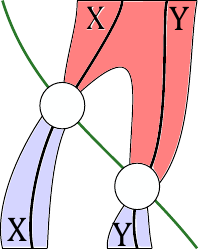}
\end{calign}
\item $\alpha_{\mathbbm{1}}$ is defined as follows:
\begin{calign}\label{eq:pntmonmonunit}
\includegraphics[scale=0.6]{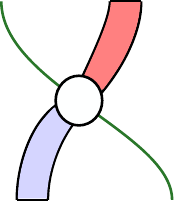}
~~=~~
\includegraphics[scale=0.7]{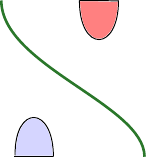}
\end{calign}
\end{itemize}
\end{itemize}
\end{definition}
\begin{remark}
Unitary pseudonatural transformations generalise the notion of \emph{unitary monoidal natural isomorphism}, which we recover when $H \cong \mathbbm{1}$.
\end{remark}
\begin{remark}
The diagrammatic calculus shows that pseudonatural transformation is a planar notion. The $H$-wire forms a boundary between two regions of the $\mathcal{D}$-plane, one in the image of $F$ and the other in the image of $G$. By pulling through the $H$-wire, morphisms from $\mathcal{C}$ can move between the two regions~\eqref{eq:pntmonnat}. 
\end{remark}
\noindent
UPTs $(\alpha,H): F_1 \to F_2$ and $(\beta,H'): F_2 \to F_3$ can be composed associatively to obtain a UPT $(\alpha \otimes \beta,H \otimes H'): F_1 \to F_3$ whose components $(\alpha \otimes \beta)_X$ are as follows (we colour the $H'$-wire orange, and the $F_3$-box brown):
\begin{calign}
\includegraphics[scale=1]{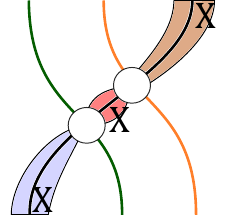}
\end{calign}
There are also morphisms between pseudonatural transformations, known as \emph{modifications}~\cite{Leinster1998}.
\begin{definition}
Let $(\alpha,H), (\beta,H'): F_1 \Rightarrow F_2$ be UPTs. (We colour the $H$-wire green and the $H'$-wire orange.) A \emph{modification} $f: \alpha \to \beta$  is a morphism $f: H \to H'$ satisfying the following equation for all components $\{\alpha_X,\beta_X\}$:
\begin{calign}\label{eq:uptmod}
\includegraphics[scale=1]{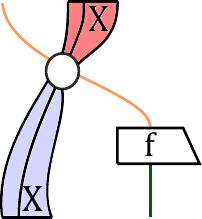}
~~=~~
\includegraphics[scale=1]{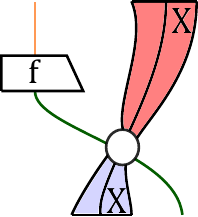}
\end{calign}
\end{definition}
\noindent
Modifications can themselves be composed horizontally and vertically --- vertical composition is composition of morphisms in $\mathcal{D}$, while horizontal composition is monoidal product of morphisms in $\mathcal{D}$. The dagger of a modification is also a modification. Altogether, the compositional structure is that of a dagger 2-category, i.e. a 2-category with a dagger on 2-cells which is compatible with horizontal composition. (For an explicit definition, see~\cite[Def. 2.5]{Heunen2016} or~\cite[Def. 2.13]{Verdon2020a}.)
\begin{definition}
Let $\mathcal{C},\mathcal{D}$ be pivotal dagger categories. The dagger 2-category $\Fun(\mathcal{C},\mathcal{D})$ is defined as follows:
\begin{itemize}
\item Objects: unitary monoidal functors $F_1,F_2,\dots,\cdot: \mathcal{C} \to \mathcal{D}$.
\item 1-morphisms: unitary pseudonatural transformations $\alpha,\beta,\dots: F_1 \to F_2$.
\item 2-morphisms: modifications $f, g,\dots: \alpha \to \beta$.
\end{itemize}
\end{definition}
Because we are able to assume that $\mathcal{C}$ and $\mathcal{D}$ are strict, $\Fun(\mathcal{C},\mathcal{D})$ is a strict 2-category.
\begin{remark}
Because $\Fun(\mathcal{C},\mathcal{D})$ is a dagger 2-category, the endomorphism categories $\End(F)$ of UPTs $F \to F$ and modifications for a given functor $F$ are monoidal dagger categories. 
\end{remark}

\subsubsection{Duals}
\begin{definition}
Let $(\alpha,H): F_1 \to F_2$ be a UPT. Then the \emph{dual} of $\alpha$ is a UPT $(\alpha^*,H^*): F_2 \to F_1$ whose components $\alpha^*_X$ are defined as follows:
\begin{calign}\label{eq:dualpnt}
\includegraphics{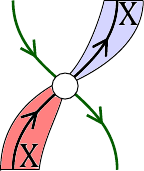}
:=
\includegraphics{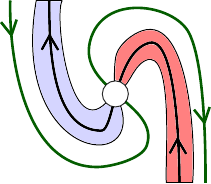}
=
\includegraphics[scale=0.7]{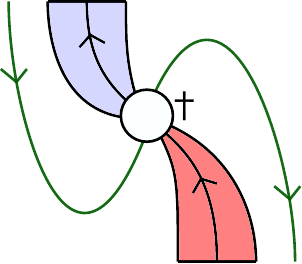}
\end{calign}
Here the second equality is by unitarity of $\alpha$~\cite[Prop. 5.2]{Verdon2020a}. We sometimes put a $*$ next to the vertex for the dual UPT $\alpha^*$ to distinguish it from $\alpha$, although the orientation of the $H$-wire is sufficient for this. 
\end{definition}
\noindent
For the composition of a UPT with its dual, the cups and caps of the dagger duality in $\mathcal{D}$ are modifications~\cite[Thm. 4.4, Cor. 5.6]{Verdon2020a} :
\begin{calign}\label{eq:cupcapmodsdualpnt}
\includegraphics{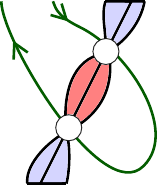}
~~=~~
\includegraphics{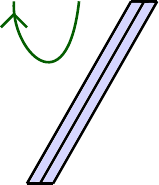}
&
\includegraphics{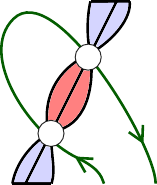}
~~=~~
\includegraphics{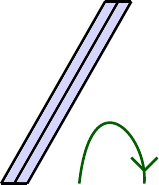}
\\
\includegraphics{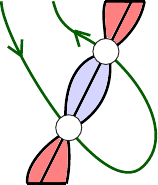}
~~=~~
\includegraphics{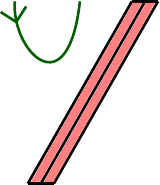}
&
\includegraphics{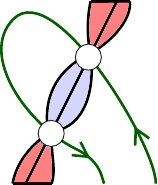}
~~=~~
\includegraphics{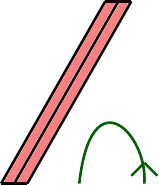}
\end{calign}

\section{UPTs and Hopf-Galois theory}\label{sec:hopfgal}
The Tannaka duality relating $C^*$-tensor categories to compact quantum group algebras extends to the characterisation of their fibre functors. For a compact quantum group $G$, the algebraic objects corresponding to fibre functors on $\Rep(G)$ are \emph{Hopf-Galois objects} for the CQG algebra $A_G$, also known as \emph{noncommutative torsors}. In this section we will show that UPTs between fibre functors can be classified in terms of the finite-dimensional representation theory of these Hopf-Galois objects.

\subsection{Background on Hopf-Galois theory}\label{sec:hopfgalbackground}
We now review the basics of Hopf-Galois theory for compact quantum groups, based on~\cite{Bichon1999}.
\begin{definition}
Let $A$ be a CQG algebra, and $Z$ a $*$-algebra. We say that a left corepresentation $\delta: Z \to A \otimes Z$ is a \emph{left coaction} if $\delta$ is a $*$-homomorphism. \ignore{We call $Z$ a \emph{right $A$-comodule algebra}.} \ignore{Left coactions \ignore{and comodule algebras} are defined similarly.
}
\end{definition}
\ignore{
The following definition starts from the observation that a space $X$ is a torsor for a compact group if and only if the following map is a bijection:
\begin{calign}
G \times X \to X \times X \\
(g,x) \mapsto (g \cdot x,x)
\end{calign}}
\begin{definition}{\cite[Def. 4.1.1]{Bichon1999}}
Let $A$ be a CQG algebra, and let $Z$ be an $*$-algebra with a left $A$-coaction. We say that $Z$ is a \emph{left $A$-Hopf-Galois object} if the following linear map is a bijection: 
$$
( \id_A \otimes m_Z ) \circ ( \delta \otimes \id_Z) : Z \otimes Z \to A \otimes Z
$$ 
Right $A$-Hopf-Galois objects may be defined similarly. For two CQG algebras $A_1,A_2$, if $Z$ is a left $A_1$- and a right $A_2$-Hopf Galois object, we say that it is an \emph{$A_1$-$A_2$-bi-Hopf-Galois object}.
\end{definition}
\noindent
From any left Hopf-Galois object for $A_G$ it is possible to construct a unitary fibre functor on $\Rep(G)$.  
\begin{construction}[{\cite[Prop. 4.3.1]{Bichon1999}}]\label{constr:galtofib}
Let $G$ be a compact quantum group, and $(Z,\delta_Z)$ a left Hopf-Galois object for $A_G$. Then a unitary fibre functor $F_Z: \Rep(G) \to \Hilb$ may be constructed, part of whose definition is as follows:
\begin{itemize}
\item \emph{On objects.} For any corepresentation $(V,\delta_V)$ of $A_G$, as a vector space $F(V)= V \wedge Z$, where $V \wedge Z$ is the equaliser of the double arrow:
$$
\delta_V \otimes \id_Z, \id_V \otimes \delta_Z: V \otimes Z \rightrightarrows V \otimes A_G \otimes Z
$$ 
\item \emph{On morphisms.} For any intertwiner $f: V \to W$, $F(f) = f \wedge \id_Z: V \wedge Z \to W \wedge Z$, where $f \wedge \id_Z$ is the map on equalisers induced by $f \otimes \id_Z$.
\end{itemize}
\end{construction}
\begin{remark}
The Hopf-Galois object corresponding to the canonical fibre functor is $A_G$ itself.
\end{remark}
\noindent
Likewise, from any unitary fibre functor on $\Rep(G)$ one can construct an left $A_G$-Hopf-Galois object.
\begin{construction}[{See~\cite[Prop. 4.3.3]{Bichon1999} for part of the proof}]\label{constr:fibtogal}
Let $\mathcal{C}$ be a rigid $C^*$-tensor category with conjugates, let $F_1,F_2: \mathcal{C} \to \Hilb$ be two unitary fibre functors, and let $G_1,G_2$ be the two compact quantum groups obtained by Tannaka reconstruction (Theorem~\ref{thm:tannakaduality}). The predual $\Hom^{\vee}(F_1,F_2)$ of the vector space $\Hom(F_1,F_2)$ of natural transformations $F_1 \to F_2$ has the structure of an $A_{G_2}$-$A_{G_1}$-bi-Hopf-Galois object. 
\end{construction}
\noindent
These constructions lead to a classification of unitary fibre functors on $\Rep(G)$. Let $\Gal(A)$ be the category whose objects are left $A$-Hopf-Galois objects and whose morphisms are $*$-homomorphisms intertwining the $A$-coactions. Let $\Fib(G)$ be the category whose objects are unitary fibre functors on $\Rep(G)$ and whose morphisms are unitary monoidal natural isomorphisms.
\begin{theorem}[{\cite[Thm 4.3.4]{Bichon1999}}]
Constructions~\ref{constr:galtofib}  and~\ref{constr:fibtogal} yield an equivalence of categories $\Gal(A) \simeq \Fib(G)$.
\end{theorem}
\noindent
We also note the following fact characterising the spectrum of a bi-Hopf-Galois object, which Theorem~\ref{thm:higherhopfgal} will generalise to all finite-dimensional $*$-representations.
\begin{proposition}[{For part of the proof see \cite[Thm 4.4.1]{Bichon1999}}]
Let $\mathcal{C}$ be a $C^*$-tensor category with conjugates, let $F_1,F_2: \mathcal{C} \to \Hilb$ be fibre functors, and let $Z$ be the corresponding $A_{G_2}$-$A_{G_1}$-bi-Hopf-Galois object. Then there is a bijection between the set of unitary monoidal natural isomorphisms $F_1 \to F_2$ and the set of 1-dimensional $*$-representations of $Z$.
\end{proposition}

\paragraph{A generators-and-relations description of $Z$.}
We will need the following generators-and-relations description of the bi-Hopf-Galois object $Z=\Hom^{\vee}(F_1,F_2)$ linking two fibre functors $F_1,F_2: \mathcal{C} \to \Hilb$, taken from~\cite{Bichon1999,Joyal1991}. We assume that $\mathcal{C}$ is a small category.

Consider the vector space $\bigoplus_{V \in \textrm{Obj}(\mathcal{C})} \Hom(F_2(V),F_1(V))$, where the sum is taken over all objects of $\mathcal{C}$. Let $\mathcal{N}$ be the following subspace:
\begin{equation}\label{eq:Zquotientsubspace}
\left\langle F_1(f) \circ v - v \circ F_2(f)~~~|~~~ \forall~W \in\textrm{Obj}(\mathcal{C}), ~\forall~f \in \Hom_{\mathcal{C}}(V,W), ~\forall~v \in \Hom(F_2(W),F_1(V)) \right\rangle
\end{equation}
Then, as a vector space:
\begin{equation}
\Hom^{\vee}(F_1,F_2) := \bigoplus_{V \in \Obj(\mathcal{C})} \Hom(F_2(V),F_1(V)) / \mathcal{N}
\end{equation}
We denote $v \in \Hom(F_2(V),F_1(V))$ by $[V,v]$  as an element of $\Hom^{\vee}(F_1,F_2)$, which is clearly generated as a vector space by all the $[V,v]$ up to the relations~\eqref{eq:Zquotientsubspace}. The algebra structure is defined as follows on the generators, where we draw $F_1$ with a blue box and $F_2$ with a red box:
\begin{calign}\label{eq:Zmultdef}
\includegraphics{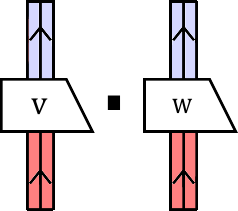}
~~=~~
\includegraphics{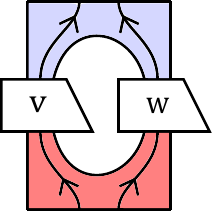}
\\
[V,v] \cdot [W,w] = [V \otimes W, m_{F_1} \circ (v \otimes w) \circ m_{F_2}^{\dagger}]
\end{calign}
\ignore{
Here $m_{F_1}$ is the multiplicator of the monoidal functor $F_1$.} The unit of the algebra is $[\mathbbm{1},u_1 \circ u_2^{\dagger}]$, where $u_i$ is the unitor for $F_i$.

The involution is defined as follows on the generators:
\begin{calign}\label{eq:Zinvoldef}
\includegraphics{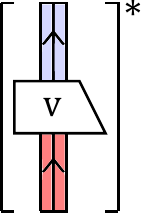}
~~=~~
\includegraphics{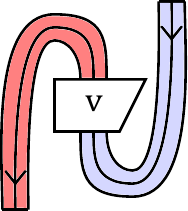}
\\
[V,v]^* = [V^*,(v^{\dagger})^{T}]
\end{calign}

\ignore{
$\Hom^{\vee}(F_1,F_2)$ is the predual of $\Hom(F_1,F_2)$; for any natural transformation $u: F_1 \to F_2$ we obtain a linear form $f_u: \Hom^{\vee}(F_1,F_2) \to \mathbb{C}$, defined on the generators as~\cite{}:
\begin{calign}
\includegraphics{Figures/svg/hopfgalois/linformdef.png}
\\
f_u([V,v]) = \Tr(u_V \circ v)
\end{calign}
For unitary monoidal natural transformations $u$, the map $u \mapsto f_u$
lands in $\Hom_*(\Hom^{\vee}(F_1,F_2),\mathbb{C})$, whence Proposition~\ref{}. The form~\eqref{} will be generalised by~\eqref{}.
}

\subsection{UPTs as $*$-representations of Hopf-Galois objects}
\label{sec:higherhopfgal}
We have just recalled that, for a compact quantum group $G$, the category $\Fib(G)$ of fibre functors on the pivotal dagger category $\Rep(G)$ and unitary monoidal natural isomorphisms is characterised by Hopf-Galois theory. In~\cite{Verdon2020a} we showed that $\Fib(G)$ can be generalised to a dagger 2-category $\Fun(\Rep(G),\Hilb)$ of fibre functors, unitary pseudonatural transformations and modifications. We will now see that this dagger 2-category is also characterised by Hopf-Galois theory: UPTs correspond to finite-dimensional $*$-representations of bi-Hopf-Galois objects, just as unitary monoidal natural transformations correspond to their one-dimensional $*$-representations.

Let $\mathcal{C}$ be a $C^*$-tensor category with conjugates, let $F_1,F_2: \mathcal{C} \to \Hilb$ be fibre functors, and let $Z$ be the corresponding $A_{G_2}$-$A_{G_1}$-bi-Hopf-Galois object. We will first show that for every unitary pseudonatural transformation $(\alpha,H): F_1 \to F_2$, one can construct a $*$-representation $\pi_{\alpha}:Z \to B(H)$. 

\begin{construction}\label{constr:alphatopi}
Recall the generators-and-relations description of $Z$ from Section~\ref{sec:hopfgalbackground}. We define the following map $\pi_{\alpha}: \oplus_{V \in \Obj(\mathcal{C})} \Hom(F_2(V),F_1(V)) \to H \otimes H^* \cong B(H)$ by its action on generators $[V,v]$:
\begin{calign}\nonumber
\includegraphics[scale=1]{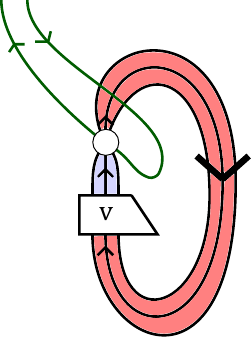}
\\
\pi_{\alpha}([V,v])
\end{calign}
Note that here we are taking the trace with respect to the dual $F_2(V)^*$ in $\Hilb$ (see~\eqref{eq:dualscompare}).
\end{construction}
\begin{proposition}
Construction~\ref{constr:alphatopi}  defines a $*$-representation $\pi_{\alpha}: Z \to B(H)$. Moreover, any modification $f: (\alpha,H) \to (\beta,H')$ induces an intertwiner $H \to H'$.
\end{proposition}
\begin{proof}
We first show that $\pi_{\alpha}$ induces a well-defined map on $Z= \oplus_{V \in \Obj(\mathcal{C})} \Hom(F_2(V),F_1(V)) / \mathcal{N}$, where $\mathcal{N}$ is the subspace defined in~\eqref{eq:Zquotientsubspace}. For this it is sufficient to show that, for any $f: V \to W$ in $\mathcal{C}$ and $x: F_2(W) \to F_1(V)$ in $\mathcal{D}$, we have $ \pi_{\alpha}([V,x \circ F_2(f)]) = \pi_{\alpha}([W,F_1(f) \circ x])$:
\begin{calign}
\includegraphics[scale=.8]{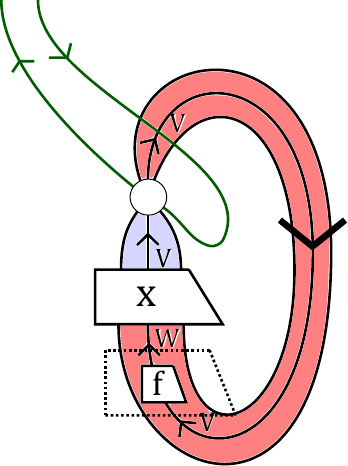}
~~=~~\ignore{
\includegraphics[scale=.8]{Figures/svg/higherhopfgal/maptobhwd2.pdf}
~~=~~}
\includegraphics[scale=.8]{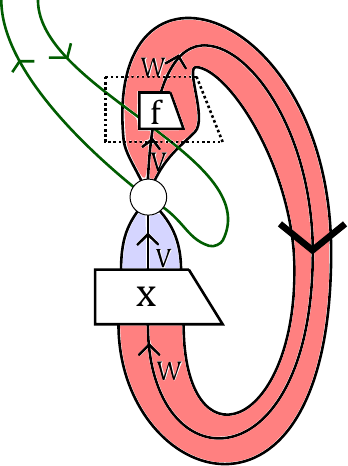}
~~=~~
\includegraphics[scale=.8]{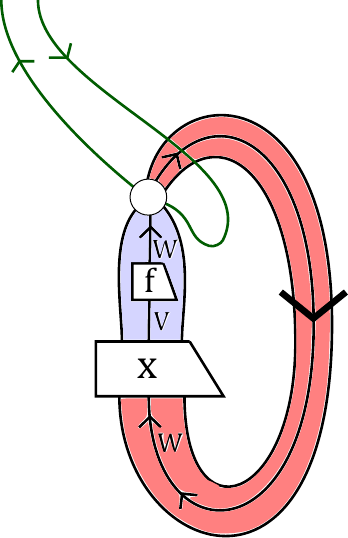}
\end{calign}
Here for the first equality we slide the morphism $F_2(f): F_2(V) \to F_2(W)$ in the dotted box around the loop using the graphical calculus of $\Hilb$ (Theorem~\ref{thm:graphcalcpiv}); the second equality is by naturality of $\alpha$. We therefore indeed have a map $Z \to B(H)$, which we now show is a $*$-homomorphism. 
\begin{itemize}
\item \emph{Multiplicative.} Recalling the definition of the multiplication~\eqref{eq:Zmultdef} of $Z$, we show $\pi_{\alpha}([V,v] \cdot [W,w]) = \pi_{\alpha}(v) \pi_{\alpha}(w)$:
\begin{calign}
\includegraphics[scale=.8]{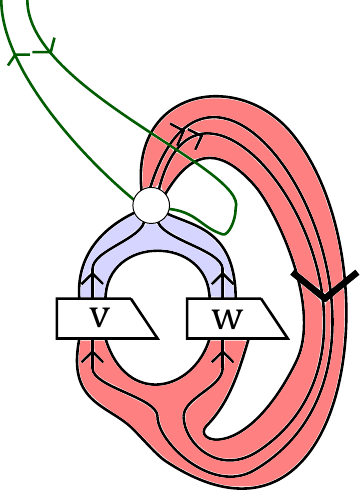}
~~=~~
\includegraphics[scale=.8]{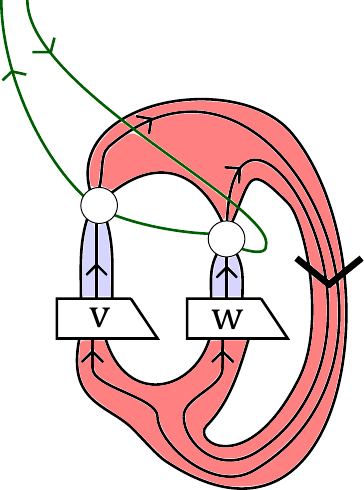}
~~=~~
\includegraphics[scale=.8]{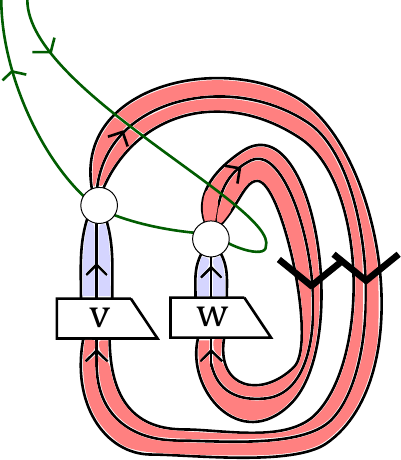}
~~=~~
\includegraphics[scale=.8]{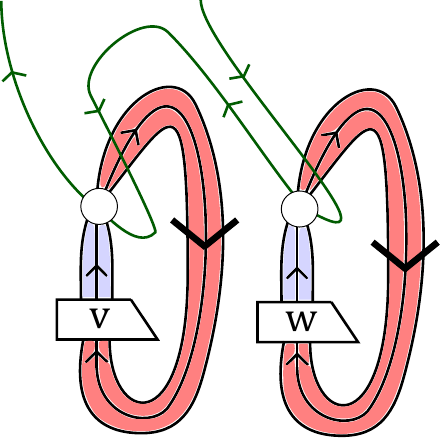}
\end{calign}
Here the first equality is by monoidality of $\alpha$, and for the second we slide the multiplicator $F_2(V \otimes W) \to F_2(V) \otimes F_2(W)$ around the loop and use unitarity of $F_2$ to cancel it and its inverse. \ignore{Note that here we have assumed the dual of $F_2(V) \otimes F_2(W)$ in $\Hilb$ is $F_2(W)^* \otimes F_2(V)^*$; to be pedantic we could have inserted a unitary isomorphism and its inverse.} The third equality is by the graphical calculus of the symmetric monoidal category $\Hilb$. In the last diagram we recognise the multiplication of $B(H) \cong H \otimes H^*$~\eqref{eq:endoalgebra}.
\item \emph{Unital.} Recalling that the unit of $Z$ is $[\mathbbm{1},u_{1} \circ u^{\dagger}_{2}]$, we observe:
\begin{calign}
\includegraphics[scale=.8]{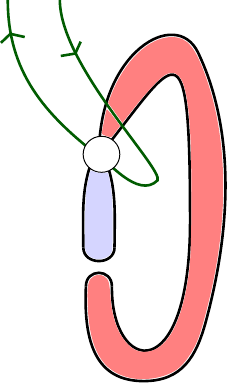}
~~=~~
\includegraphics[scale=.8]{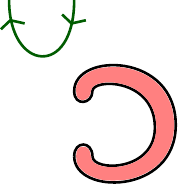}
~~=~~
\includegraphics[scale=.8]{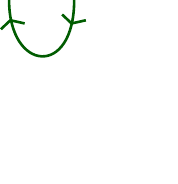}
\end{calign}
Here the first equality is by monoidality of $\alpha$; for the second we slide the unitor around the loop and use unitarity of $F_2$ to cancel it with its inverse. In the final diagram we recognise the unit of $B(H) \cong H \otimes H^*$~\eqref{eq:endoalgebra}.
\item \emph{Involution-preserving.} Recalling the definition of the involution~\eqref{eq:Zinvoldef} on generators $[V,v]$, we have the following equations. We use a large double upwards arrow for the dual $F_2(V^*)^*$ in $\Hilb$, and a large downwards arrow for the dual $F_2(V)^*$ in $\Hilb$:
\begin{calign}\label{eq:maptobhstareq1}
\includegraphics[scale=.8]{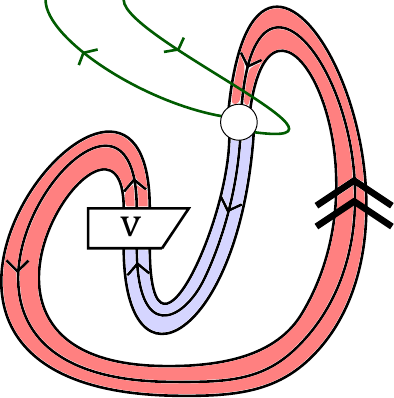}
~~=~~
\includegraphics[scale=.8]{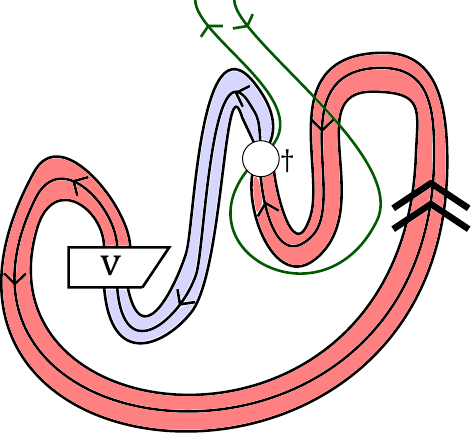}
~~=~~
\includegraphics[scale=.8]{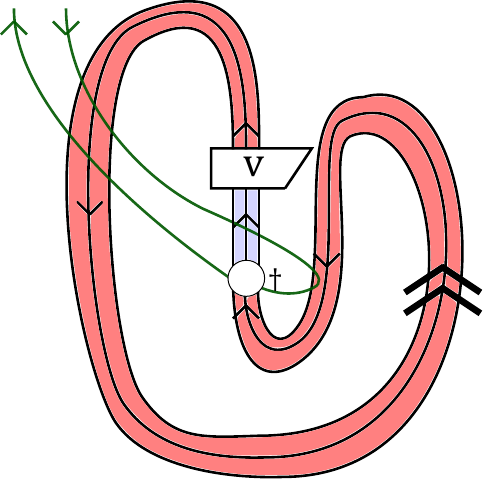}
\end{calign}
Here the first equality is by unitarity of $\alpha$~\eqref{eq:dualpnt}, and the second equality is by a snake equation for the induced duality on $F_1(V)$. We now observe that, since $F_2(V^*)$ is dual to $F_2(V)$ in $\Hilb$ by the induced cup and cap, by Proposition~\ref{prop:relateduals} there is an invertible morphism $f: F_2(V^*)^* \to F_2(V)$ relating the two caps:
\begin{calign}\label{eq:maptobhstarcap}
\includegraphics[scale=.8]{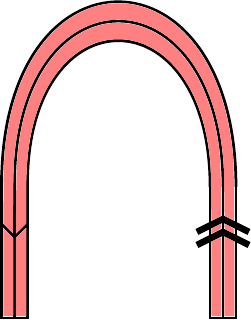}
~~=~~
\includegraphics[scale=.8]{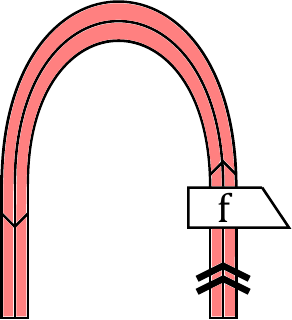}
\end{calign} 
We now continue~\eqref{eq:maptobhstareq1}:
\begin{calign}
\includegraphics[scale=.8]{Figures/svg/higherhopfgal/maptobhstar3.pdf}
~~=~~
\includegraphics[scale=.8]{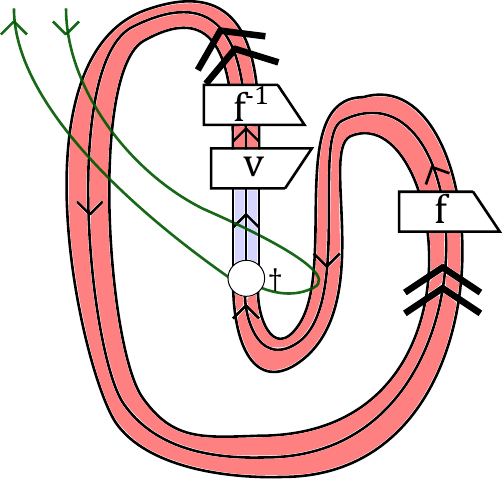}
~~=~~
\includegraphics[scale=.8]{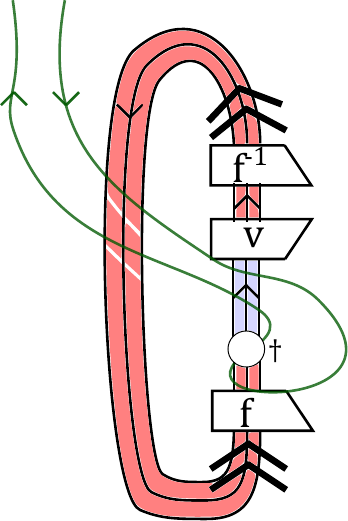}
\\
=~~
\includegraphics[scale=.8]{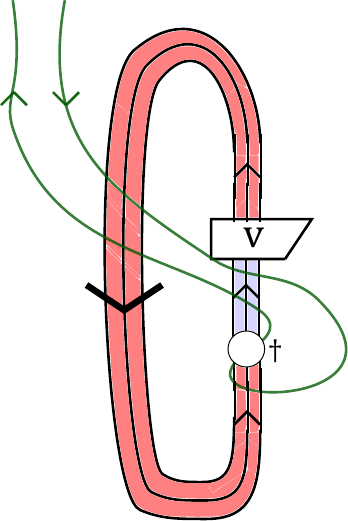}
~~=~~
\includegraphics[scale=.8]{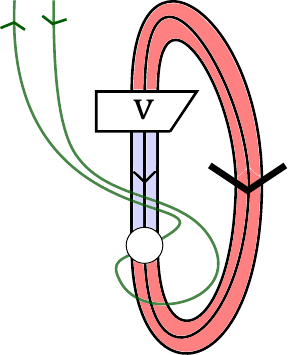}
~~=~~
\includegraphics[scale=.8]{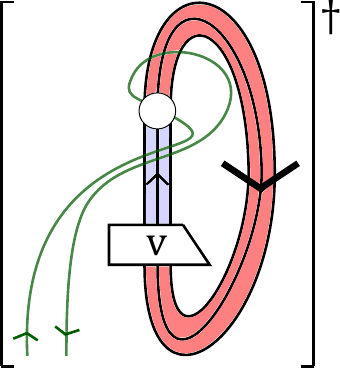}
~~=~~
\includegraphics[scale=.8]{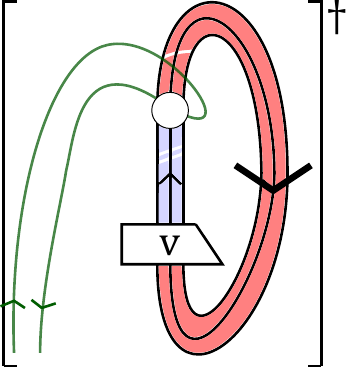}
\end{calign}
For the first equality we used~\eqref{eq:maptobhstarcap}; for the second we used a snake equation for the induced dual on $F_2(V)$; for the third we slid $f^{-1}$ around the loop and cancelled it with $f$; for the fourth we used the graphical calculus of $\Hilb$ to pull the loop around; the fifth equality is by the `horizontal reflection' calculus of the dagger in $\Hilb$; and for the sixth equality we used the graphical calculus of $\Hilb$ to untangle the $H$-wires. In the final diagram we observe the involution of $B(H)$~\eqref{eq:endoalgebra}.
\end{itemize}
Finally, every modification clearly induces an intertwiner:
\begin{calign}
\includegraphics[scale=1]{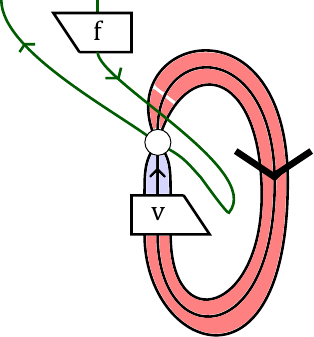}
~~=~~
\includegraphics[scale=1]{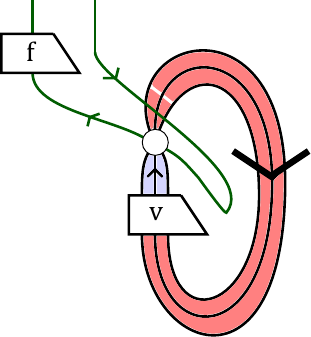}
\end{calign}
\end{proof}
\noindent
We will now produce a construction in the other direction, which gives a UPT $F_1 \to F_2$ for any $*$-representation of $Z$.
\begin{construction}\label{constr:pitoalpha}
Let $\pi: Z \to B(H)$ be a $*$-representation. Recall that $Z \cong \oplus_{V \in \Obj(\mathcal{C})} F_1(V) \otimes F_2(V)^*/\mathcal{N}$. For every object $V$ of $\mathcal{C}$ we define a map $U_V: F_1(V) \to F_2(V) \otimes Z$ as follows (c.f.~\cite[Thm 2.3.11]{Neshveyev2013}):
\begin{calign}\label{eq:zcorep}
\includegraphics[scale=.8]{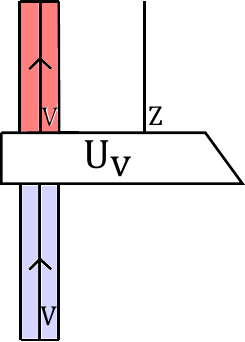}
~~:=~~
\includegraphics[scale=.8]{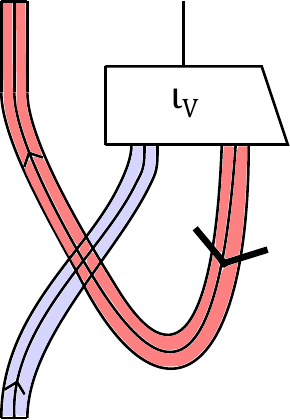}
\end{calign}
Here $\iota_V: F_1(V) \otimes F_2(V)^* \to Z$ is the canonical map induced by the direct sum.

We then define a map $(\alpha_{\pi})_V: F_1(V) \otimes H \to H \otimes F_2(V) $ (here $\pi: Z \to B(H) \cong H \otimes H^*$ is represented by a white vertex):
\begin{calign}
\includegraphics[scale=.8]{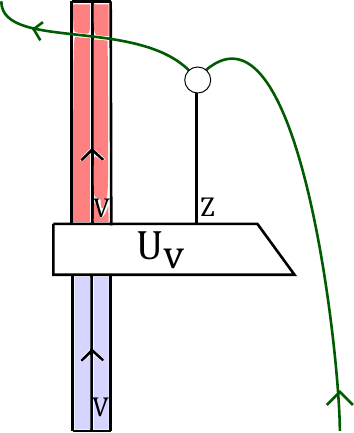}
\end{calign}
\end{construction}
\begin{remark}\label{rem:zwire}
The diagrams with a $Z$-wire are to be interpreted in the symmetric monoidal category $\Vec$ of possibly infinite-dimensional complex vector spaces, since $Z$ is not generally a finite-dimensional Hilbert space. We use more structure (wire bending, dagger, etc.) only on the parts of the diagram which come from $\Hilb$. 
\end{remark}
\begin{proposition}
Construction~\ref{constr:pitoalpha} defines a UPT $(\alpha_{\pi},H): F_1 \to F_2$. Moreover, any intertwiner $f: (\pi,H) \to (\pi',H')$ induces a modification $f: (\alpha_{\pi},H) \to (\alpha_{\pi'},H')$.
\end{proposition}
\begin{proof}
We show that the map $U_V$ satisfies certain properties which will imply that $\alpha_{\pi}$ is a UPT.
\begin{itemize}
\item \emph{Naturality}. For any $f:V \to W$ in $\mathcal{C}$: 
\begin{calign}
\includegraphics[scale=.8]{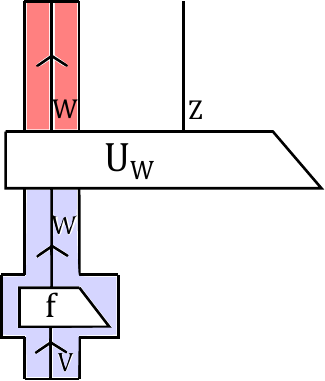}
~~=~~
\includegraphics[scale=.8]{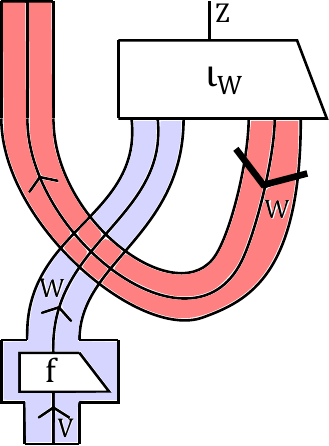}
~~=~~
\includegraphics[scale=.8]{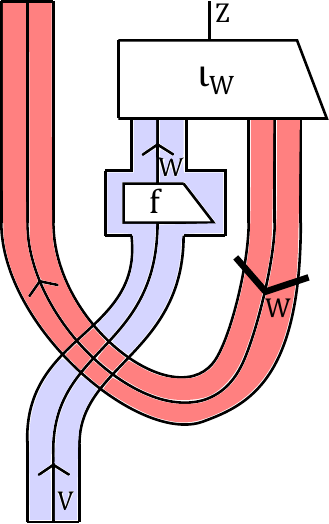}
~~=~~
\includegraphics[scale=.8]{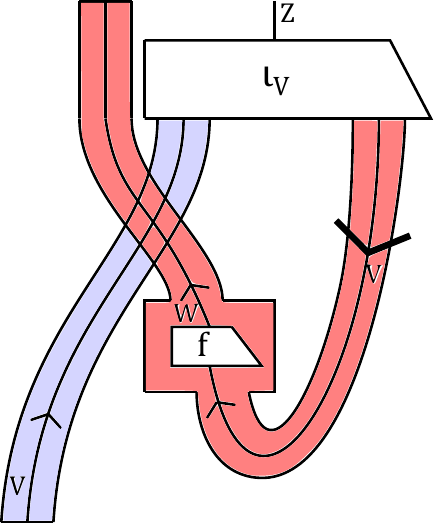}\\
=~~
\includegraphics[scale=.8]{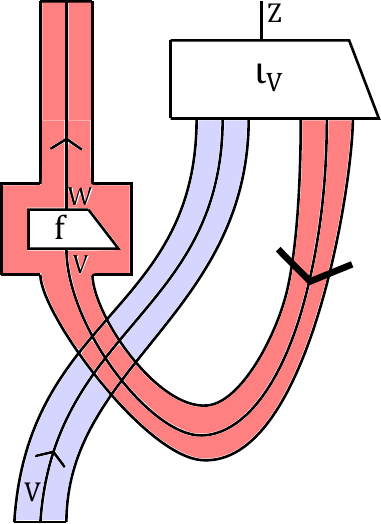}
~~=~~
\includegraphics[scale=.8]{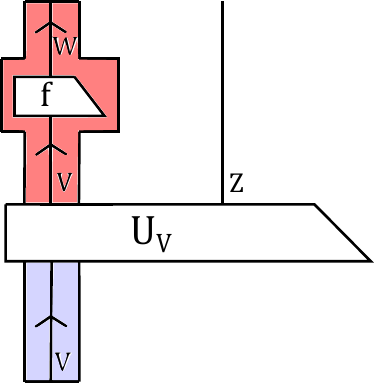}
\end{calign}
Here the first and fifth equalities are by definition, and the second and fourth equalities are by naturality of the symmetry in $\Hilb$. The third equality follows from the definition of $Z$ as the quotient space $\oplus_{V \in \Obj(\mathcal{C})} F_1(V) \otimes F_2(V)^*/\mathcal{N}$.
\ignore{can be seen by picking orthonormal bases $\{\ket{i}\}$ for $F(V)$ and $\{\ket{j}\}$ for $G(W)$ --- for every pair $\ket{i},\ket{j}$ we then have:
\begin{calign}
\includegraphics[scale=.8]{Figures/svg/higherhopfgal/uptcoactionnatjump1.png}
~~=~~
\includegraphics[scale=.8]{Figures/svg/higherhopfgal/uptcoactionnatjump2.png}
~~=~~
\includegraphics[scale=.8]{Figures/svg/higherhopfgal/uptcoactionnatjump3.png}
~~=~~
\includegraphics[scale=.8]{Figures/svg/higherhopfgal/uptcoactionnatjump4.png}
\end{calign}
Here the second equality is by~\eqref{eq:Zquotientsubspace} and the others by the graphical calculus of $\Hilb$.}
\item \emph{Monoidality.} 
\begin{itemize}
\item For any object $V$ of $\mathcal{C}$ let $\{\ket{i}\}$ and $\{\ket{j}\}$ be orthonormal bases of $F_1(V)$ and $F_2(V)$ respectively. Then inserting resolutions of the identity in~\eqref{eq:zcorep} we obtain the following expression for the map $U_V$:
\begin{equation}\label{eq:ukets}
U_V = \sum_{i,j} \ket{j}\bra{i} \otimes [V, \ket{i}\bra{j}]
\end{equation}
We observe that the map $U_V$ does not depend on the choice of orthonormal basis, since it was defined in a basis-independent way to begin with. 
We want to show the following equation:
\begin{calign}\label{eq:umodmonoidal1}
\includegraphics[scale=.8]{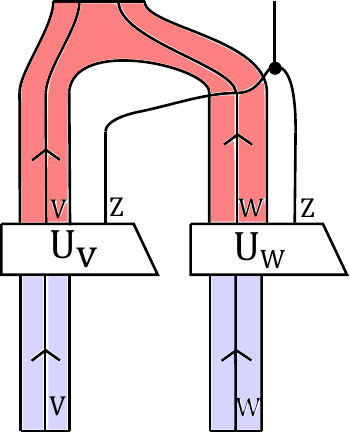}
~~=~~\ignore{
\includegraphics[scale=.8]{Figures/svg/higherhopfgal/uptcoactionmon2.png}
~~=~~
\includegraphics[scale=.8]{Figures/svg/higherhopfgal/uptcoactionmon3.png}
\\
~~=~~}
\includegraphics[scale=.8]{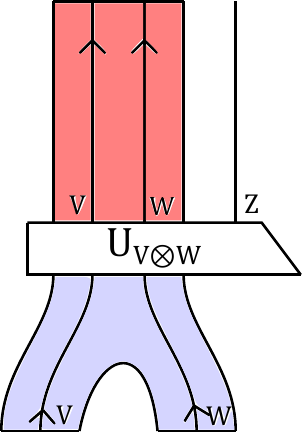}
\end{calign}
\ignore{Here the first and third equalities are by definition (note we have written the multiplication of $Z$ as a black vertex). For the second equality we use the graphical calculus of $\Hilb$ to untangle the wires and then slide the transpose of the comultiplicator of $F_2$ around the loop to the left, cancelling it with the multiplicator.}
For this we have the following equation, where $\{\ket{i}\}$, $\{\ket{j}\}$, $\{\ket{k}\}$, $\{\ket{l}\}$ are orthonormal bases of $F_1(V)$, $F_2(V)$, $F_1(W)$ and $F_2(W)$ respectively:
\begin{align*}
\sum_{i,j,k,l} (m_{V,W} \circ (\ket{j} \bra{i} \otimes \ket{l}\bra{k})) \otimes [V \otimes W, m_{V,W} \circ (\ket{i}\bra{j} \otimes \ket{k}\bra{l}) \circ m_{V,W}^{\dagger}]
\\
=\sum_{i,j,k,l} (m_{V,W} \circ (\ket{j} \bra{i} \otimes \ket{l}\bra{k}) \circ m_{V,W}^{\dagger} \circ m_{V,W}) \otimes [V \otimes W, m_{V,W} \circ (\ket{i}\bra{j} \otimes \ket{k}\bra{l}) \circ m_{V,W}^{\dagger}]
\end{align*}
Looking at the last expression, we observe that it has the form~\eqref{eq:ukets} for the orthonormal bases $m_{V,W} \circ (\ket{i} \otimes \ket{k})$ and $m_{V,W} \circ (\ket{j} \otimes \ket{l})$ of $F_1(V \otimes W)$ and $F_2(V \otimes W)$. The equation~\eqref{eq:umodmonoidal1} follows.
\item \begin{calign}
\includegraphics[scale=.8]{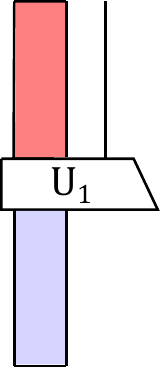}
~~=~~
\includegraphics[scale=.8]{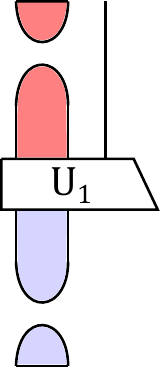}
~~=~~
\includegraphics[scale=.8]{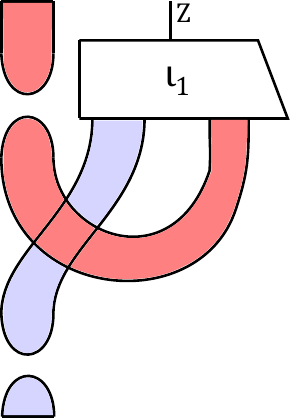}
~~=~~\includegraphics[scale=.8]{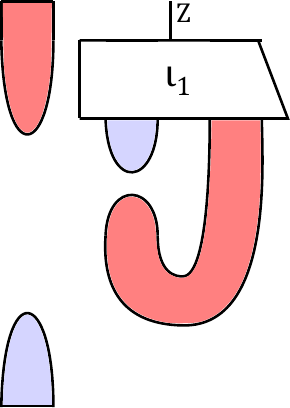}
~~=~~\includegraphics[scale=.8]{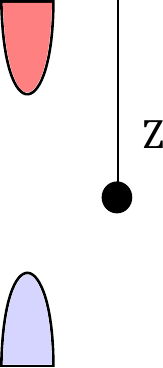}
\end{calign}
\end{itemize}
Here the first equality is by unitarity of the unitors of $F,F'$; the second equality is by definition of $U_{\mathbbm{1}}$; the third equality is by the graphical calculus of $\Hilb$; and the fourth equality is by definition of the unit of $Z$ (which we have drawn as a black vertex).
\item \emph{Unitarity}. 
Choose orthonormal bases of $F(V),F'(V)$. We show that with respect to such a basis $U_V$ is a unitary matrix. We prove the first unitarity equation $\sum_k [U_V]_{ki}^* [U_V]_{kj} = \delta_{ij} 1_Z$. Let $\{\ket{i}\},\{\ket{j}\}$ be elements of the orthonormal basis of $F(V)$ and let $\{\ket{k}\}$ be the orthonormal basis of $F'(V)$. Then:
\begin{calign}
{\Huge \sum_k}~\includegraphics[scale=0.8]{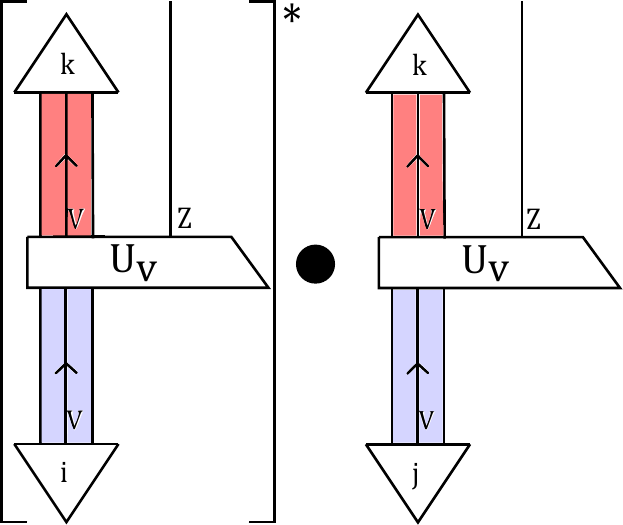}
~~=~~
{\Huge \sum_k}~\includegraphics[scale=0.8]{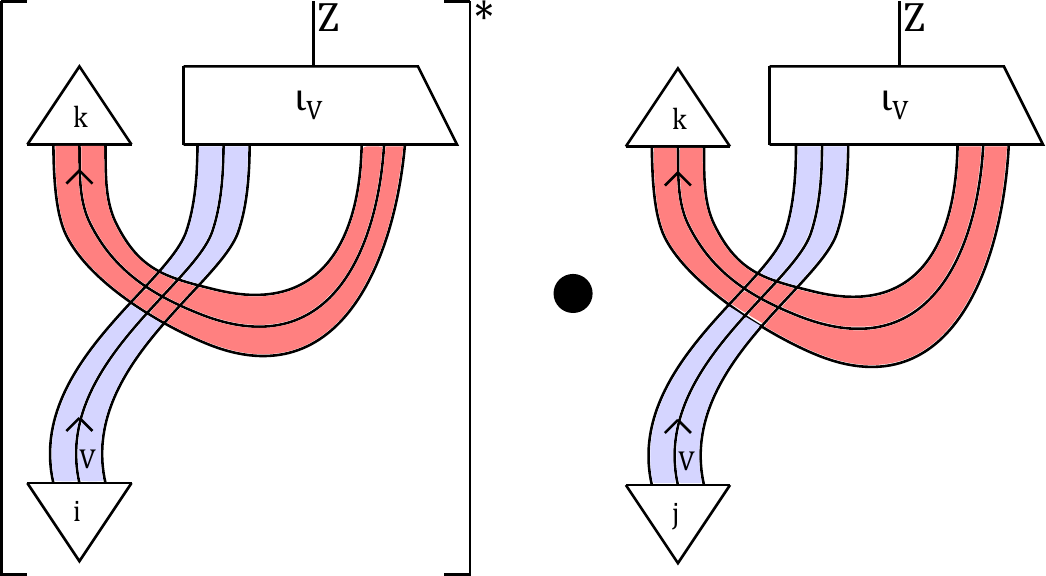}\\
=~~{\Huge \sum_k}~~\includegraphics[scale=0.8]{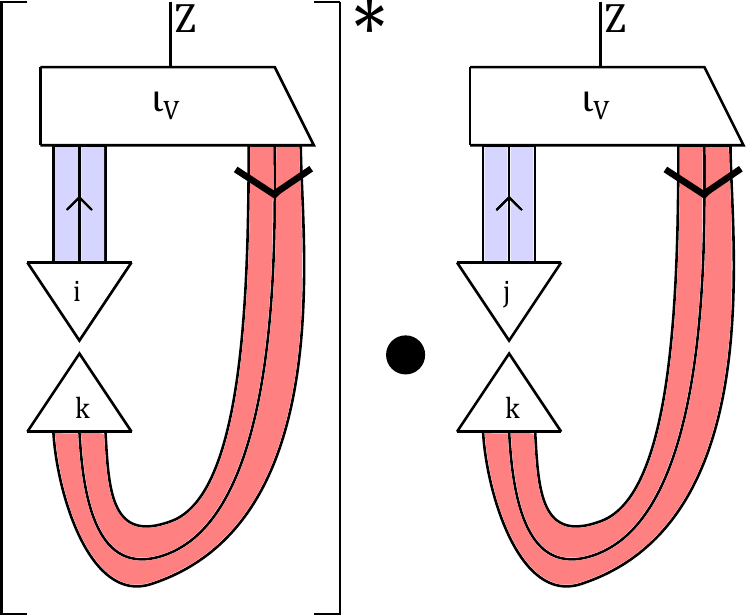}
~~=~~
{\Huge \sum_k}~~\includegraphics[scale=0.8]{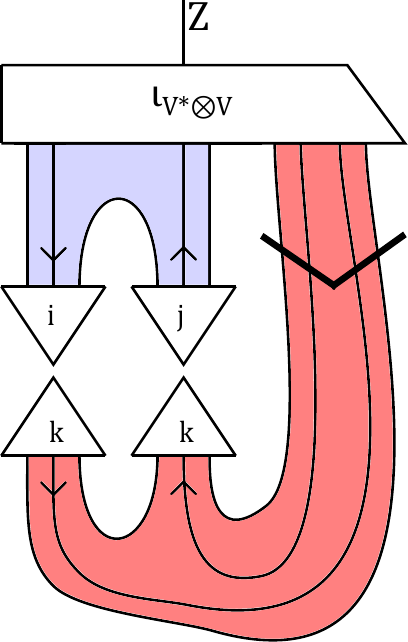}\\
=~~
{\Huge \sum_k}~~\includegraphics[scale=0.8]{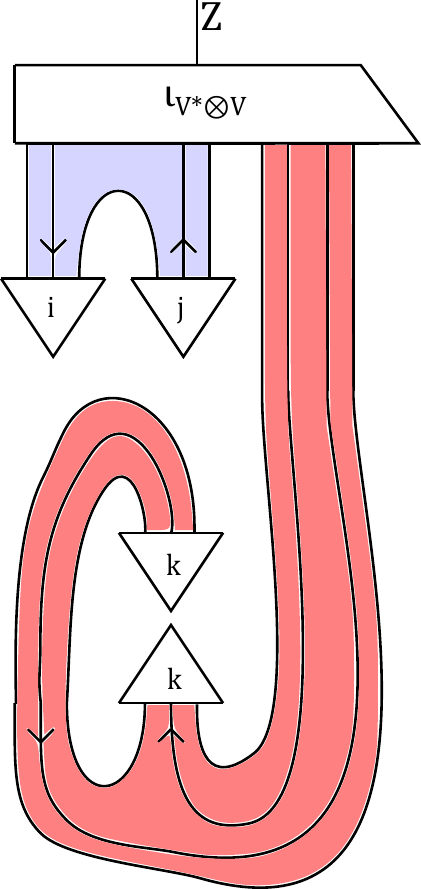}
~~=~~
\includegraphics[scale=0.8]{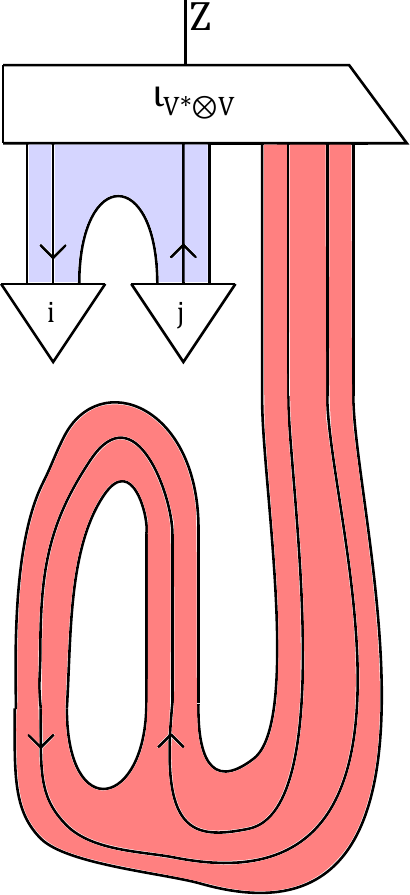}
~~=~~
\includegraphics[scale=0.8]{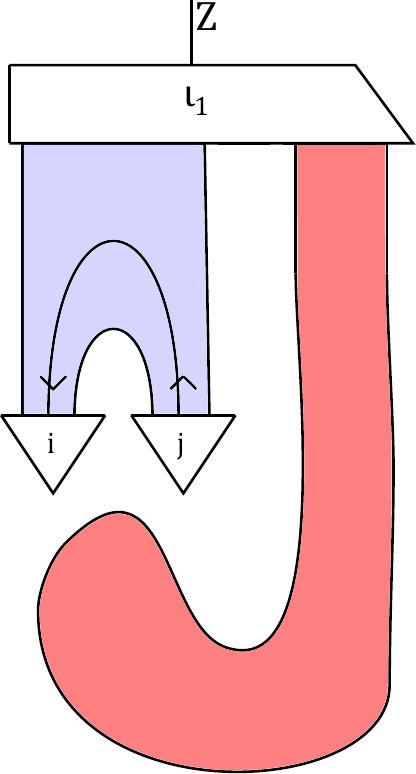}\\
=~~\includegraphics[scale=0.8]{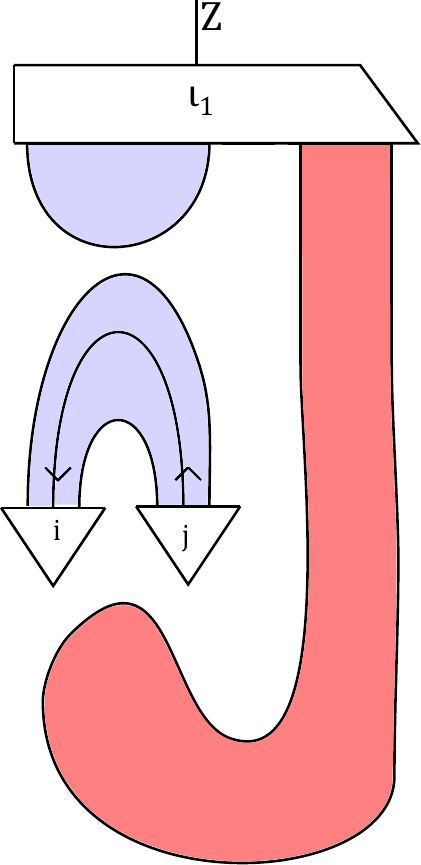}
~~=~~ 
{\Huge \delta_{ij}}~~
\includegraphics[scale=0.8]{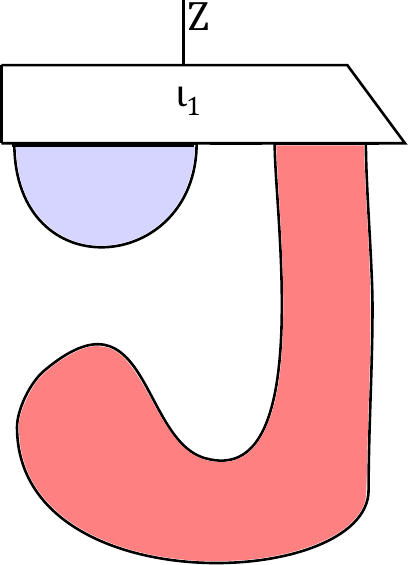}
\end{calign}
Here for the first equality we used the definition of $U_V$; for the second equality we used the graphical calculus of $\Hilb$; for the third equality we used the definition of multiplication and involution in $Z$; for the fourth equality we used the definition of the involution in $Z$; for the fifth equality we removed the resolution of the identity and used unitarity of the functor $F'$; for the sixth equality we used the definition of the quotient space defining $Z$~\eqref{eq:Zquotientsubspace}; for the seventh equality we used unitarity of the functor $F$; and for the eighth equality we used the induced duality from $\mathcal{C}$ and orthogonality of the basis. In the final diagram we recognise the unit of $Z$.

The second unitarity equation, $\sum_k [U_V]_{ik} [U_V]_{jk}^* = \delta_{ij} 1_Z$ is shown similarly.
\end{itemize}
We now show that $\alpha_{\pi}$ is a UPT.
\begin{itemize}
\item \emph{Naturality}. Follows immediately from naturality of $\{U_V\}$ and of the swap. 
\item \emph{Monoidality}. 
\begin{calign}
\includegraphics[scale=.8]{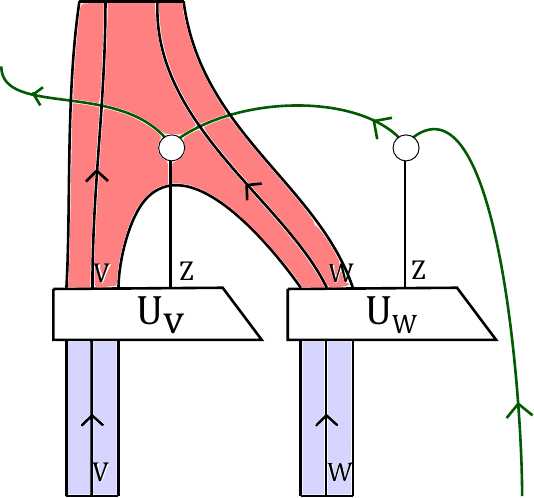}
~~=~~
\includegraphics[scale=.8]{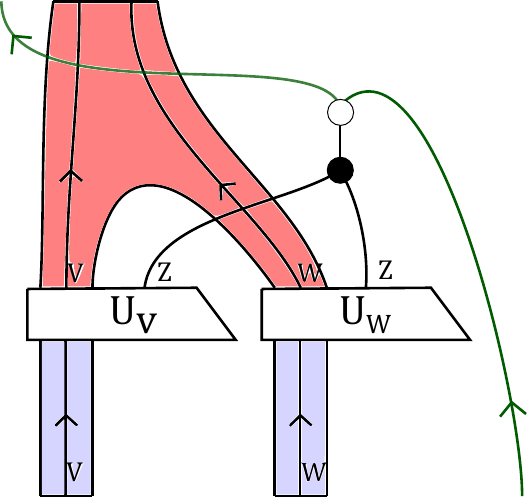}
~~=~~
\includegraphics[scale=.8]{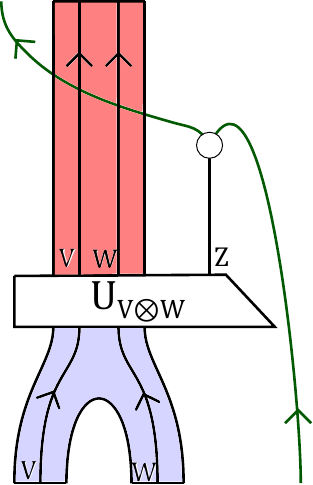}
\end{calign}
In the first equation we recognised the multiplication of $B(H) \cong H \otimes H^*$~\eqref{eq:endoalgebra} and used the fact that $\pi$ is a homomorphism. In the second equation we used monoidality of $U$.
\item \emph{Unitarity}. \ignore{The isometry $\iota_V: \Hom(F_2(V),F_1(V)) \hookrightarrow Z$ has been assumed thus far in our diagrams; for clarity we will now explicitly depict it.} We must show that $\alpha_V^{\dagger}$ is a 2-sided inverse for $\alpha_V$. Consider the first of the 2  unitarity equations:
\begin{calign}\label{eq:reptoalphaunitarystatement}
\includegraphics[scale=.8]{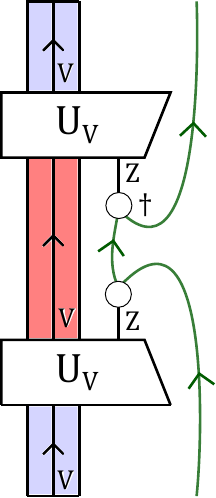}
~~=~~
\includegraphics[scale=.8]{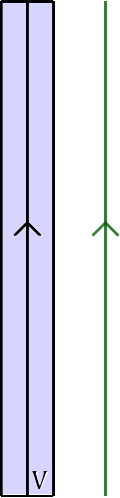}
\end{calign}
\begin{remark}
As was already stated in Remark~\ref{rem:zwire}, there is not necessarily a Hilbert space structure on $Z$; however, in~\eqref{eq:reptoalphaunitarystatement} the dagger of $\pi:Z \to B(H)$ appears. This is just a convenient way of representing the dagger of the morphism $(\id_{F_2(V)} \otimes \pi) \circ U_V$ and should only be interpreted in that context.
\end{remark}
To prove~\eqref{eq:reptoalphaunitarystatement}, we choose orthonormal bases for $F_1(V)$ and $F_2(V)$. Let $\ket{i},\ket{j}$ be any two elements of the orthonormal basis for $F_1(V)$ and let $\{\ket{k}\}$ be the orthonormal basis elements of $F_2(V)$. Then we have the following equation:
\begin{calign}\label{eq:reptoalphaeq1}
\includegraphics[scale=.8]{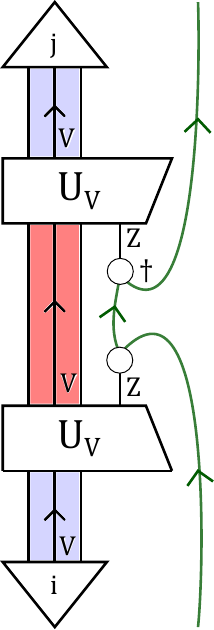}
~~=~~
{\Huge \sum_k}~~
\includegraphics[scale=.8]{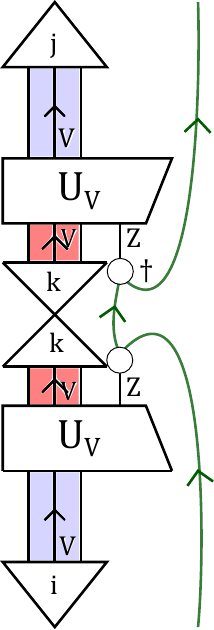}
\end{calign}
Here we simply inserted a resolution of the identity in $F_2(V)$. Now we define the following elements of $B(H)$:
\begin{calign}
\includegraphics[scale=.8]{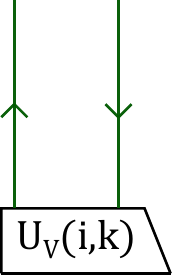}
~~:=~~
\includegraphics[scale=.8]{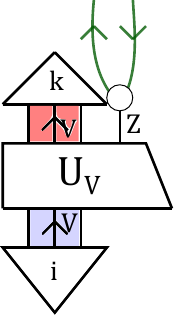}
\end{calign}
With these elements~\eqref{eq:reptoalphaeq1} becomes:
\begin{calign}\label{eq:reptoalphaeq2}
{\Huge \sum_k}~~
\includegraphics[scale=.8]{Figures/svg/higherhopfgal/alphapiunitpf12.pdf}
~~=
{\Huge \sum_k}~~\includegraphics[scale=.8]{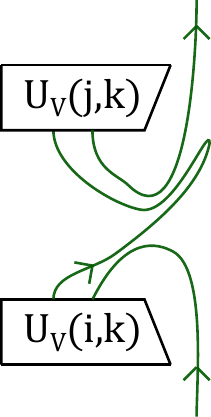}
~~=
{\Huge \sum_k}~~\includegraphics[scale=.8]{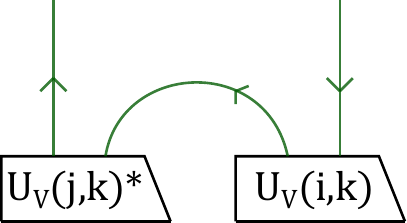}
~~=~~
{\huge \delta_{ij}}~\includegraphics[scale=.8]{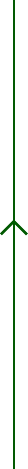}
\end{calign}
Here in the second equality we noticed the involution and multiplication of $B(H) \cong H \otimes H^*$~\eqref{eq:endoalgebra}. For the third equality we used involutivity and multiplicativity of $\pi$, and then unitarity of $U_V$. Therefore~\eqref{eq:reptoalphaunitarystatement} 
is proved. The proof of the other unitarity equation is similar.  
\ignore{
At this point we could use unitarity of $U_V$, but here is an alternative proof. We calculate $U_V(i,k)^*$:
\begin{calign}
\includegraphics[scale=.8]{Figures/svg/higherhopfgal/alphapiunitpfuikdef1.png}
~~=~~
\includegraphics[scale=.8]{Figures/svg/higherhopfgal/alphapiunitpfuikeq2.png}
~~=~~
\includegraphics[scale=.8]{Figures/svg/higherhopfgal/alphapiunitpfuikeq3.png}
~~=~~
\includegraphics[scale=.8]{Figures/svg/higherhopfgal/alphapiunitpfuikeq4.png}
\end{calign}
Here for the first equality we used the definition of $U_V$; for the second equality we used the graphical calculus and the fact that $\pi$ is involution-preserving; and for the third equality we used the definition of the involution in $Z$~\eqref{eq:Zinvoldef}. 
Using this calculation we continue~\eqref{eq:reptoalphaeq2}:
\begin{calign}
{\Huge \sum_k}~~\includegraphics[scale=.8]{Figures/svg/higherhopfgal/alphapiunitpf23.png}
~~={\Huge \sum_k}~~\includegraphics[scale=.8]{Figures/svg/higherhopfgal/alphapiunitpf32.png}
~~=~~\includegraphics[scale=.8]{Figures/svg/higherhopfgal/alphapiunitpf33.png}\\
~~=~~\includegraphics[scale=.8]{Figures/svg/higherhopfgal/alphapiunitpf34.png}
~~=~~{\huge \delta_{ij}}~\includegraphics[scale=.8]{Figures/svg/higherhopfgal/alphapiunitpf35.png}
\end{calign}
Here for the second equality we used the graphical calculus of $\Hilb$ to remove the resolution of the identity,  for the third equality we used naturality and monoidality of $U$.
Therefore~\eqref{eq:reptoalphaunitarystatement} is proved. The proof of the other unitarity equation is similar.}
\end{itemize}
Finally, an intertwiner clearly induces a modification:
\begin{calign}
\includegraphics[scale=.8]{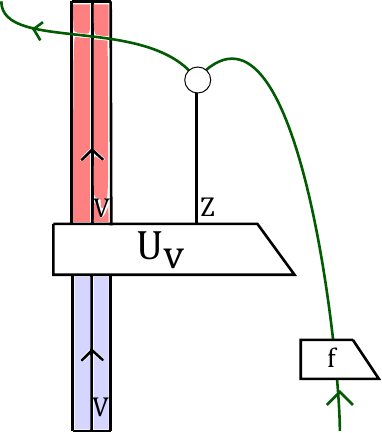}
~~=~~
\includegraphics[scale=.8]{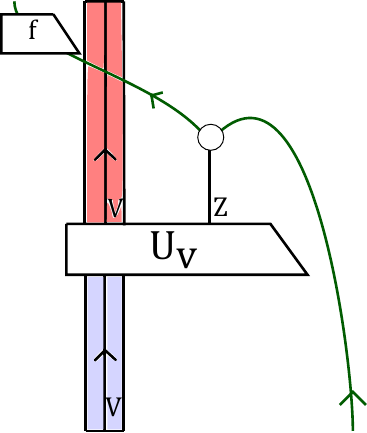}
\end{calign}
\end{proof}
\begin{theorem}\label{thm:higherhopfgal}
Let $\mathcal{C}$ be a $C^*$-tensor category with conjugates, and let $F_1,F_2:\mathcal{C} \to \Hilb$ be fibre functors. Let $Z = Hom^{\vee}(F_1,F_2)$ be the corresponding $A_{G_2}$-$A_{G_1}$-bi-Hopf-Galois-object. There is an isomorphism of categories between:
\begin{itemize}
\item The category $\Rep(Z)$ of f.d. $*$-representations of $Z$ and intertwiners.
\item The category $\Fun(F_1,F_2)$ of UPTs $F_1 \to F_2$ and modifications.  
\end{itemize}
\end{theorem}
\begin{proof}
We show that Constructions~\ref{constr:alphatopi} and~\ref{constr:pitoalpha} are strictly inverse. 
First we observe that, for a UPT $\alpha:F_1 \to F_2$, the $*$-homomorphism $\pi_{\alpha}: Z \to B(H)$ is defined on the subspace in the image of $\iota_V: F_1(V) \otimes F_2(V)^* \to Z$ as follows:
\begin{calign}
\includegraphics[scale=.8]{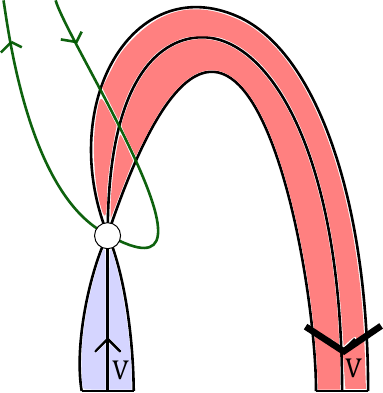}
\end{calign}
We also observe that, for a $*$-homomorphism $\pi: Z \to B(H)$, the UPT $\alpha_{\pi}: F_1 \to F_2$ is defined on $V$ as follows:
\begin{calign}
\includegraphics[scale=.8]{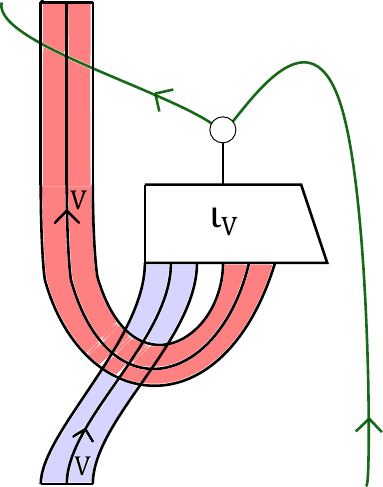}
\end{calign}
For one direction, let $\alpha:F_1 \to F_2$ be a UPT and use Construction~\eqref{constr:alphatopi} to obtain a $*$-representation $\pi_{\alpha}: Z \to B(H)$. Now use Construction~\ref{constr:pitoalpha} to obtain a UPT $F_1 \to F_2$. That the resulting UPT is $\alpha$ follows immediately from the graphical calculus of $\Hilb$:
\begin{calign}
\includegraphics[scale=.8]{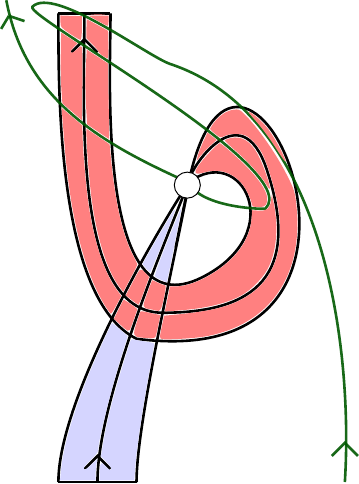}
~~=~~
\includegraphics[scale=.8]{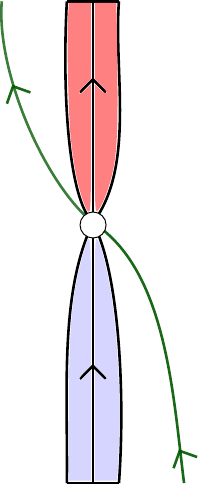}
\end{calign}
In the other direction, let $\pi$ be a $*$-representation $\pi_{\alpha}: Z\to B(H)$ and use Construction~\eqref{constr:pitoalpha} to obtain a UPT  $\alpha_{\pi}:F_1 \to F_2$ . Now use Construction~\ref{constr:alphatopi} to obtain a $*$-representation $Z \to B(H)$. That the resulting UPT is $\alpha$ again follows immediately from the graphical calculus of $\Hilb$:
\begin{calign}
\includegraphics[scale=.8]{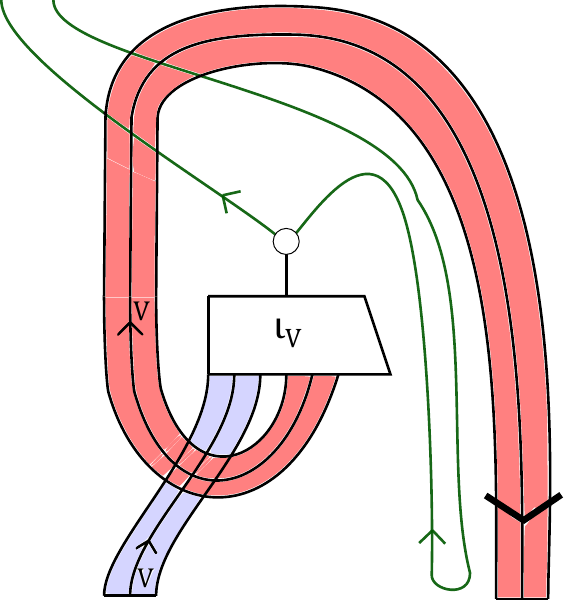}
~~=~~
\includegraphics[scale=.8]{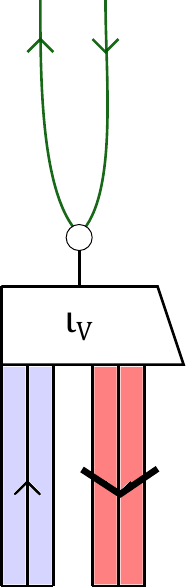}
\end{calign}
That the maps on modifications and intertwiners are inverse is clear. 
\end{proof}
\begin{corollary}
Two fibre functors $F_1,F_2$ on $\Rep(G)$ are related by a UPT $F_1 \to F_2$ precisely when the corresponding $A_{G_2}$-$A_{G_1}$-bi-Hopf-Galois object $\Hom^{\vee}(F_1,F_2)$ has a finite-dimensional $*$-representation.
\end{corollary}
\begin{definition}
Let $\mathcal{C}$ be a semisimple pivotal dagger category and let $F: \mathcal{C} \to \Hilb$ be a fibre functor. We say that a fibre functor $F': \mathcal{C} \to \Hilb$ is \emph{accessible from $F$} if there exists a UPT $\alpha: F \to F'$.
\end{definition}
\noindent
Let $G$ be a compact quantum group  and $F: \Rep(G) \to \Hilb$ the canonical fibre functor. We observe that a UPT --- since its components are unitary --- must preserve dimensions of Hilbert spaces in the sense that $\dim(F(V)) = \dim(F'(V))$, so an $A_G$-Hopf-Galois object corresponding to a fibre functor accessible from $F$ must be \emph{cleft}~\cite[Thm 1.17]{Bichon2014}. It is unknown to the author whether all cleft Hopf-Galois objects admit a finite-dimensional $*$-representation.

At least in one case the situation is clear, since cleft Hopf-Galois objects for a compact quantum group algebra $A$ are all obtained as cocycle twists of $A$~\cite[Thm 1.8]{Bichon2014}, so for finite-dimensional $A$ they are also finite-dimensional with a faithful state~\cite[Prop. 4.2.5, Cor. 4.3.5]{Bichon2014}.
\begin{corollary}\label{cor:finiteacc}
When $G$ is a finite CQG (i.e the algebra $A_G$ is finite-dimensional)  all fibre functors on $\Rep(G)$ are accessible from the canonical fibre functor by a UPT. 
\end{corollary}

\section{Morita theory of accessible fibre functors and UPTs}
\label{sec:morita}
In Section~\ref{sec:higherhopfgal} we showed that a fibre functor on the category $\Rep(G)$ is accessible from the canonical fibre functor $F$ by a UPT precisely when the corresponding $A_G$-Hopf-Galois object admits a finite-dimensional $*$-representation. In this Section we will use Morita theory to show that accesible fibre functors $F'$ and UPTs $F \to F'$ can be  classified in terms of the finite-dimensional representation theory of the compact quantum group algebra $A_G$.
 
\subsection{Background on Frobenius monoids and Morita equivalence}

Morita theory relates 1-morphisms $X: r \to s$ out of an object $r$ of a 2-category $\mathcal{C}$ to Frobenius monoids in its category of endomorphisms $\mathcal{C}(r,r)$. In our case, the 2-category in question is $\Fun(\Rep(G),\Hilb)$, and we consider the category $\End(F)$ of endomorphisms of the canonical fibre functor $F: \Rep(G) \to \Hilb$, which we have just shown (Theorem~\ref{thm:higherhopfgal}) is isomorphic to $\Rep(A_G)$. We now recall the definition of a Frobenius monoid, and two notions of equivalence which will be important in our classification.
\ignore{
For this theory to work fully, we need a technical condition on $\mathcal{C}$.
\begin{definition}
Let $\mathcal{C}$ be a pivotal dagger 2-category. We say that $\mathcal{C}$ has \emph{split dagger idempotents} when dagger idempotents split (Definition~\ref{}) in all of its Hom-categories.
\end{definition}
}
\ignore{Let $X:r \to s$ be a 1-morphism in $\mathcal{C}$. Then $X \circ X^*: r \to r$ is a object in the pivotal dagger category $\mathcal{C}(r,r)$. We show that it possesses a natural structure of a monoid.}
\paragraph{Frobenius monoids.}
\begin{definition}\label{def:Frobenius}
Let $\mathcal{C}$ be a monoidal dagger category. A \emph{monoid} in $\mathcal{C}$ is an object $A$ with multiplication and a unit morphisms, depicted as follows:%
\begin{calign}
\begin{tz}[zx,master]
\coordinate (A) at (0,0);
\draw (0.75,1) to (0.75,2);
\mult{A}{1.5}{1}
\end{tz}
&
\begin{tz}[zx,slave]
\coordinate (A) at (0.75,2);
\unit{A}{1}
\end{tz}
\\[0pt]\nonumber
m:A\otimes A \to A& u: \mathbbm{1} \to A 
\end{calign}\hspace{-0.2cm}
These morphisms satisfy the following associativity and unitality equations:
\begin{calign}\label{eq:assocandunitality}
\begin{tz}[zx]
\coordinate(A) at (0.25,0);
\draw (1,1) to [out=up, in=-135] (1.75,2);
\draw (1.75,2) to [out=-45, in=up] (3.25,0);
\draw (1.75,2) to (1.75,3);
\mult{A}{1.5}{1}
\node[zxvertex=\zxwhite,zxdown] at (1.75,2){};
\end{tz}
\quad = \quad
\begin{tz}[zx,xscale=-1]
\coordinate(A) at (0.25,0);
\draw (1,1) to [out=up, in=-135] (1.75,2);
\draw (1.75,2) to [out=-45, in=up] (3.25,0);
\draw (1.75,2) to (1.75,3);
\mult{A}{1.5}{1}
\node[zxvertex=\zxwhite,zxdown] at (1.75,2){};
\end{tz}
&
\begin{tz}[zx]
\coordinate (A) at (0,0);
\draw (0,-0.25) to (0,0);
\draw (0.75,1) to (0.75,2);
\mult{A}{1.5}{1}
\node[zxvertex=\zxwhite,zxdown] at (1.5,0){};
\end{tz}
\quad =\quad
\begin{tz}[zx]
\draw (0,0) to (0,2);
\end{tz}
\quad= \quad
\begin{tz}[zx,xscale=-1]
\coordinate (A) at (0,0);
\draw (0,-0.25) to (0,0);
\draw (0.75,1) to (0.75,2);
\mult{A}{1.5}{1}
\node[zxvertex=\zxwhite,zxdown] at (1.5,0){};
\end{tz}
\end{calign}
Analogously, a \textit{comonoid} is an object $A$ with a coassociative comultiplication $\delta: A \to A\otimes A$ and a counit $\epsilon:A\to \mathbb{C}$. The dagger of an monoid $(A,m,u)$ is a comonoid $(A,m^{\dagger},u^{\dagger})$.  A monoid $(A,m,u)$ in $\mathcal{C}$ is called \textit{Frobenius} if the monoid and adjoint comonoid structures are related by the following \emph{Frobenius equation}:
\begin{calign}\label{eq:Frobenius}
\begin{tz}[zx]
\draw (0,0) to [out=up, in=-135] (0.75,2) to (0.75,3);
\draw (0.75,2) to [out=-45, in=135] (2.25,1);
\draw (2.25,0) to (2.25,1) to [out=45, in=down] (3,3);
\node[zxvertex=\zxwhite,zxup] at (2.25,1){};
\node[zxvertex=\zxwhite,zxdown] at (0.75,2){};
\end{tz}
\quad = \quad
\begin{tz}[zx]
\coordinate (A) at (0,0);
\coordinate (B) at (0,3);
\draw (0.75,1) to (0.75,2);
\mult{A}{1.5}{1}
\comult{B}{1.5}{1}
\end{tz}
\quad = \quad 
\begin{tz}[zx]
\draw (0,0) to [out=up, in=-135] (0.75,2) to (0.75,3);
\draw (0.75,2) to [out=-45, in=135] (2.25,1);
\draw (2.25,0) to (2.25,1) to [out=45, in=down] (3,3);
\node[zxvertex=\zxwhite,zxup] at (2.25,1){};
\node[zxvertex=\zxwhite,zxdown] at (0.75,2){};
\end{tz}
\end{calign}
\end{definition}
\noindent
Frobenius monoids are canonically self-dual. Indeed, it is easy to see that for any Frobenius monoid the following cup and cap fulfil the snake equations~\eqref{eq:rightsnakes}:
\begin{calign}\label{eq:cupcapfrob}
\begin{tz}[zx]
\clip (-0.1,0) rectangle (2.1,2.);
\draw (0,0) to [out=up, in=up, looseness=2] node[zxvertex=\zxwhite, pos=0.5]{} (2,0);
\end{tz}
~:=~
\begin{tz}[zx]
\clip (-0.1,0) rectangle (2.1,2.2);
\draw (0,0) to [out=up, in=up, looseness=2] node[front,zxvertex=\zxwhite, pos=0.5](A){} (2,0);
\draw[string] (A.center) to (1,1.8);
\node[zxvertex=\zxwhite] at (1,1.8) {};
\end{tz}
&
\begin{tz}[zx,yscale=-1]
\clip (-0.1,0) rectangle (2.1,2.);
\draw (0,0) to [out=up, in=up, looseness=2] node[zxvertex=\zxwhite, pos=0.5]{} (2,0);
\end{tz}
~:=~
\begin{tz}[zx,yscale=-1]
\clip (-0.1,0) rectangle (2.1,2.2);
\draw (0,0) to [out=up, in=up, looseness=2] node[front,zxvertex=\zxwhite, pos=0.5](A){} (2,0);
\draw[string] (A.center) to (1,1.8);
\node[zxvertex=\zxwhite] at (1,1.8) {};
\end{tz}
\end{calign}
A Frobenius monoid is \emph{special} if the following additional equation is satisfied:
\begin{calign}
\begin{tz}[zx,every to/.style={out=up, in=down}]\draw (0,0) to (0,1) to [out=135] (-0.75,2) to [in=-135] (0,3) to (0,4);
\draw (0,1) to [out=45] (0.75,2) to [in=-45] (0,3);
\node[zxvertex=\zxwhite, zxup] at (0,1){};
\node[zxvertex=\zxwhite,zxdown] at (0,3){};\end{tz}
\quad = \quad 
\begin{tz}[zx]
\draw (0,0) to +(0,4);
\end{tz}
\end{calign}
\begin{example}\label{ex:normalisedendoalg}
The following normalisation of the endomorphism algebra of Definition~\ref{def:endoalgebra} is a special Frobenius monoid on $H \otimes H^*$, where $d$ is the dimension of $H$: 
\begin{calign}
\frac{1}{\sqrt{d}}~
\begin{tz}[zx,master,every to/.style={out=up, in=down}]
\draw[arrow data={0.5}{>}] (0,0) to [looseness=1.6] (1,3.5);
\draw[arrow data={0.5}{<}] (3,0) to [looseness=1.6] (2, 3.5);
\draw[arrow data={0.5}{<}] (1,0) to [out=up, in=up, looseness=3] (2,0);
\end{tz}
&
\sqrt{d}\!\!\!\!\!
\begin{tz}[zx,slave,every to/.style={out=up, in=down}]
\draw[arrow data={0.52}{<}] (1,3.5) to [out=down, in=down, looseness=3] (2,3.5);
\end{tz}
&
\frac{1}{\sqrt{d}}~
\begin{tz}[zx,master,every to/.style={out=up, in=down},yscale=-1]
\draw[arrow data={0.5}{<}] (0,0) to [looseness=1.6] (1,3.5);
\draw[arrow data={0.5}{>}] (3,0) to [looseness=1.6] (2, 3.5);
\draw[arrow data={0.5}{>}] (1,0) to [out=up, in=up, looseness=3] (2,0);
\end{tz}
&
\sqrt{d}\!\!\!\!\!
\begin{tz}[zx,slave,every to/.style={out=up, in=down},yscale=-1]
\draw[arrow data={0.52}{>}] (1,3.5) to [out=down, in=down, looseness=3] (2,3.5);
\end{tz}
\end{calign}
\end{example}
\noindent
We now consider two equivalence relations on special Frobenius monoids in a dagger category with split idempotents. The stricter of the relations is \emph{unitary $*$-isomorphism}.
\paragraph{Unitary $*$-isomorphism.}
\begin{definition}
Let $A,B$ be Frobenius monoids in a monoidal dagger category. We say that a morphism $f: A \to B$ is a \emph{$*$-homomorphism} precisely when it satisfies the following equations (we use black vertices for the Frobenius monoid $B$):
\begin{calign}\label{eq:homo}
\begin{tz}[zx, master, every to/.style={out=up, in=down},yscale=-1]
\draw (0,0) to (0,2) to [out=135] (-0.75,3);
\draw (0,2) to [out=45] (0.75, 3);
\node[zxnode=\zxwhite] at (0,1) {$f$};
\node[zxvertex=\zxwhite, zxdown] at (0,2) {};
\end{tz}
=
\begin{tz}[zx, every to/.style={out=up, in=down},yscale=-1]
\draw (0,0) to (0,0.75) to [out=135] (-0.75,1.75) to (-0.75,3);
\draw (0,0.75) to [out=45] (0.75, 1.75) to +(0,1.25);
\node[zxnode=\zxwhite] at (-0.75,2) {$f$};
\node[zxnode=\zxwhite] at (0.75,2) {$f$};
\node[zxvertex=\zxblack, zxdown] at (0,0.75) {};
\end{tz}
&
\begin{tz}[zx,slave, every to/.style={out=up, in=down},yscale=-1]
\draw (0,0) to (0,2) ;
\node[zxnode=\zxwhite] at (0,1) {$f$};
\node[zxvertex=\zxwhite, zxup] at (0,2) {};
\end{tz}
=
\begin{tz}[zx,slave, every to/.style={out=up, in=down},yscale=-1]
\draw (0,0) to (0,0.75) ;
\node[zxvertex=\zxblack, zxup] at (0,0.75) {};
\end{tz}
&
\begin{tz}[zx,slave, every to/.style={out=up, in=down},scale=-1]
\draw (0,0) to (0,3);
\node[zxnode=\zxwhite] at (0,1.5) {$f^\dagger$};
\end{tz}
=~~
\begin{tz}[zx,slave,every to/.style={out=up, in=down},scale=-1,xscale=-1]
\draw (0,1.5) to (0,2) to [in=left] node[pos=1] (r){} (0.5,2.5) to [out=right, in=up] (1,2)  to [out=down, in=up] (1,0);
\draw (-1,3) to [out=down,in=up] (-1,1) to [out=down, in=left] node[pos=1] (l){} (-0.5,0.5) to [out=right, in=down] (0,1) to (0,1.5);
\node[zxnode=\zxwhite] at (0,1.5) {$f$};
\node[zxvertex=\zxblack] at (l.center){};
\node[zxvertex=\zxwhite] at (r.center){};
\end{tz}
\end{calign}
We call a unitary $*$-homomorphism a \emph{unitary $*$-isomorphism}. It is easy to see that a unitary $*$-isomorphism satisfies the following additional equations:
\begin{calign}\label{eq:cohomo}
\begin{tz}[zx, master, every to/.style={out=up, in=down}]
\draw (0,0) to (0,2) to [out=135] (-0.75,3);
\draw (0,2) to [out=45] (0.75, 3);
\node[zxnode=\zxwhite] at (0,1) {$f$};
\node[zxvertex=\zxblack, zxup] at (0,2) {};
\end{tz}
=
\begin{tz}[zx, every to/.style={out=up, in=down}]
\draw (0,0) to (0,0.75) to [out=135] (-0.75,1.75) to (-0.75,3);
\draw (0,0.75) to [out=45] (0.75, 1.75) to +(0,1.25);
\node[zxnode=\zxwhite] at (-0.75,2) {$f$};
\node[zxnode=\zxwhite] at (0.75,2) {$f$};
\node[zxvertex=\zxwhite, zxup] at (0,0.75) {};
\end{tz}
&
\begin{tz}[zx,slave, every to/.style={out=up, in=down}]
\draw (0,0) to (0,2) ;
\node[zxnode=\zxwhite] at (0,1) {$f$};
\node[zxvertex=\zxblack, zxup] at (0,2) {};
\end{tz}
=
\begin{tz}[zx,slave, every to/.style={out=up, in=down}]
\draw (0,0) to (0,0.75) ;
\node[zxvertex=\zxwhite, zxup] at (0,0.75) {};
\end{tz}
&
\begin{tz}[zx,slave, every to/.style={out=up, in=down}]
\draw (0,0) to (0,3);
\node[zxnode=\zxwhite] at (0,1.5) {$f^\dagger$};
\end{tz}
=~~
\begin{tz}[zx,slave,every to/.style={out=up, in=down},xscale=1]
\draw (0,1.5) to (0,2) to [in=left] node[pos=1] (r){} (0.5,2.5) to [out=right, in=up] (1,2)  to [out=down, in=up] (1,0);
\draw (-1,3) to [out=down,in=up] (-1,1) to [out=down, in=left] node[pos=1] (l){} (-0.5,0.5) to [out=right, in=down] (0,1) to (0,1.5);
\node[zxnode=\zxwhite] at (0,1.5) {$f$};
\node[zxvertex=\zxwhite] at (l.center){};
\node[zxvertex=\zxblack] at (r.center){};
\end{tz}
\end{calign}
\end{definition}
\paragraph{Morita equivalence.}
To define the weaker equivalence relation, we introduce the notion of a dagger bimodule. 
\def\d{0.5}
\def\h{2.25}
\def\inang{-45}
\begin{definition} \label{def:daggerbimodule}Let $A$ and $B$ be special Frobenius monoids in a monoidal dagger category. An $A{-}B$-\textit{dagger bimodule} is an object $M$ together with an morphism $\rho:A\otimes M\otimes B \to M$ fulfilling the following equations:
\begin{calign}\label{eq:bimodule}
\begin{tz}[zx,every to/.style={out=up, in=down}]
\draw (0,0) to (0,3);
\draw (-\d-1.5,0) to [in=-135] (-\d-0.75,1) to[in=180-\inang] (0,\h);
\draw (-\d,0) to [in=-45] (-\d-0.75,1) ;
\draw (\d,0) to [in=-135] (\d+0.75,1) to [in=\inang] (0,\h);
\draw (\d+1.5,0) to [in=-45] (\d+0.75,1) ;
\node[zxvertex=\zxwhite, zxdown] at (-\d-0.75,1){};
\node[zxvertex=\zxwhite, zxdown] at (\d+0.75,1){};
\node[box,zxdown] at (0,\h) {$\rho$};
\end{tz}
~~=~~
\def\htop{2.25}
\def\hbot{1.25}
\begin{tz}[zx,every to/.style={out=up, in=down}]
\draw (0,0) to (0,3);
\draw (-\d-1.5,0) to [in=-135] (0,\htop);
\draw (-\d,0) to [in=-135] (0,\hbot);
\draw (\d,0) to [in=-45] (0,\hbot);
\draw (\d+1.5,0) to [in=-45] (0,\htop);
\node[box,zxdown] at (0,\hbot) {$\rho$};
\node[box,zxdown] at (0,\htop) {$\rho$};
\end{tz}
&
\begin{tz}[zx, every to/.style={out=up, in=down}]
\draw (0,0) to (0,3);
\draw (-\d,1.2) to [in=-135] (0,\h);
\draw (\d,1.2) to [in=-45] (0,\h);
\node[box,zxdown] at (0,\h) {$\rho$};
\node[zxvertex=\zxwhite] at (-\d,1.2){};
\node[zxvertex=\zxwhite] at (\d,1.2){};
\end{tz}
~~=~~~
\begin{tz}[zx, every to/.style={out=up, in=down}]
\draw (0,0) to (0,3);
\end{tz}
&
\def\x{0.2}
\begin{tz}[zx, every to/.style={out=up, in=down},xscale=0.8]
\draw (0,0) to (0,3);
\draw (-\x,1.5) to [out=up, in=right] (-0.75-\x, 2.25) to [out=left, in=up] (-1.5-\x, 1.5) to (-1.5-\x,0);
\draw (\x,1.5) to [out=up, in=left] (0.75+\x, 2.25) to [out=right, in=up] (1.5+\x, 1.5) to (1.5+\x,0);
\node[zxvertex=\zxwhite] at (-0.75-\x, 2.25){};
\node[zxvertex=\zxwhite] at (0.75+\x, 2.25){};
\node[box] at (0,1.5) {$\rho^\dagger$};
\end{tz}
~~=~~
\begin{tz}[zx, every to/.style={out=up, in=down},xscale=0.8]
\draw (0,0) to (0,3);
\draw (-1.25, 0) to [in=-135] (0,1.95) ;
\draw (1.25,0) to [in=-45] (0,1.95);
\node[box] at (0,1.95) {$\rho$};
\end{tz}
\end{calign}
\end{definition}
\noindent
We usually denote an $A{-}B$-dagger bimodule $M$ by $_AM_B$.
\begin{definition} A \textit{morphism of dagger bimodules} $_AM_B\to {}_AN_B$ is a morphism $f:M\to N$ that commutes with the action of the Frobenius monoids:
\begin{calign}
\begin{tz}[zx]
\draw (0,0) to (0,3);
\draw (-1,0) to [out=up, in=-135] (0,2.15);
\draw (1,0) to [out=up, in=-45] (0,2.15) ;
\node[zxnode=\zxwhite] at (0,0.85) {$f$};
\node[boxvertex,zxdown] at (0,2.15){};
\end{tz}
=
\begin{tz}[zx]
\draw (0,0) to (0,3);
\draw (-1,0) to [out=up, in=-135] (0,0.85);
\draw (1,0) to [out=up, in=-45] (0,0.85) ;
\node[zxnode=\zxwhite] at (0,2.15) {$f$};
\node[boxvertex,zxdown] at (0,0.85){};
\end{tz}
\end{calign}
Two dagger bimodules are \textit{isomorphic}, here written ${}_AM_B\cong{}_AN_B$, if there is a unitary morphism of dagger bimodules ${}_AM_B\to{}_AN_B$.
\end{definition}
\noindent
In a monoidal dagger category in which dagger idempotents split (Definition~\ref{def:c*tensorcats}), we can compose dagger bimodules ${}_AM_B$ and ${}_BN_C$ to obtain an $A{-}C$-dagger bimodule ${}_AM{\otimes_B}N_C$ as follows. First we observe that the following endomorphism is a dagger idempotent:
\begin{calign}\label{eq:idempotentforrelprod}
\begin{tz}[zx,every to/.style={out=up, in=down}]
\draw (0,0) to (0,3);
\draw (2,0) to (2,3);
\draw (0,2.25) to [out=-45, in=left] (1, 1.2) to[out=right, in=-135] (2,2.25);
\draw (-1,1.2) to[out=90,in=-135] (0,2.25);
\draw (3,1.2) to[out=90,in=-45] (2,2.25);
\node[zxvertex=\zxwhite] at (1,1.2){};
\node[zxvertex=\zxwhite] at (-1,1.2){};
\node[zxvertex=\zxwhite] at (3,1.2){};
\node[boxvertex,zxdown] at (0,2.25){};
\node[boxvertex,zxdown] at (2,2.25){};
\node[dimension, left] at (0,0) {$M$};
\node[dimension, right] at (2,0) {$N$};
\end{tz}
\end{calign}
This may be seen as follows. For idempotency:
\begin{calign}
\includegraphics{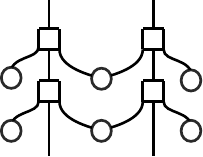}
~~=~~
\includegraphics{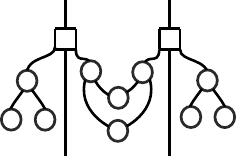}
~~=~~
\includegraphics{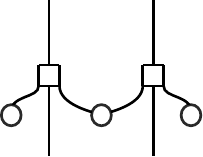}
\end{calign}
Here for the first equality we used~\eqref{eq:bimodule} and for the second we used speciality of $B$. For self-adjointness:
\begin{calign}
\includegraphics{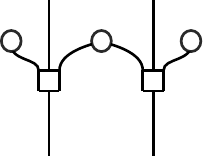}
~~=~~
\includegraphics{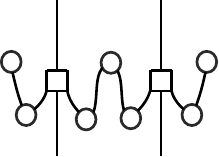}
~~=~~
\includegraphics{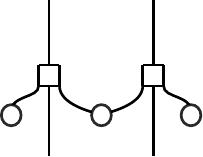}
\end{calign}
Here for the first equality we used the third equation of~\eqref{eq:bimodule}, and for the second equality we used the axioms of a Frobenius algebra.

Now the \emph{relative tensor product} ${}_AM{\otimes_B}N_C$ is defined as the splitting of the dagger idempotent~\eqref{eq:idempotentforrelprod}. We depict the isometry $i: M\otimes_B N\to M\otimes N$ as a downwards pointing triangle:
\begin{calign}\label{eq:moritaidempotentsplit}
\begin{tz}[zx,every to/.style={out=up, in=down}]
\draw (0,0) to (0,3);
\draw (2,0) to (2,3);
\draw (0,2.25) to [out=-45, in=left] (1, 1.2) to[out=right, in=-135] (2,2.25);
\draw (-1,1.2) to[out=90,in=-135] (0,2.25);
\draw (3,1.2) to[out=90,in=-45] (2,2.25);
\node[zxvertex=\zxwhite] at (1,1.2){};
\node[zxvertex=\zxwhite] at (-1,1.2) {};
\node[zxvertex=\zxwhite] at (3,1.2){};
\node[boxvertex,zxdown] at (0,2.25){};
\node[boxvertex,zxdown] at (2,2.25){};
\end{tz}
~~=~~
\begin{tz}[zx,every to/.style={out=up, in=down}]
\draw (0,0) to (0,0.5);
\draw (0,2.5) to (0,3);
\draw (2,0) to (2,0.5);
\draw (2,2.5) to (2,3);
\draw (1,0.5) to (1,2.5);
\node[dimension, right] at (1,1.5) {$M{\otimes_B}N$};
\node[triangleup=2] at (1,0.5){};
\node[triangledown=2] at (1,2.5){};
\end{tz}
&
\begin{tz}[zx]
\clip (-1.2, -0.3) rectangle (1.9,3.3);
\draw (0,0) to (0,1);
\draw (-1,1) to (-1,2);
\draw (1,1) to (1,2);
\draw (0, 2) to (0,3);
\node[triangleup=2] at (0,2){};
\node[triangledown=2] at (0,1){};
\node[dimension, right] at (0,0) {$M{\otimes_B}N$};
\end{tz}
=~~
\begin{tz}[zx]
\clip (-0.2, -0.3) rectangle (1.9,3.3);
\draw (0,0) to (0,3);
\node[dimension, right] at (0,0) {$M{\otimes_B}N$};
\end{tz}
\end{calign}
\ignore{One can convince onself that $f:M\otimes N\to M\otimes_B N$ is a coequalizer for the }
For dagger bimodules ${}_AM_B$ and ${}_BN_C$, the relative tensor product $M\otimes_B N$ is itself an $A{-}C$-dagger bimodule with the following action $A\otimes(M{\otimes_B}N) \otimes C\to M{\otimes_B}N$:
\begin{equation}
\begin{tz}[zx,every to/.style={out=up, in=down}]
\draw (0,-0.) to (0,1) ;
\draw (0,3) to (0,3.5);
\draw (-1,1) to (-1,2.5);
\draw (1,1) to (1,2.5);
\draw (-2,-0.) to [in=-135] (-1,1.75);
\draw (2,-0.) to [in=-45] (1,1.75);
\node[triangledown=2] at (0,1){};
\node[triangleup=2] at (0,2.5){};
\node[boxvertex] at (-1,1.75){};
\node[boxvertex] at (1,1.75){};
\end{tz}
\end{equation}
\noindent
\begin{definition} \label{def:daggermoritaequiv}Two Frobenius monoids $A$ and $B$ are \textit{Morita equivalent} if there are dagger bimodules $_AM_B$ and $_BN_A$ such that $_AM{\otimes_B}N_A\cong {}_AA_A$ and ${_BN{\otimes_A}M_B \cong {}_BB_B}$.
\end{definition}
\noindent
It may straightforwardly be verified that unitarily $*$-isomorphic Frobenius monoids are Morita equivalent.
\ignore{
\paragraph{Morita theory.}
We now see how these equivalence relations on Frobenius monoids $f \circ f^*$ relate to the corresponding morphisms $f: r \to s$. For the proof of the following theorems we refer to~\cite[Appendix]{Musto2019}.
\begin{definition}\label{def:splitdags}
We say that a 2-category has \emph{split dagger idempotents} if all of its Hom-categories have split dagger idempotents (Definition~\ref{def:semisimple}).
\end{definition}
\begin{theorem}\label{thm:morclassobjs}
Let $\mathcal{C}$ be a pivotal dagger 2-category with split dagger idempotents, and let $f:r \to s$ and $g: r \to t$ be 1-morphisms. The Frobenius monoids $f \circ f^*$ and $g \circ g^*$ are Morita equivalent iff there exists a dagger equivalence $X: s \to t$ in $\mathcal{C}$.
\end{theorem}
\begin{theorem}\label{thm:morclass1morphs}
Let $\mathcal{C}$ be a pivotal dagger 2-category with split dagger idempotents, and let $f:r \to s$ and $g: r \to t$ be 1-morphisms. 
The Frobenius monoids $f \circ f^*$ and $g \circ g^*$ are $*$-isomorphic iff there exists a dagger equivalence $X: t \to s$ and a unitary 2-morphism $\alpha: f \to g \circ X$. 
\end{theorem}
\begin{figure}
\end{figure}
}
\subsection{Morita classification of accessible fibre functors and UPTs}

We first observe that every UPT out of the canonical fibre functor $F$ gives rise to a special Frobenius monoid in $\End(F)$. 
\begin{proposition}\label{prop:upttofrobmon}
Let $G$ be a compact quantum group, let $F: \Rep(G) \to \Hilb$ be the canonical fibre functor, and let $F'$ be another fibre functor and $(\alpha,H): F \to F'$ a UPT. Then the object $\alpha^* \circ \alpha$ of $\End(F)$ has the structure of a special Frobenius monoid in $\End(F)$ with the following multiplication and unit modifications, where $d=\dim(H)$:
\begin{calign}
\frac{1}{\sqrt{d}}~
\includegraphics{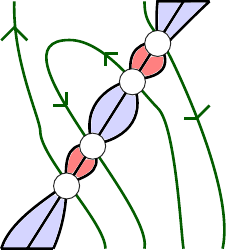}
~~=~~\frac{1}{\sqrt{d}}~
\includegraphics{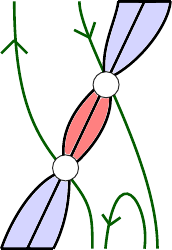}
\qquad&& \qquad
\sqrt{d}~\includegraphics{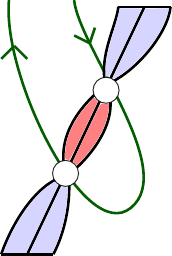}
~~=~~
\sqrt{d}~
\includegraphics{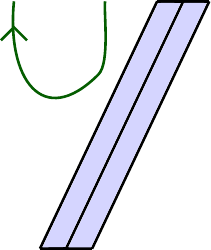}
\\
m: (\alpha \otimes \alpha^*) \otimes (\alpha \otimes \alpha^*) \to (\alpha \otimes \alpha^*) 
&&
u: \id_F \to (\alpha \otimes \alpha^*)
\end{calign}
\end{proposition}
\begin{proof}
That these are modifications as claimed follows from the pull-through equations for the cup and cap for the dual UPT~\eqref{eq:cupcapmodsdualpnt}. It is a special Frobenius monoid because the underlying algebra of the modifications is the normalised endomorphism algebra (Example~\ref{ex:normalisedendoalg}).
\end{proof}
\begin{definition}
The monoidal functor $\For: \End(F) \to \Hilb$ is defined as follows:
\begin{itemize}
\item \emph{On objects}: Every UPT $(\alpha,H): F \to F$ is taken to its underlying Hilbert space $H$.
\item \emph{On morphisms}: Every modification $f: (\alpha_1,H_1) \to (\alpha_2,H_2)$ is taken to its underlying linear map $f: H_1 \to H_2$.
\end{itemize}
\end{definition}
\begin{remark}
By the isomorphism $\End(F) \cong \Rep(A_G)$ of Theorem~\ref{thm:higherhopfgal}, $\For$ is precisely the canonical fibre functor on the category $\Rep(A_G)$.
\end{remark}
\ignore{
We now observe the following fact, recalling the definition of the endomorphism algebra in $\Hilb$ (Definition~\ref{def:endoalgebra}).
\begin{lemma}
The Frobenius monoid $\alpha^* \circ \alpha$ of Proposition~\eqref{} is mapped to an endomorphism algebra in $\Hilb$ by the functor $\For$.
\end{lemma}}
\begin{definition}\label{def:simplefrobmon}
A Frobenius monoid $A=((\alpha,H),m,u)$ in $\End(F)$ is \emph{simple} if $\For(A)$ is unitarily $*$-isomorphic to a normalised endomorphism algebra in $\Hilb$.
\end{definition}
\noindent
Every simple Frobenius monoid is in particular special, since it is $*$-isomorphic to a special Frobenius monoid.

\noindent
Every fibre functor $F'$ and UPT $F \to F'$ gives rise to a simple Frobenius monoid in $\End(F)$ by the construction of Proposition~\ref{prop:upttofrobmon}.
We now give a construction in the other direction --- any simple Frobenius monoid in $\End(F)\cong \Rep(A_G)$ can be `split' to obtain a fibre functor $F'$ and a UPT $F \to F'$. First observe that we may conjugate any simple Frobenius monoid $((\tilde{\alpha},\tilde{H}),\tilde{m},\tilde{u})$ in $\End(F)$ by the $*$-isomorphism of Definition~\ref{def:simplefrobmon} to obtain a simple Frobenius monoid $((\alpha,H\otimes H^*),m,u)$ where the modifications $m,u$ have the standard form:
\begin{calign}\frac{1}{\sqrt{d}}
\includegraphics[scale=.8]{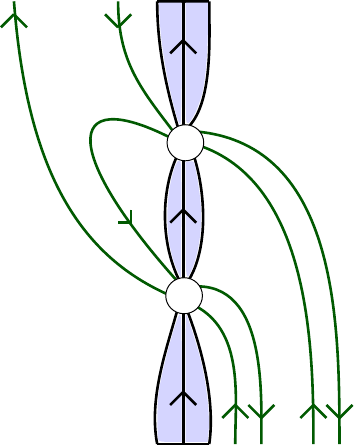}
~~=~~
\frac{1}{\sqrt{d}}
\includegraphics[scale=.8]{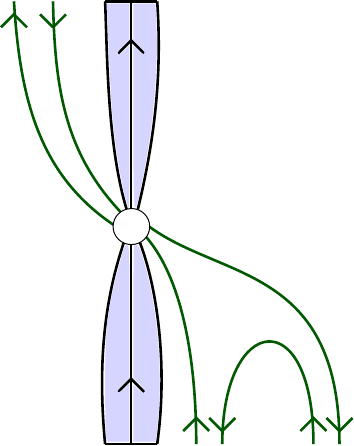}
&~~
\sqrt{d}
\includegraphics[scale=.8]{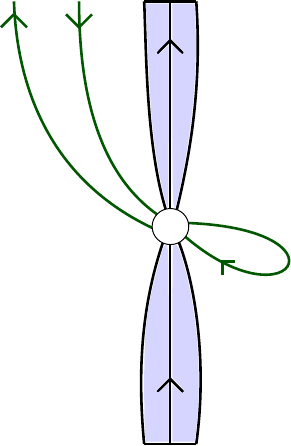}
~~=~~
\sqrt{d}
\includegraphics[scale=.8]{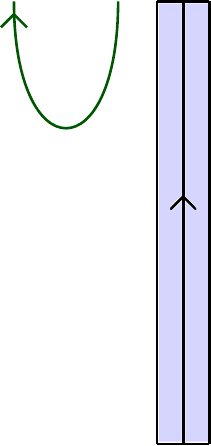}
\\
m: \alpha \circ \alpha \to \alpha
&
u: \id_F \to \alpha
\end{calign}
\begin{lemma}
For a simple Frobenius monoid $((\alpha,H \otimes H^*),m,u)$ in $\End(F)$ and any object $V$ of $\Rep(G)$, the following defines a dagger idempotent on $H^* \otimes F(V) \otimes H$:
\begin{calign}
{\large \frac{1}{d}}~
\includegraphics[scale=.8]{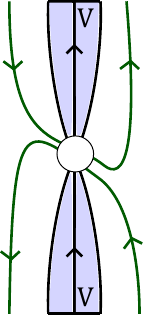}
\end{calign}
\end{lemma}
\begin{proof}
Idempotency follows from the fact that the algebra multiplication is a modification:
\begin{calign}
{\large \frac{1}{d^2}}~
\includegraphics[scale=.8]{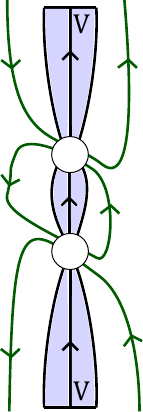}
~~=~~
{\large \frac{1}{d^2}}~
\includegraphics[scale=.8]{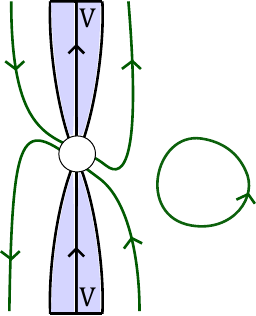}
~~=~~
{\large \frac{1}{d}}~
\includegraphics[scale=.8]{Figures/svg/morita/idempotdef.pdf}
\end{calign}
We must now show that the idempotent is self-adjoint. Recall that a Frobenius algebra has a self-duality with the cup and cap~\eqref{eq:cupcapfrob}. By Proposition~\ref{prop:relateduals} applied in $\End(F)$, there is therefore an invertible modification $P: (\alpha,H \otimes H^*) \to (\alpha^*,(H \otimes H^*)^*)$ satisfying the following equations in $\Hilb$:
\begin{calign}\label{eq:Pmod}
\includegraphics[scale=.8]{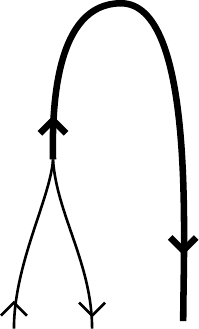}
~~=~~
\includegraphics[scale=.8]{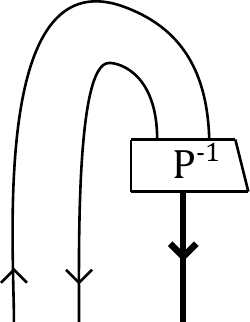}
&
\includegraphics[scale=.8]{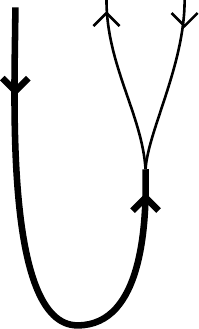}
~~=~~
\includegraphics[scale=.8]{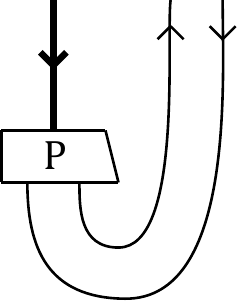}
\end{calign}
Here we drew the chosen right dual of $(H \otimes H^*)$ using a thick wire and a thick downward arrow in the spirit of~\eqref{eq:dualscompare}. Now we show that the idempotent is self-adjoint:
\begin{calign}
{\large \frac{1}{d}}~
\includegraphics[scale=.8]{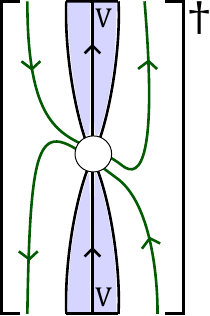}
~~=~~
{\large \frac{1}{d}}~
\includegraphics[scale=.8]{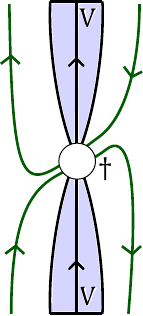}
~~=~~
{\large \frac{1}{d}}~
\includegraphics[scale=.8]{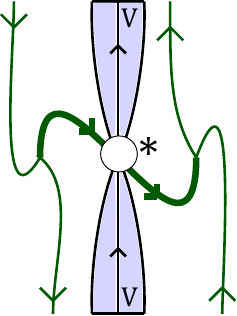}
~~=~~
{\large \frac{1}{d}}~
\includegraphics[scale=.8]{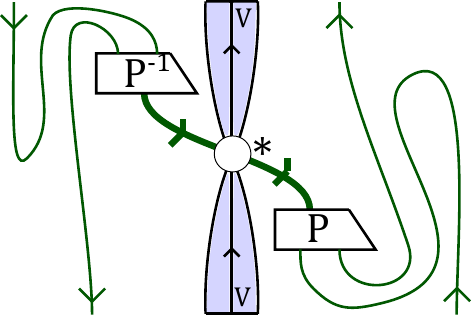}
~~=~~
{\large \frac{1}{d}}~
\includegraphics[scale=.8]{Figures/svg/morita/idempotdef.pdf}
\end{calign}
Here for the first equality we used the graphical calculus of the dagger; for the second equality we used unitarity of $\alpha$~\eqref{eq:dualpnt}; for the third equality we used~\eqref{eq:Pmod}; and for the fourth equality we used the fact that $P$ is a modification $\alpha \to \alpha^*$ to cancel $P$ with its inverse.  
\end{proof}
\noindent
This dagger idempotent splits to give a new Hilbert space, which, foreshadowing Theorem~\ref{thm:splitting}, we call $F_{\alpha}(V)$ and draw as $V$ surrounded by a red box, and an isometry $\iota_V: F_{\alpha}(V) \to H^* \otimes F(V) \otimes H$ satisfying the following equations:
\begin{calign}\label{eq:splitisom}
\includegraphics[scale=.8]{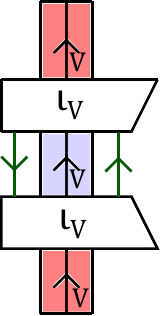}
~~=~~
\includegraphics[scale=.8]{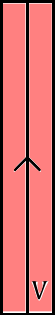}
&
\includegraphics[scale=.8]{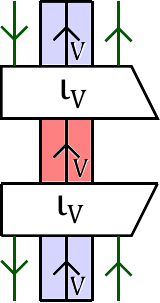}
~~=~~
{\large \frac{1}{d}}~\includegraphics[scale=.8]{Figures/svg/morita/idempotdef.pdf}
\end{calign}
\begin{theorem}\label{thm:splitting}
Let $((\alpha,H\otimes H^*),m,u)$ be a simple Frobenius monoid in $\End(F) \cong \Rep(A_G)$, and let $d = \dim(H)$. For every representation $V$ of $G$, let $F_{\alpha}(V)$ and $\iota_V: F_{\alpha}(V) \to H^* \otimes V \otimes H$ be the Hilbert space and isometry defined in the foregoing discussion. 

Then the following defines a fibre functor $F_{\alpha}: \Rep(G) \to \Hilb$:
\begin{itemize}
\item On objects: $V \mapsto F_{\alpha}(V)$.
\item On morphisms: 
\begin{calign}\label{eq:splitmorphs}
\includegraphics[scale=.8]{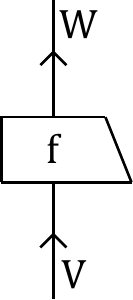}
~~\mapsto~~
\includegraphics[scale=.8]{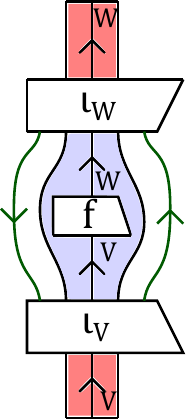}
\end{calign} 
\item Monoidal structure:
\begin{calign}
\sqrt{d}~
\includegraphics[scale=.8]{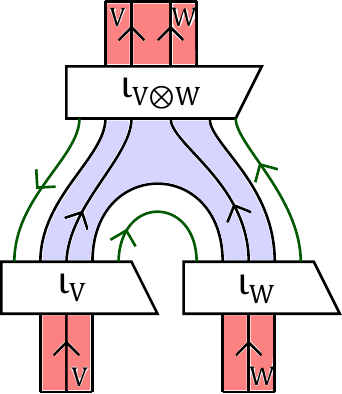}&
\frac{1}{\sqrt{d}}~
\includegraphics[scale=.8]{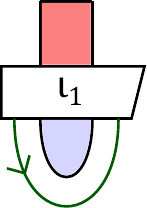}\\
m_{V,W}: F_{\alpha}(V) \otimes F_{\alpha}(W) \to F_{\alpha}(V \otimes W)
&
u: \mathbb{C} \to F_{\alpha}(\mathbbm{1})
\end{calign}
\end{itemize}
Moreover, there is a UPT $(\sqrt{\alpha},H): F \to F_{\alpha}$ with the following components $\sqrt{\alpha}_V$:
\begin{calign}
\sqrt{d}~\includegraphics[scale=.8]{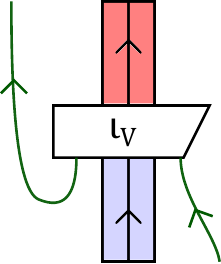}
\end{calign}
This UPT `splits' $A$ in the sense that $\sqrt{\alpha}^* \circ \sqrt{\alpha} = \alpha$.
\end{theorem}
\begin{proof}
We first show that $F_{\alpha}$ is a fibre functor. Compositionality is seen as follows:
\begin{calign}
\includegraphics[scale=.8]{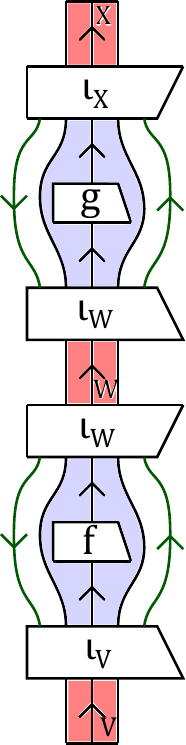}
~~=~~
\frac{1}{d}
\includegraphics[scale=.8]{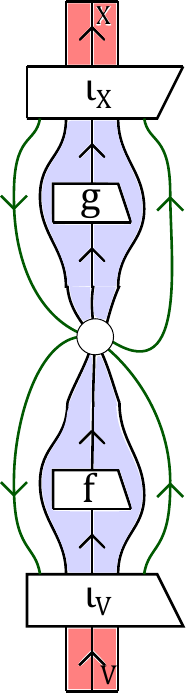}
~~=~~
\frac{1}{d}
\includegraphics[scale=.8]{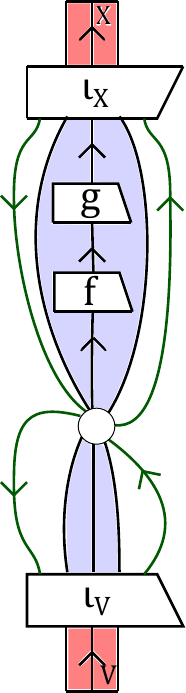}
~~=~~
\includegraphics[scale=.8]{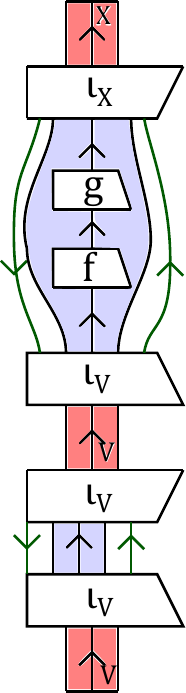}
~~=~~
\includegraphics[scale=.8]{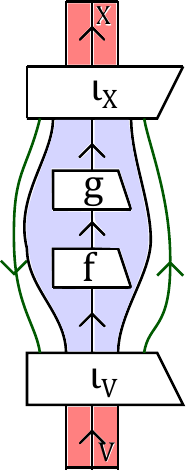}
\end{calign}
Here the first equality is by the second equation of~\eqref{eq:splitisom}; the second equality is naturality of the UPT $(\alpha,H \otimes H^*):F \to F$; the third equality is by the second equation of~\eqref{eq:splitisom}; and the fourth equality is by the first equation of~\eqref{eq:splitisom}.

$G$ clearly takes identity morphisms to identity morphisms since $\iota$ is an isometry, and the functor preserves the dagger by symmetry of~\eqref{eq:splitmorphs} in a horizontal axis. We therefore already have a unitary functor. For monoidality, we must check that $\{m_{V,W}\}$ and $u$ are unitary and that they obey the associativity and unitality equations~\eqref{eq:psfctassoc} and~\eqref{eq:psfctunital}. 
\begin{itemize}
\item \emph{Unitarity of $\{m_{V,W}\}$.} We prove the first equation of unitarity:
\begin{calign}
d~
\includegraphics[scale=.8]{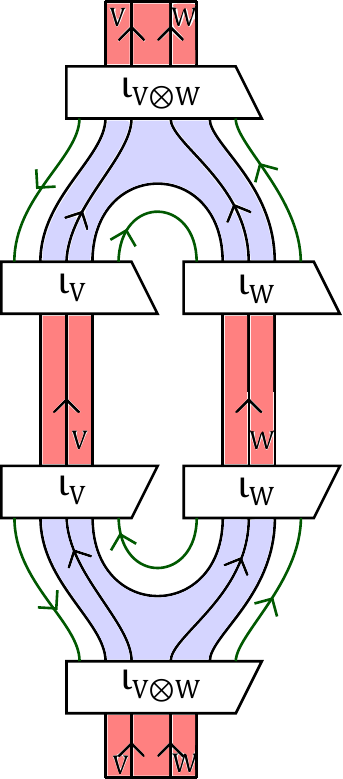}
~~=~~
\frac{1}{d}~
\includegraphics[scale=.8]{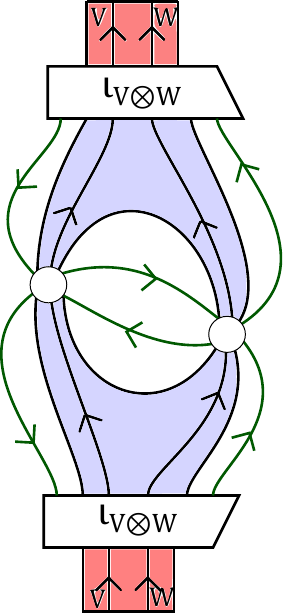}
~~=~~
\frac{1}{d}
\includegraphics[scale=.8]{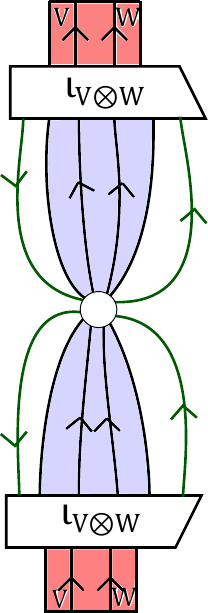}
~~=~~
\includegraphics[scale=.8]{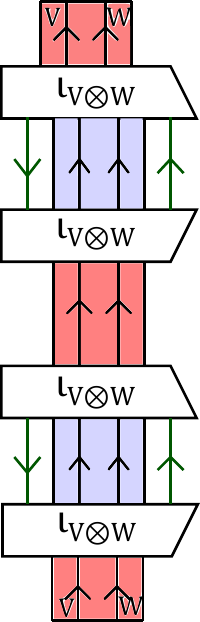}
~~=~~
\includegraphics[scale=.8]{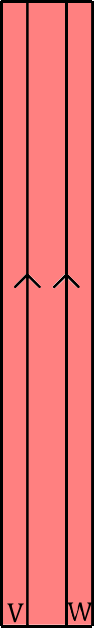}
\end{calign}
Here for the first equality we used the second equation of~\eqref{eq:splitisom}; for the second equality we used monoidality of $\alpha$; for the third equation we used the second equation of~\eqref{eq:splitisom}; and for the fourth equality we used the first equation of~\eqref{eq:splitisom}.

For the other unitarity equation:
\begin{calign}
d~
\includegraphics[scale=.8]{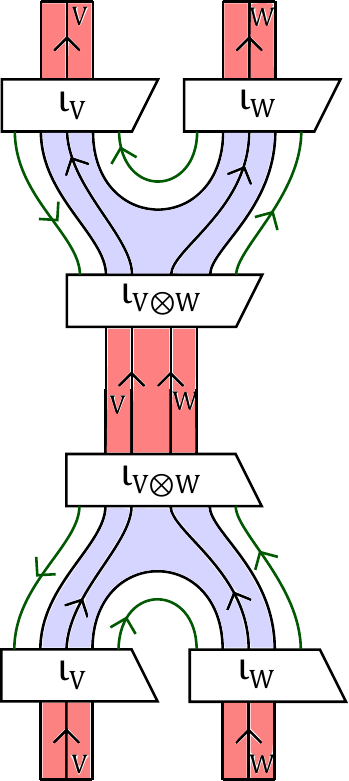}
~~=~~
d~
\includegraphics[scale=.8]{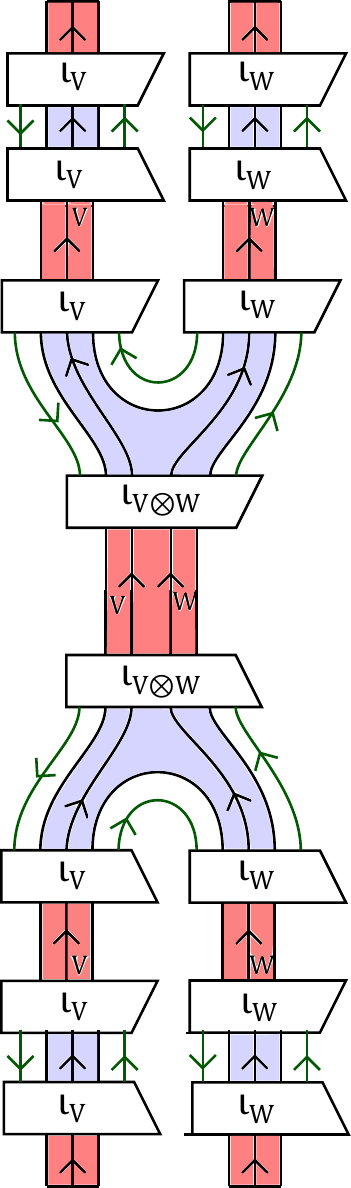}
~~=~~
\frac{1}{d^4}~
\includegraphics[scale=.8]{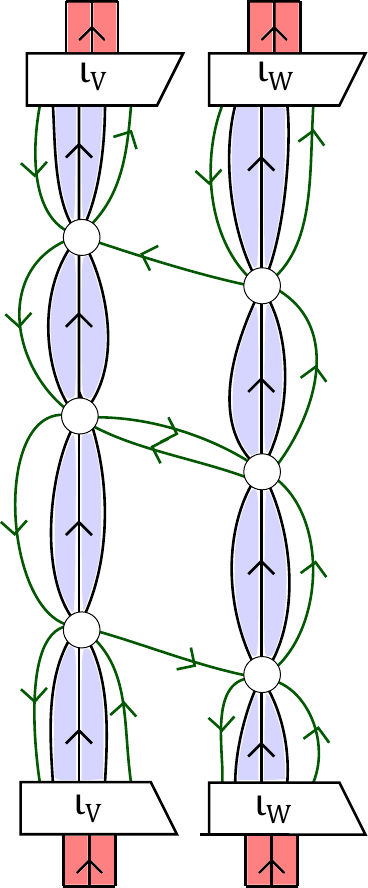}\\
~~=~~
\frac{1}{d^4}~
\includegraphics[scale=.8]{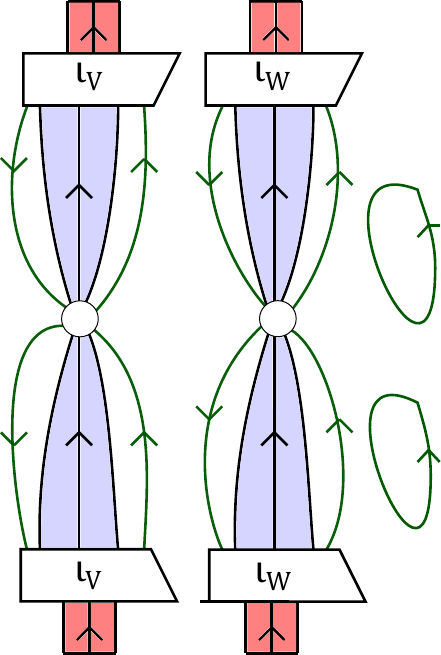}
~~=~~
\includegraphics[scale=.8]{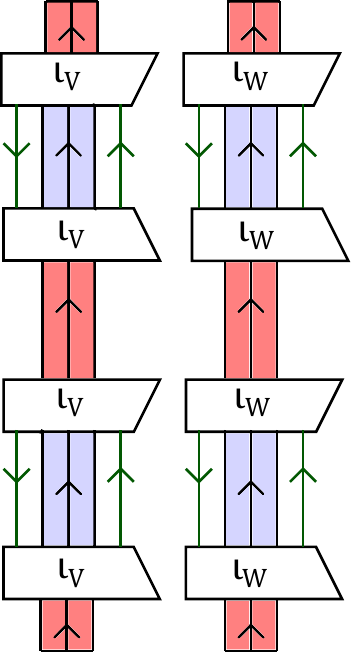}
~~=~~
\includegraphics[scale=.8]{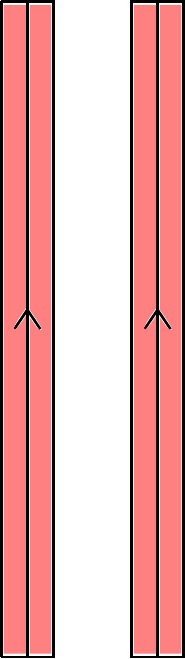}
\end{calign}
Here for the first equality we used the first equation of~\eqref{eq:splitisom}; for the second equality we used the second equation of~\eqref{eq:splitisom} and monoidality of $\alpha$; for the third equality we used the fact that the multiplication of the Frobenius algebra is a modification; for the fourth equality we evaluated the loops and used the second equation of~\eqref{eq:splitisom}; for the fifth equality we used the first equation of~\eqref{eq:splitisom}.
\item \emph{Unitarity of $u$}. For the first equation:
\begin{calign}
\frac{1}{d}
\includegraphics[scale=.8]{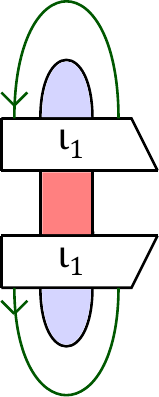}
~~=~~
\frac{1}{d^2}
\includegraphics[scale=.8]{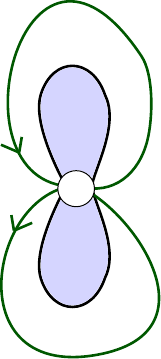}
~~=~~
\frac{1}{d^2}
\includegraphics[scale=.8]{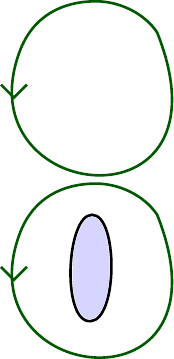}
~~=~~
\end{calign}
Here for the first equality we used the second equation of~\eqref{eq:splitisom}; for the second equality we used monoidality of $\alpha$~\eqref{eq:pntmonmonunit}; and for the third equality we evaluated the loops and used unitarity of $F$. 

For the second equation:
\begin{calign}
\frac{1}{d}
\includegraphics[scale=.8]{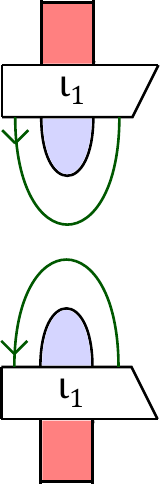}
~~=~~
\frac{1}{d}
\includegraphics[scale=.8]{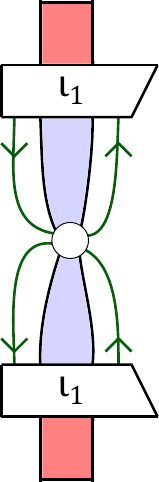}
~~=~~
\includegraphics[scale=.8]{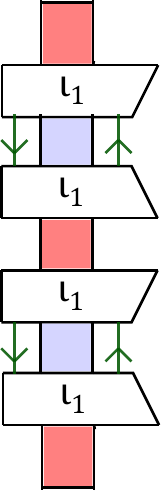}
~~=~~
\includegraphics[scale=.8]{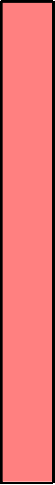}
\end{calign}
For the first equality we used monoidality of $\alpha$~\eqref{eq:pntmonmonunit}; for the second equality we used the second equation of~\eqref{eq:splitisom}; and for the third equality we used the first equation of~\eqref{eq:splitisom}.
\item \emph{Associativity~\eqref{eq:psfctassoc}}. We have the following sequence of equations:
\begin{calign}
d~
\includegraphics[scale=.8]{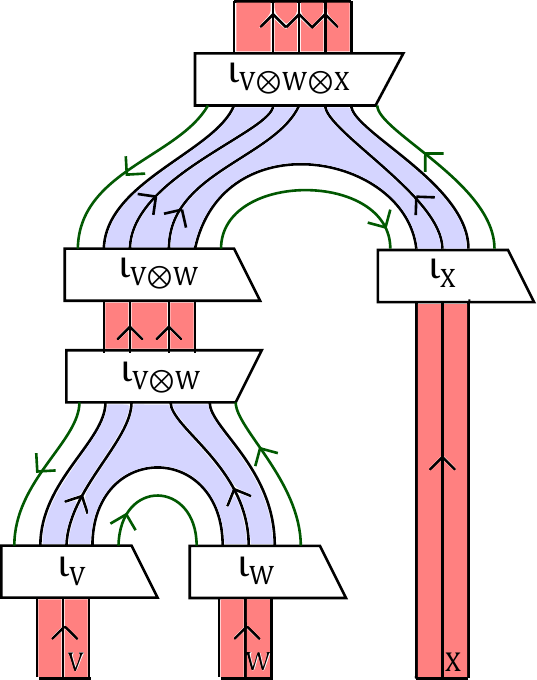}
~~=~~
\includegraphics[scale=.8]{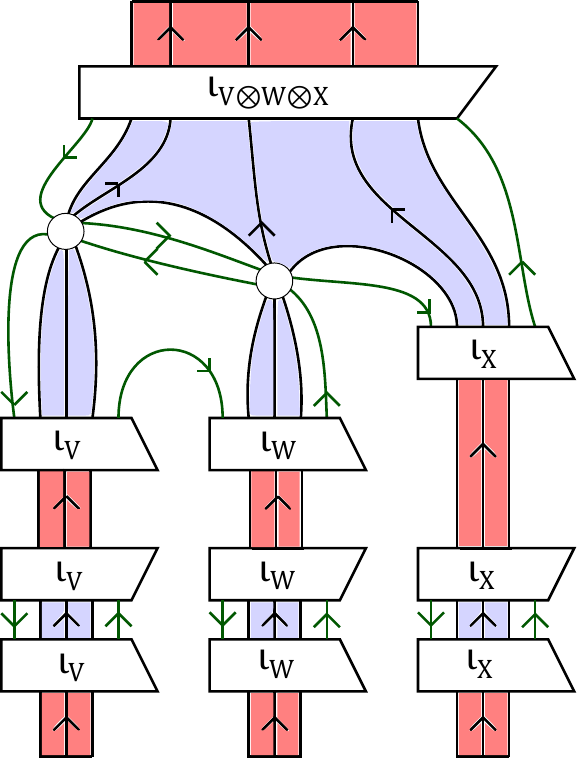}\\
=~~
\frac{1}{d^3}~
\includegraphics[scale=.8]{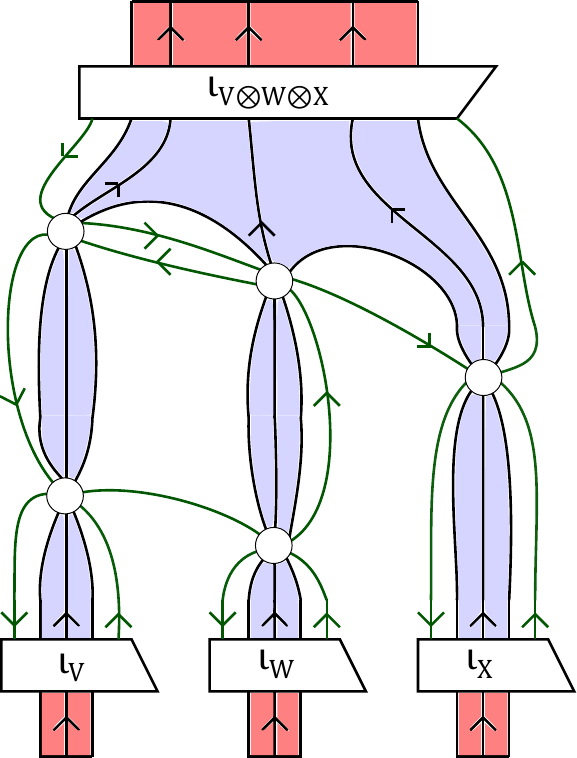}
~~=~~
\frac{1}{d^2}
\includegraphics[scale=.8]{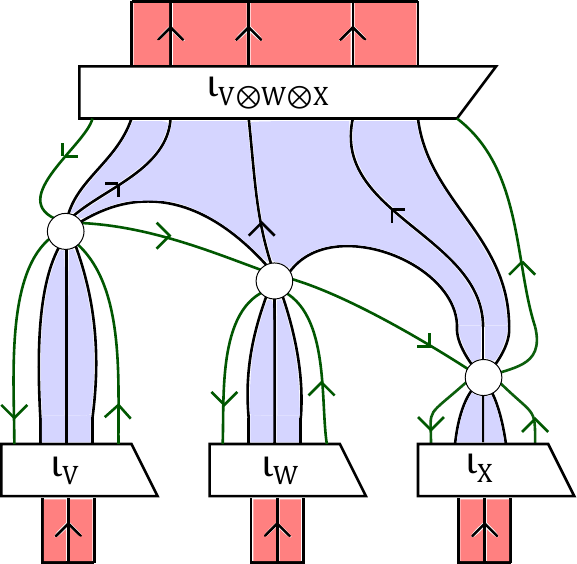}
\end{calign}
Here for the first equality we used the first equation of~\eqref{eq:splitisom} to insert isometries and their inverses on all three legs, and the second equation of~\eqref{eq:splitisom} and monoidality of $\alpha$. For the second equality we used the second equation of~\eqref{eq:splitisom}. For the third equality we used the fact that the multiplication of the Frobenius algebra is a modification and evaluated the resulting loop. 

We leave the rest of the proof to the reader: use monoidality of $\alpha$ on the two rightmost legs, use the second equation of~\eqref{eq:splitisom} to eliminate all occurences of $\alpha$, then cancel isometries using the first equation of~\eqref{eq:splitisom}.
\item \emph{Unitality~\eqref{eq:psfctunital}}. The left unitality equation is shown as follows:
\begin{calign}
\includegraphics[scale=.8]{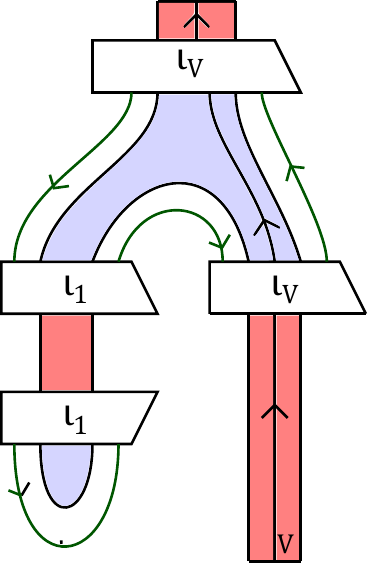}
~~=~~
\frac{1}{d}
\includegraphics[scale=.8]{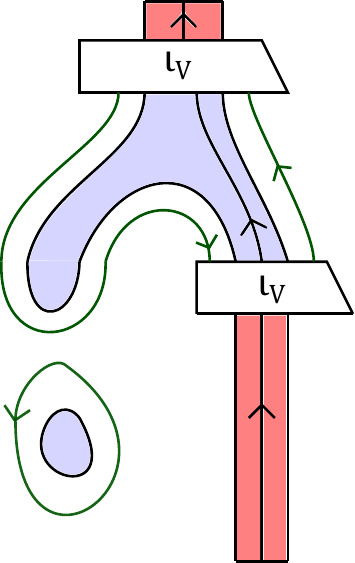}
~~=~~
\includegraphics[scale=.8]{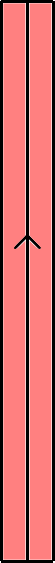}
\end{calign}
For the first equation we used the second equation of~\eqref{eq:splitisom} and monoidality of $\alpha$; for the second equation we evaluated the loop and used unitarity of $F_{\alpha}$.
The right unitality equation is shown similarly.
\end{itemize}
We have therefore shown that $F_{\alpha}$ is a fibre functor on $\Rep(G)$. We must now show that $\sqrt{\alpha}$ is a UPT $F \to F_{\alpha}$. For this we must show naturality and monoidality (\ref{eq:pntmonnat}-\ref{eq:pntmonmonunit}).
\begin{itemize}
\item \emph{Naturality}. For any $f: V \to W$:
\begin{calign}
\includegraphics[scale=.8]{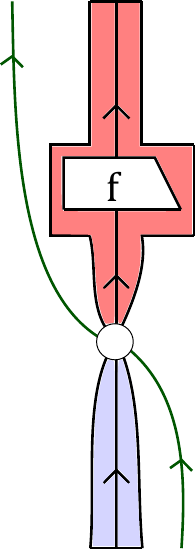}
~~=~~
\sqrt{d}~
\includegraphics[scale=.8]{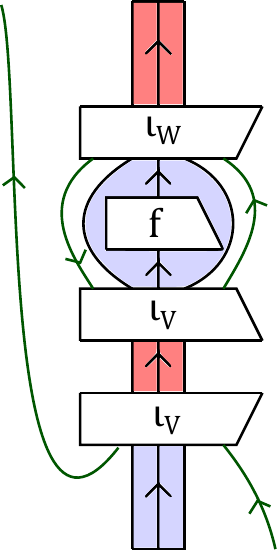}
~~=~~
\sqrt{d}~
\includegraphics[scale=.8]{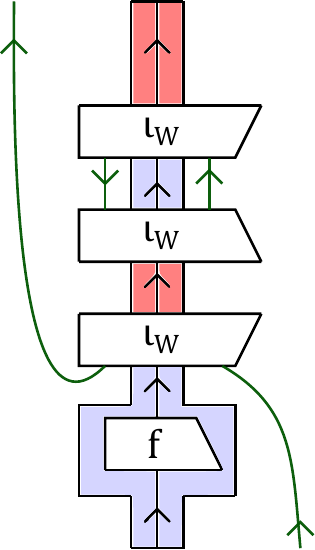}
~~=~~
\sqrt{d}~
\includegraphics[scale=.8]{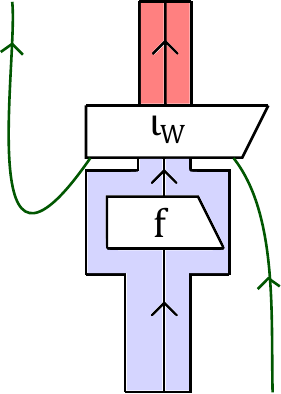}
~~=~~
\includegraphics[scale=.8]{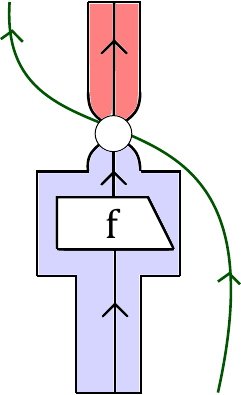}
\end{calign}
Here the first and fourth equalities are by definition; the second equality is by the second equation of~\eqref{eq:splitisom} and naturality of $\alpha$; and the third equality is by the first equation of~\eqref{eq:splitisom}.
\item \emph{Monoidality}. 
\begin{itemize}
\item We show~\eqref{eq:pntmonmon}:
\begin{calign}
\includegraphics[scale=.8]{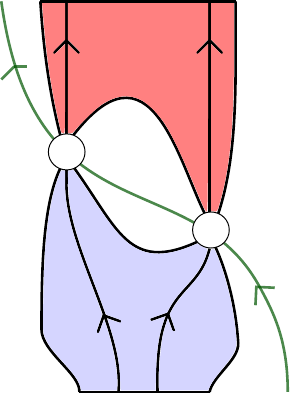}
~~=~~
d^{3/2}~
\includegraphics[scale=.8]{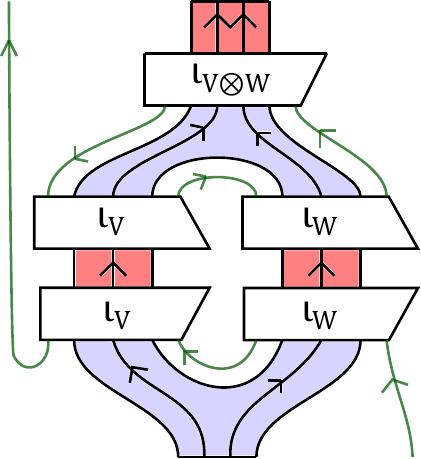}
~~=~~
\frac{1}{\sqrt{d}}~
\includegraphics[scale=.8]{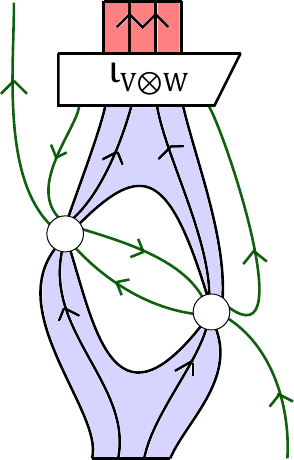}
~~=~~
\frac{1}{\sqrt{d}}~
\includegraphics[scale=.8]{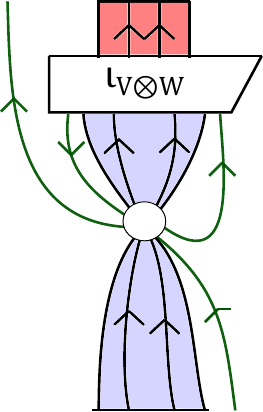}\\
~~=~~
\sqrt{d}~
\includegraphics[scale=.8]{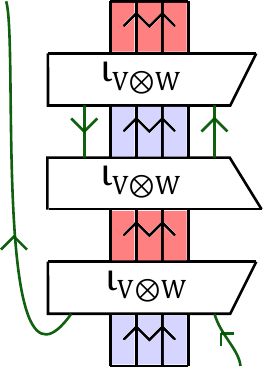}
~~=~~
\includegraphics[scale=.8]{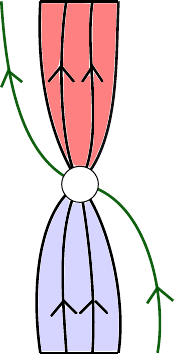}
\end{calign}
Here the first equality is by definition of $\sqrt{\alpha}$ and the multiplicator of $F_{\alpha}$; the second equality is by the second equation of~\eqref{eq:splitisom}; the third equality is by monoidality of $\alpha$; the fourth equality is by the second equation of~\eqref{eq:splitisom}; and the fifth equality is by the first equation of~\eqref{eq:splitisom} and the definition of $\sqrt{\alpha}$.
\item We show~\eqref{eq:pntmonmonunit}:
\begin{calign}
\includegraphics[scale=.8]{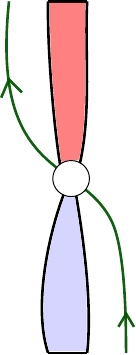}
~~=~~
\sqrt{d}
\includegraphics[scale=.8]{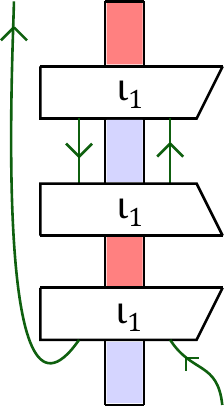}
~~=~~
\frac{1}{\sqrt{d}}
\includegraphics[scale=.8]{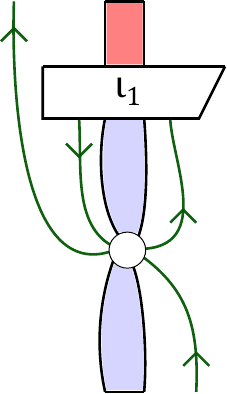}
~~=~~
\frac{1}{\sqrt{d}}
\includegraphics[scale=.8]{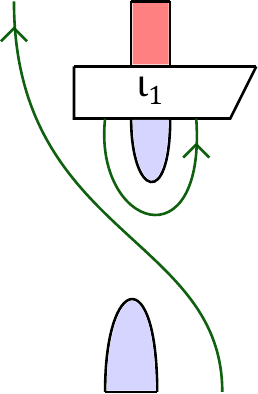}
~~=~~
\includegraphics[scale=.8]{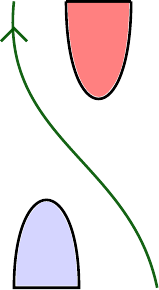}
\end{calign}
Here the first equality is by definition of $\sqrt{\alpha}$ and by the first equation of~\eqref{eq:splitisom}; the second equality is by the second equation of~\eqref{eq:splitisom}; the third equality is by monoidality of $\alpha$; and the fourth equality is by definition of the unitor of $F_{\alpha}$.
\end{itemize}
\item \emph{Unitarity}. We show the first equation of unitarity:
\begin{calign}
d~
\includegraphics[scale=.8]{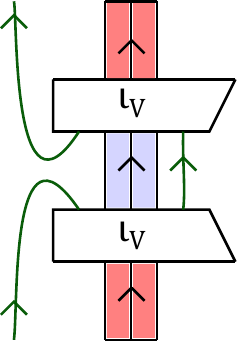}
~~=~~
d~
\includegraphics[scale=.8]{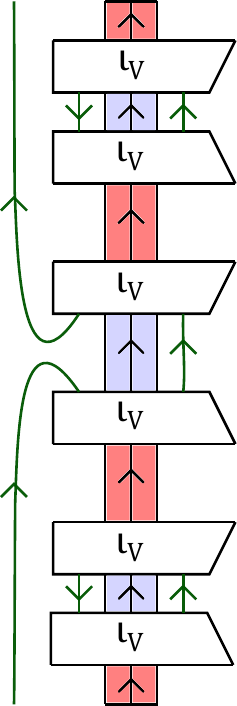}
~~=~~\frac{1}{d}~
\includegraphics[scale=.8]{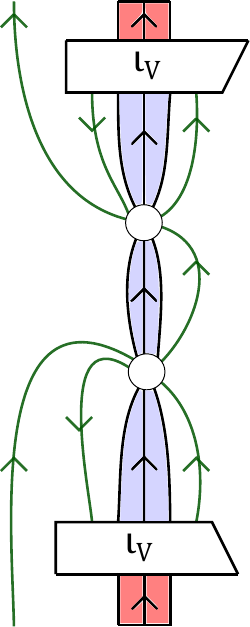}\\
=~~
\frac{1}{d}~
\includegraphics[scale=.8]{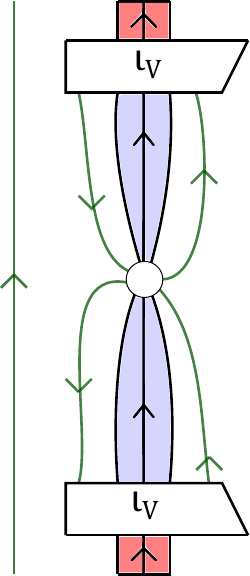}
~~=~~
\includegraphics[scale=.8]{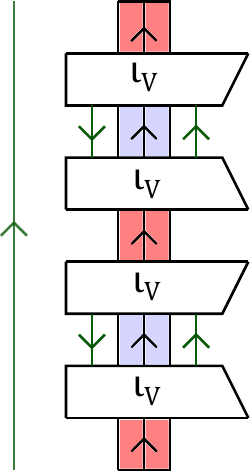}
~~=~~
\includegraphics[scale=.8]{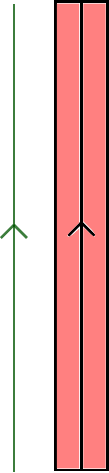}
\end{calign}
Here for the first equality we used the first equation of~\eqref{eq:splitisom}; for the second equality we used the second equation of~\eqref{eq:splitisom}; for the third equality we used the fact that the comultiplication of $((\alpha,H\otimes H^*),m,u)$ is a modification; for the fourth equality we used the second equation of~\eqref{eq:splitisom}; and for the final equality we used the first equation of~\eqref{eq:splitisom}.

We show the second equation of unitarity:
\begin{calign}
d~
\includegraphics[scale=.8]{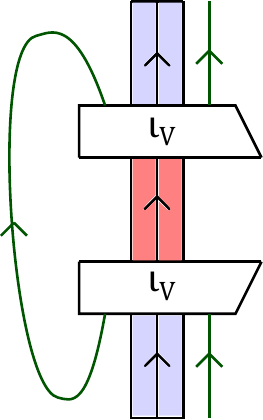}
~~=~~
\includegraphics[scale=.8]{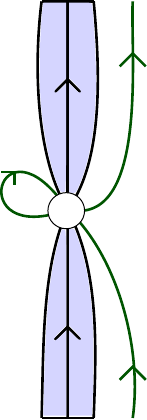}
~~=~~
\includegraphics[scale=.8]{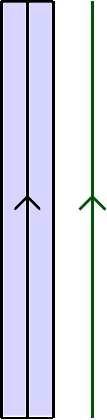}
\end{calign}
For the first equality we used the second equation of~\eqref{eq:splitisom}, and for the second equality we used that the counit of $((\alpha,H\otimes H^*),m,u)$ is a modification.
\end{itemize}
We have therefore shown that $\alpha:G\to F$ is a UPT. Finally, we show that $\sqrt{\alpha}$ splits $\alpha$, i.e. $\sqrt{\alpha} \circ \sqrt{\alpha}^* = \alpha$. By unitarity of $\sqrt{\alpha}$ it is equivalent to show that $\sqrt{\alpha} \circ \sqrt{\alpha}^{\dagger} = A$, which follows immediately from the second equation of~\eqref{eq:splitisom}:
\begin{calign}d~
\includegraphics[scale=.8]{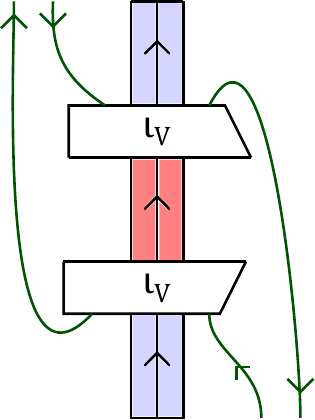}
~~=~~
\includegraphics[scale=.8]{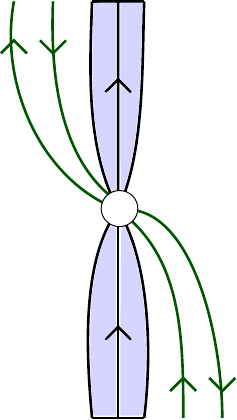}
\end{calign}
\end{proof}

\paragraph{Classification of UPTs and accessible fibre functors.}
We have seen that every UPT $\alpha: F \to F'$, where $F'$ is some fibre functor accessible from $F$, yields a simple Frobenius algebra in $\End(F)$, and that every unitary $*$-isomorphism class of simple Frobenius algebras is obtained in this way. By means of this correspondence we can classify UPTs and accessible fibre functors. 

\begin{definition}\label{def:equivupts}
Let $\alpha_1: F \to F_1$ and $\alpha_2: F \to F_2$ be UPTs. We say that $\alpha_1$ and $\alpha_2$ are \emph{equivalent} if there exists a a unitary monoidal natural isomorphism $E: F_2 \to F_1$ and a unitary modification $\tau: \alpha_1 \to \alpha_2 \otimes E$.
\end{definition}
\begin{theorem}\label{thm:moritaclass}
Let $F$ be the canonical fibre functor $\Rep(G) \to \Hilb$. Then the constructions of Proposition~\ref{prop:upttofrobmon} and Theorem~\ref{thm:splitting} give an explicit bijective correspondence between:
\begin{itemize}
\item Unitary monoidal isomorphism classes of unitary fibre functors accessible from $F$ by a UPT; and Morita equivalence classes of simple Frobenius monoids in $\Rep(A_G)$.
\item Equivalence classes of UPTs $\alpha: F \to F'$ for some accessible fibre functor $F'$; and unitary $*$-isomorphism classes of simple Frobenius monoids in $\Rep(A_G)$.
\end{itemize}
\end{theorem}
\begin{proof}
A full proof is given in~\cite[\S{}6]{Verdon2020a} (the relevant theorems are~\cite[Thm. 6.11, Thm 6.19]{Verdon2020a}). The result is an application of Morita theory in the 2-category $\Fun(\Rep(G),\Hilb)$. We have just shown that the special Frobenius monoids in $\End(F)$ which split --- i.e. which are isomorphic to `pair of pants' algebras $\alpha \otimes \alpha^*$ for 1-morphisms $\alpha: F \to F'$ --- are precisely the simple dagger Frobenius monoids. Morita theory implies a bijective correspondence between Morita equivalence classes of split special Frobenius monoids in $\End(F)$ and equivalence classes of objects $F'$ in $\Fun(\Rep(G),\Hilb)$ such that there exists a 1-morphism $\alpha: F \to F'$; equivalence of objects in $\Fun(\Rep(G),\Hilb)$ corresponds to unitary monoidal natural isomorphism of fibre functors, which gives the first bullet point. Finally, Morita theory implies a bijective correspondence between unitary $*$-isomorphism classes of simple Frobenius monoids in $\Rep(A_G)$, and equivalence classes of 1-morphisms $\alpha: F \to F'$ in $\Fun(\Rep(G),\Hilb)$. Equivalence of 1-morphisms in this setting comes down to Definition~\ref{def:equivupts}, whence the second bullet point.
\end{proof}
\ignore{
\begin{remark}
In the case where $A_G$ is finite-dimensional, and all fibre functors are therefore accessible by a UPT (Corollary~\ref{}) we recover a unitary version of the classification of module categories in~\cite{}, in the special case of rank one module categories. 
\end{remark}
}

\begin{example}
One application of this result is to construct fibre functors. For example, it is straightforward to use Theorem~\ref{thm:moritaclass} to obtain a concrete construction of all fibre functors on the category $\Rep(G)$ for a finite group $G$, since by Corollary~\ref{cor:finiteacc} they are all accessible by a UPT. It is well-known that they correspond to conjugacy classes of pairs $(H,\psi)$, where $H<G$ is a subgroup of \emph{central type} and $\psi: H \times H \to U(1)$ is a \emph{nondegenerate 2-cocycle}. However, to our knowledge the only known way to construct them was to apply Construction~\ref{constr:galtofib}, which is not quite concrete enough for some applications. Our construction, which uses only very basic representation theory, is stated in~\cite[\S{}6.1]{Verdon2020b}.
\end{example}

\section{Quantum graphs and their isomorphisms}\label{sec:qgraphiso}
In this Section we give an example of UPTs arising in the study of finite quantum graph theory~\cite{Musto2018}. Specifically, we will show that finite-dimensional quantum graph isomorphisms from a quantum graph $X$ are UPTs from the canonical fibre functor on the category of representations of its quantum automorphism group. \ignore{In Section~\ref{sec:cmqgupts} we will give a  simplified definition of UPTs and their modifications in the case where $G$ is a compact matrix quantum group. In Section~\ref{sec:qgraphisossubsec} we will use this result to obtain the desired characterisation.}

\subsection{UPTs for compact matrix quantum groups}
\label{sec:cmqgupts}
\begin{definition}
We say that a $C^*$-tensor category $\mathcal{C}$ is \emph{generated} by a family of objects $Q$ if, for any object $V$ of $\mathcal{C}$, there exists a family $\{b_k\}$ of \emph{reduction morphisms} $b_k \in \Hom(r_k,V)$, where each $r_k$ is a monoidal product of objects in $Q$, such that $\sum_k b_k b_k^{\dagger} = \id_V$.
\end{definition}
\begin{definition}
We say that a pair $(G,u)$ of a compact quantum group $G$ and some representation $u$ is a \emph{compact matrix quantum group} when $\Rep(G)$ is generated by the objects $\{u,u^*\}$.
\end{definition}
\noindent
For $(G,u)$ a compact matrix quantum group, we will now show that a UPT between fibre functors on $\Rep(G)$ is completely determined by its component on the fundamental representation $u$. 

First we define some notation: for a vector $\vec{x} \in \{\pm 1\}^n$, $n \in \mathbb{N}$, we write $u^{\vec{x}}$ for the object $u^{x_1} \otimes \dots \otimes u^{x_n}$, where we take $u^{-1}:= u^*$. We additionally define $u^0:= \mathbbm{1}$.
\begin{definition}
Let $(G,u)$ be a compact matrix quantum group, and let $F,F':\Rep(G) \to \Hilb$ be fibre functors. We define a \emph{reduced unitary pseudonatural transformation} $(\tilde{\alpha},H):F \to F'$ to be:
\begin{itemize}
\item A Hilbert space $H$ (drawn as a green wire).
\item A unitary morphism $\tilde{\alpha}: F(u) \otimes H \to H \otimes F'(u)$ (drawn as a white vertex) which is: 
\begin{itemize}
\item \emph{Natural}. For any 2-morphism $f: u^{\vec{x}} \to u^{\vec{y}}$ in $\mathcal{C}$:
\begin{calign}\label{eq:reduceduptnat}
\includegraphics[scale=1]{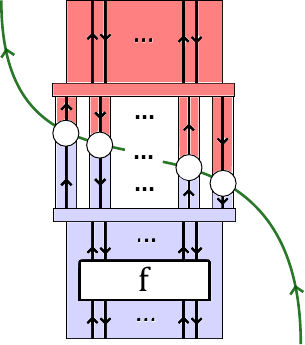}
~~=~~
\includegraphics[scale=1]{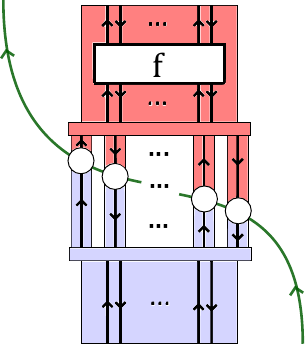}
\end{calign}
Here the empty horizontal blue and red rectangles represent some sequence of multiplicators, comultiplicators, unitors and counitors. For the purpose of drawing the diagram~\eqref{eq:reduceduptnat} we have supposed that $\vec{x},\vec{y}$ are both of the form $(1,-1,\dots,1,-1)$; it should be clear how to generalise to other $\vec{x},\vec{y}$ or to $u^{0}$ (e.g. if $f: u^{0} \to u^{\vec{y}}$, on the RHS of~\eqref{eq:reduceduptnat} the blue rectangle will be the counitor of $F$, the red will be the unitor of $F'$ and there will be no white vertices). We also used the following definition:
\begin{calign}\label{eq:reduceduptdualdef}
\includegraphics[scale=1]{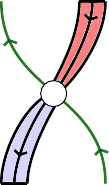}
~~=~~
\includegraphics[scale=1]{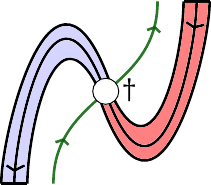}
\end{calign}
\end{itemize}
\end{itemize}
\end{definition}
\noindent
It is immediate from Definition~\ref{def:pntmon} and unitarity of $\alpha$~\eqref{eq:dualpnt} that a UPT $\alpha: F \to F'$  restricts to a reduced UPT $\alpha_u$. We now show that this correspondence is bijective: every reduced UPT induces a unique UPT.
\begin{lemma}\label{lem:reduceduptdual}
If $\tilde{\alpha}$ is a reduced UPT, then the morphism defined in~\eqref{eq:reduceduptdualdef} is unitary.
\end{lemma}
\begin{proof}
We show one of the two unitarity equations; the other is proved similarly.
\begin{calign}
\includegraphics[scale=1]{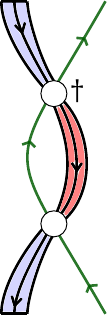}
~~=~~
\includegraphics[scale=1]{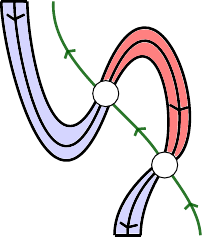}
~~=~~
\includegraphics[scale=1]{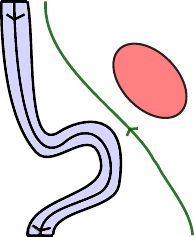}
~~=~~
\includegraphics[scale=1]{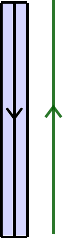}
\end{calign}
Here the first equality is by definition~\eqref{eq:reduceduptdualdef}, the second is by naturality~\eqref{eq:reduceduptnat} of the reduced UPT $\tilde{\alpha}$ for the morphism $\epsilon: u^{(-1,1)} \to u^0$, and the third is by unitarity of the monoidal functor $F'$.
\end{proof}
\begin{proposition}
Let $(G,u)$ be a compact matrix quantum group, let $F,F'$ be fibre functors, and let $(\tilde{\alpha},H): F \to F'$ be a reduced UPT. There is a unique UPT $(\alpha,H): F \to F'$ which restricts to $\tilde{\alpha}$ on $\alpha_u$, whose components $\alpha_V$ are defined as follows for any representation $V$ of $G$:
\begin{calign}\label{eq:induceduptdef}
\includegraphics[scale=1]{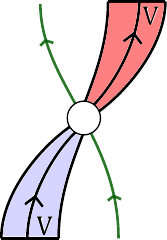}
~~:=~~
\Huge \sum_k
\includegraphics[scale=1]{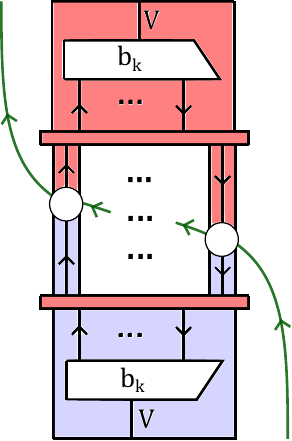}
\end{calign}
Here $\{b_k: u^{\vec{x_k}} \to V\}$ is any family of reduction morphisms.
\end{proposition}
\begin{proof}
First we show that $\alpha$ is well-defined, i.e. it does not depend on the choice of reduction morphisms. Let $V$ be some representation of $G$ and let $\{b_k: u^{\vec{x_k}} \to V\}$ and $\{c_l: u^{\vec{y_l}} \to V\}$ be two families of reduction morphisms. Then:
\begin{calign}
{\Huge \sum_k}
\includegraphics[scale=1]{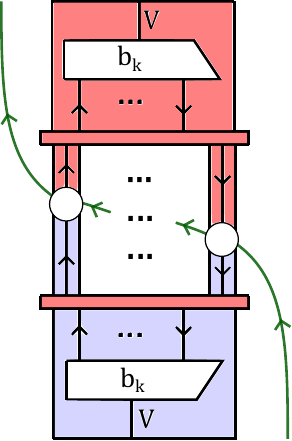}
~~=~~
{\Huge \sum_{k,l}}
\includegraphics[scale=1]{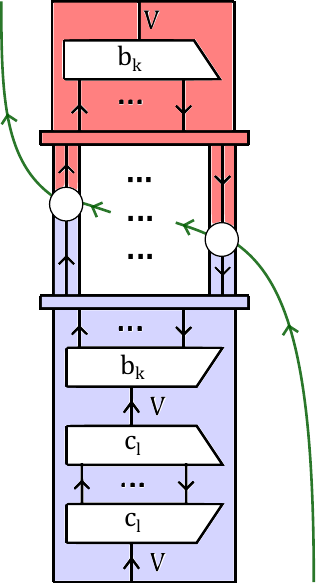}
~~=~~
{\Huge \sum_{k,l}}
\includegraphics[scale=1]{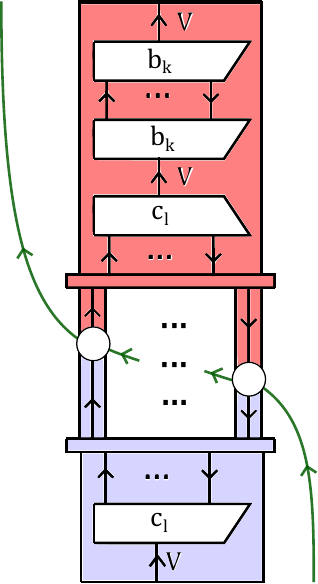}
~~\\=
{\Huge \sum_l}
\includegraphics[scale=1]{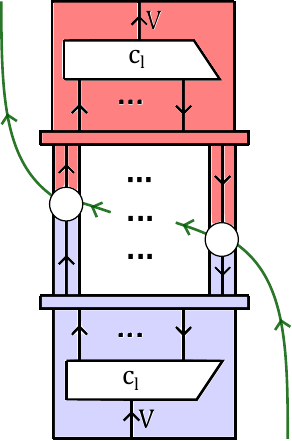}
\end{calign}
Here the first equality is by $\sum_l c_lc_l^{\dagger} = \id_V$, the second is by naturality of the reduced UPT $\tilde{\alpha}$, and the third is by $\sum_k b_k b_k^{\dagger} = \id_V$. 

We now show that $\alpha$ is indeed a UPT.
\begin{itemize}
\item \emph{Naturality.} Let $\{b_k: u^{\vec{x}_k} \to V\}$ and $\{c_l: u^{\vec{y}_l} \to W\}$ be reduction morphisms for representations $V,W$ of $\mathcal{C}$. We show~\eqref{eq:pntmonnat} for any morphism $f: V \to W$:
\begin{calign}
{\Huge \sum_{l}}
\includegraphics[scale=.9]{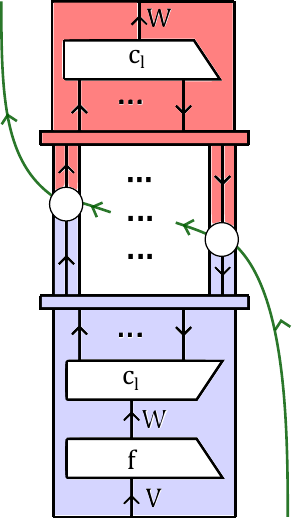}
~~={\Huge \sum_{k,l}}
\includegraphics[scale=.9]{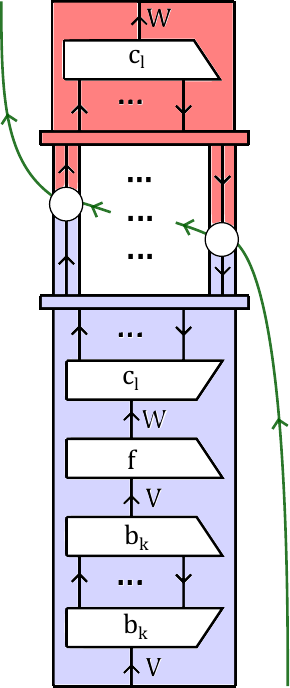}
~~={\Huge \sum_{k}}
\includegraphics[scale=.9]{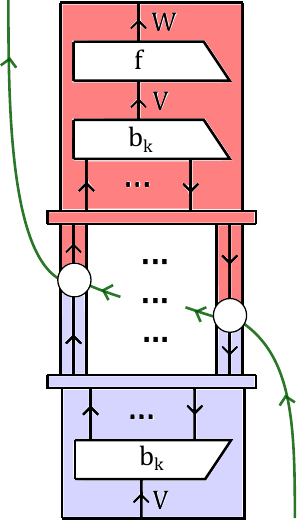}
\end{calign}
Here the first equality is by $\sum_k b_k b_k^{\dagger} = \id_V$, and the second is by naturality of the reduced UPT $\tilde{\alpha}$ and $\sum_l c_l c_l^{\dagger} = \id_W$. 
\item \emph{Monoidality.}
\begin{itemize}
\item Let $V,W$ be some representations of $G$, and pick some reduction morphisms $\{b_k: u^{\vec{x}_k} \to V\}, \{c_l: u^{\vec{y}_l} \to W\}$. It is clear that $\{b_k \otimes c_l: u^{\vec{x}_k} \otimes u^{\vec{y}_l} \to V \otimes W\}$ are reduction morphisms for $V \otimes W$. Now~\eqref{eq:pntmonmon} is immediate by manipulation of functorial boxes:
\begin{calign}{\Huge \sum_{k,l}}
\includegraphics[scale=1]{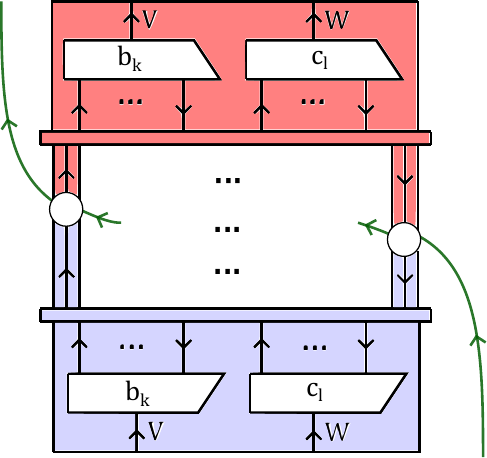}
~~=~~{\Huge \sum_{k,l}}
\includegraphics[scale=1]{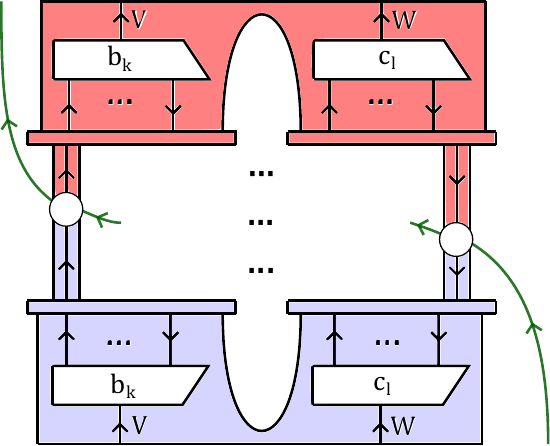}
\end{calign}
\item The equation~\eqref{eq:pntmonmonunit} is precisely~\eqref{eq:induceduptdef} for the object $\mathbbm{1}$ where the reduction morphism is $\id_{\mathbbm{1}}: \mathbbm{1} \to u^{0}$.
\end{itemize}
\item \emph{Unitarity.} We show one of the unitarity equations; the other is proved similarly. For any representation $V$ of $G$:
\begin{calign}
{\Huge \sum_{k,l}}
\includegraphics[scale=1]{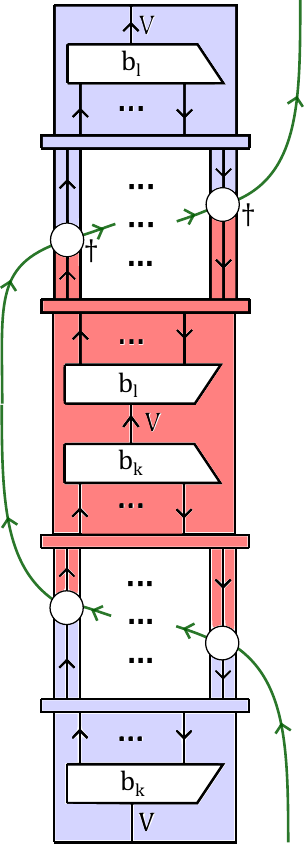}
~~=~~
{\Huge \sum_{k,l}}
\includegraphics[scale=1]{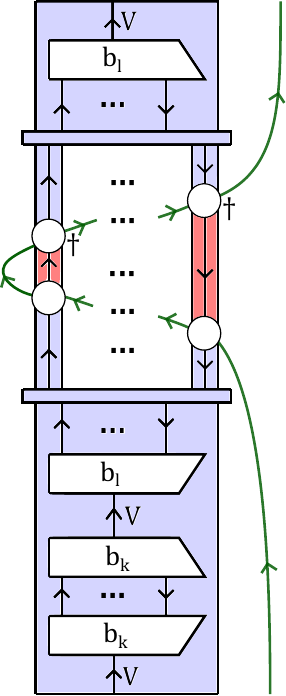}
~~=~~
{\Huge \sum_{l}}
\includegraphics[scale=1]{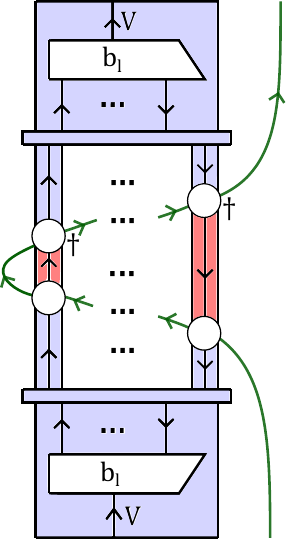}\\
~~=~~
{\Huge \sum_{l}}
\includegraphics[scale=1]{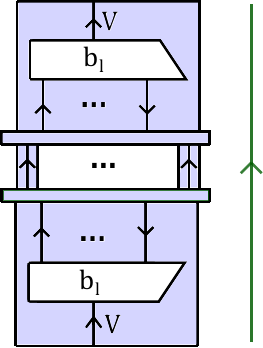}
~~=~~
\includegraphics[scale=1]{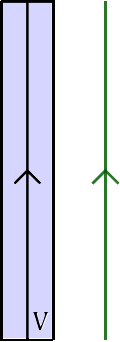}
\end{calign}
\end{itemize}
Here the first equation is by naturality for $\tilde{\alpha}$, the second is by $\sum_k b_k b_k^{\dagger} = \id_V$, the third is by unitarity of $\alpha_u$ and Lemma~\ref{lem:reduceduptdual}, and the fourth is by manipulation of functorial boxes and $\sum_l b_l b_l^{\dagger} = \id_V$.

Finally, uniqueness of the induced UPT $\alpha$ follows from the fact that, by naturality of the UPT $\alpha:F \to F'$, the component $\alpha_V$ for any $V$ is defined by $\alpha_U$ and $\alpha_{U^*}$ by~\eqref{eq:induceduptdef}.
\end{proof}
\noindent
We can also introduce a notion of modification for reduced UPTs.
\begin{definition}
Let $(G,u)$ be a compact matrix quantum group, let $F,F'$ be fibre functors, and let $(\tilde{\alpha},H_{\alpha})$, $(\tilde{\beta},H_{\beta})$ be reduced UPTs (the first drawn with a green wire, the second with an orange wire). Then a \emph{modification} $f:\tilde{\alpha} \to \tilde{\beta}$ is a linear map $f: H_{\alpha} \to H_{\beta}$ satisfying the following equations:
\begin{calign}\label{eq:uptredmods}
\includegraphics[scale=1]{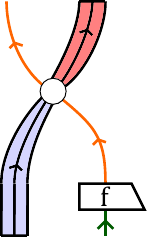}
~~=~~
\includegraphics[scale=1]{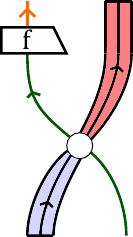}
&
\includegraphics[scale=1]{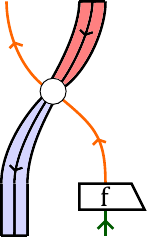}
~~=~~
\includegraphics[scale=1]{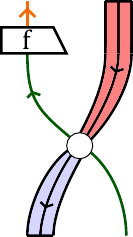}
\end{calign}
\end{definition}
\begin{proposition}
Let $(\tilde{\alpha},H_{\alpha}),(\tilde{\beta},H_{\beta}):F \to F'$ be reduced UPTs and let $(\alpha,H_{\alpha}),(\beta,H_{\beta}):F \to F'$ be the unique induced UPTs. Then a modification $f: \tilde{\alpha} \to \tilde{\beta}$ is precisely a modification $f: \alpha \to \beta$.
\end{proposition}
\begin{proof}
It is clear that every modification $\alpha \to \beta$ is in particular a modification $\tilde{\alpha} \to \tilde{\beta}$. That a modification $f:\tilde{\alpha} \to \tilde{\beta}$ is also a modification $f:\alpha \to \beta$ is clear from the definition~\eqref{eq:induceduptdef} of the induced UPTs; $f$ pulls through all the vertices, thereby satisfying the equations~\eqref{eq:uptmod}.
\end{proof}
\noindent
The results of this section are summarised by the following theorem.
\begin{theorem}\label{thm:redupt}
Let $(G,u)$ be a compact matrix quantum group and let $F,F'$ be fibre functors. Then restriction of UPTs $\alpha \mapsto \alpha_u$ defines an isomorphism of categories between:
\begin{itemize}
\item UPTs $F \to F'$ and modifications. 
\item Reduced UPTs $F \to F'$ and modifications.
\end{itemize}
\end{theorem}
\subsection{Quantum graphs and their isomorphisms}\label{sec:qgraphisossubsec}
We have seen that for a compact matrix quantum group $(G,u)$, a UPT $\alpha: F_1 \to F_2$ between fibre functors $F_1,F_2: \Rep(G) \to \Hilb$ is determined by a single unitary $(\alpha_U,H): F_1(u) \otimes H \to H \otimes F_2(u)$ obeying the naturality condition~\eqref{eq:reduceduptnat} for the intertwiner spaces $\Hom(u^{\otimes m}, u^{\otimes n})$.
 
\ignore{
We first recall the original Tannaka  reconstruction theorem for compact matrix quantum groups, which has the benefit of applying to non-semisimple categories. The theorem has been abridged and translated into the language of Section~\ref{}.
\begin{theorem}[{\cite{}}]
Let $(\mathcal{C},F)$ be a pair of a pivotal dagger category $\mathcal{C}$ (not necessarily semisimple) generated by $\{u,u^*\}$ for some object $u$, and a strict monoidal functor $U: \mathcal{C} \to \Hilb$. Then there exists a compact matrix quantum group $(G,u)$ and a fully faithful strict monoidal functor $I: \mathcal{C} \hookrightarrow \Rep(G)$ taking $u \mapsto u$, such that $U = F \circ I$. \ignore{, which is an monoidal equivalence if $\mathcal{C}$ is semisimple.} \ignore{Moreover, $F = \tilde{F} \circ I$, where $\tilde{F}$ is the canonical fibre functor on $\Rep(G)$.}
\end{theorem}
\ignore{
\noindent
The second is an immediate corollary of Theorem~\ref{}.
\begin{corollary}
Let $\mathcal{C}$ be a semisimple  pivotal dagger category generated by $\{u,u^*\}$ for some object $u$, and let $U: \mathcal{C} \to \Hilb$ be a unitary fibre functor. Then there exists a compact matrix quantum group $(G,u)$ and a unitary monoidal equivalence $E_U: \mathcal{C} \to \Rep(G)$ taking $u \mapsto u$, such that $U = F \circ E_U$.
\end{corollary}
}
\noindent
}
We now recall the notions of quantum confusability graph and finite-dimensional quantum graph isomorphism. For more information about these structures, their significance in quantum information theory, and how they generalise their classical counterparts, see~\cite{Musto2018, Brannan2019, Atserias2019,Duan2012}.
\begin{definition}[{\cite[Def. 5.1]{Musto2018}}]\label{def:qgraph}
A \emph{quantum graph} $X=(A,\Gamma)$ is a pair of:
\begin{itemize}
\item A special Frobenius monoid $(A,m,u)$ in the category $\Hilb$ (Definition~\ref{def:Frobenius}) satisfying the following \emph{symmetry} equation:
\begin{calign}\label{eq:frobsymmetry}
\begin{tz}[zx,every to/.style={out=up, in=down}]
\draw (0.25,-1.5) to (1.75,0) to [in=-45] (1,1) to (1,1.5);
\draw (1.75,-1.5) to (0.25,0) to [in=-135] (1,1);
\node[zxvertex=\zxwhite, zxdown] at (1,1){};
\node[zxvertex=\zxwhite] at (1,1.5){};
\end{tz}
\quad = \quad
\begin{tz}[zx,every to/.style={out=up, in=down}]
\draw (1.75,-1.5) to (1.75,0) to [in=-45] (1,1) to (1,1.5);
\draw (0.25,-1.5) to (0.25,0) to [in=-135] (1,1);
\node[zxvertex=\zxwhite, zxdown] at (1,1){};
\node[zxvertex=\zxwhite] at (1,1.5){};
\end{tz}
\end{calign}
\ignore{
\begin{calign}
\begin{tz}[zx,every to/.style={out=up, in=down}]\draw (0,0) to (0,1) to [out=135] (-0.75,2) to [in=-135] (0,3) to (0,4);
\draw (0,1) to [out=45] (0.75,2) to [in=-45] (0,3);
\node[zxvertex=\zxwhite, zxup] at (0,1){};
\node[zxvertex=\zxwhite,zxdown] at (0,3){};\end{tz}
\quad = \quad \delta^2~~
\begin{tz}[zx]
\draw (0,0) to +(0,4);
\end{tz}
\end{calign}
}
\item A self-adjoint linear map $\Gamma:A \to A$ satisfying the following equations:
\begin{calign}\label{eq:propadjacency}
\begin{tz}[zx]
\draw (0,0) to (0,0.5) to [out=135, in=-135,looseness=1.5] node[zxnode=\zxwhite, pos=0.5] {$\Gamma$} (0,2.5) to (0,3);
\draw[string] (0,0.5) to [out=45, in=-45,looseness=1.5] node[zxnode=\zxwhite, pos=0.5] {$\Gamma$} (0,2.5);
\node[zxvertex=\zxwhite,zxup] at (0,0.5){};
\node[zxvertex=\zxwhite,zxdown] at (0,2.5){};
\end{tz}\quad  = \quad 
\begin{tz}[zx]
\draw (0,0) to node[zxnode=\zxwhite, pos=0.5] {$\Gamma$} (0,3);
\end{tz}
&
\begin{tz}[zx]
\draw (1,0) to (1,1.5) to [out=up, in=up, looseness=2.5]node[zxvertex=\zxwhite, pos=0.5] {} (0,1.5) to [out=down, in=down, looseness=2.5]node[zxvertex=\zxwhite, pos=0.5] {} (-1,1.5) to (-1,3);
\node[zxnode=\zxwhite] at (0,1.5) {$\Gamma$};
\end{tz}
\quad = \quad 
\begin{tz}[zx]
\draw (0,0) to (0,3);
\node[zxnode=\zxwhite] at (0,1.5) {$\Gamma$};
\end{tz}
\ignore{
&
\begin{tz}[zx]
\draw (0,0) to (0,0.5) to [out=135, in=-135,looseness=1.5] node[zxnode=\zxwhite, pos=0.5] {$\Gamma$} (0,2.5) to (0,3);
\draw[string] (0,0.5) to [out=45, in=-45,looseness=1.5](0,2.5);
\node[zxvertex=\zxwhite,zxup] at (0,0.5){};
\node[zxvertex=\zxwhite,zxdown] at (0,2.5){};
\end{tz}\quad  = \quad 
\begin{tz}[zx]
\draw (0,0) to (0,3);
\end{tz}
}
\end{calign}
\end{itemize}
\end{definition}
\begin{example}
We briefly recall how quantum graphs generalise ordinary graphs. Let $\{\ket{i}\}$ be an orthonormal basis for a $d$-dimensional Hilbert space $A$. This Hilbert space has the structure of a special symmetric Frobenius algebra in $\Hilb$ with the following operations:
\begin{calign}
\ket{i} \otimes \ket{j} \mapsto \delta_{ij} \ket{i} 
&&
1 \mapsto \sum_{i} \ket{i}
&&
\ket{i} \mapsto \ket{i} \otimes \ket{i}
&&
\ket{i} \mapsto 1
\\
\textrm{multiplication}
&&
\textrm{unit}
&&
\textrm{comultiplication}
&&
\textrm{counit}
\end{calign}
We can express any linear map $\Gamma: A \to A$ in the basis $\{\ket{i}\}$ as a matrix $(\Gamma_{ij})_{i,j}$. It is then easy to check that~\eqref{eq:propadjacency} reduces to the following conditions for the matrix entries:
\begin{calign}
\Gamma_{ij}^2 = \Gamma_{ij}
&&
\Gamma_{ij} = \Gamma_{ji}
\end{calign}
These are clearly the equations specifying that $\Gamma$ is the adjacency matrix for a graph with $d$ vertices, i.e. a symmetric $d \times d$ matrix with entries 0 and 1. In order to restrict to e.g. simple graphs we could add another equation, but we do not do this here.
\end{example}
\ignore{
\begin{remark}
A symmetric Frobenius monoid in $\Hilb$ is precisely a f.d. $C^*$-algebra equipped with a faithful trace (the inner product is obtained as $\braket{a | b} = \Tr(a^*b)$, and the trace is the counit of the Frobenius monoid). There are various ways to normalise this trace. One approach is to require that the Frobenius monoid be special --- this approach was taken in the definition of quantum graphs given in~\cite{Musto2018}. Another normalisation was chosen in~\cite{Brannan2019}, where the trace was required to be a \emph{$\delta$-form} for some $\delta>0$:
\begin{calign}
\begin{tz}
\draw (0,1) to (0,2);
\node[zxvertex=\zxwhite] at (0,1){};
\node[zxvertex=\zxwhite] at (0,2){};
\end{tz}
~~=~~
1
&&
\begin{tz}[zx,every to/.style={out=up, in=down}]\draw (0,0) to (0,1) to [out=135] (-0.75,2) to [in=-135] (0,3) to (0,4);
\draw (0,1) to [out=45] (0.75,2) to [in=-45] (0,3);
\node[zxvertex=\zxwhite, zxup] at (0,1){};
\node[zxvertex=\zxwhite,zxdown] at (0,3){};\end{tz}
\quad = \quad \delta^2~~
\begin{tz}[zx]
\draw (0,0) to +(0,4);
\end{tz}
\end{calign}
 These normalisations produce slightly different definitions of quantum graph. Definition~\ref{def:qgraph} includes them both by not stipulating any normalisation for the trace. In view of Lemma~\ref{lem:fctrtoqgraph} we remark that, if either of these normalisations is chosen, it will be preserved under an accessible fibre functor.
\end{remark}}
\begin{definition}[{\cite[Def. 5.11]{Musto2018}}]\label{def:qgraphiso}
Let $X=(A,\Gamma)$ and $Y=(A',\Gamma')$ be quantum graphs. A \emph{finite-dimensional quantum graph isomorphism} $(P,H): X \to Y$ is a pair of a Hilbert space $H$ and a unitary linear map $P: A \otimes H \to H \otimes A'$ satisfying the following equations, where the monoids $A$ and $A'$ are depicted as white and grey nodes respectively:
\begin{calign}\nonumber
\begin{tz}[zx,xscale=1,every to/.style={out=up, in=down},scale=-1]
\draw (0,0) to (0,2) to [out=45] (0.75,3);
\draw (0,2) to [out=135] (-0.75,3);
\draw[arrow data={0.2}{<}, arrow data={0.8}{<}] (1.75,0) to [looseness=0.9] node[zxnode=\zxwhite, pos=0.5] {$P$} (-1.75,2.5) to (-1.75,3);
\node[zxvertex=\zxwhite, zxdown] at (0,2){};
\end{tz}
=
\begin{tz}[zx,xscale=1,every to/.style={out=up, in=down},scale=-1]
\draw (0,0) to (0,0.75) to [out=45] (0.75,1.75) to (0.75,3);
\draw (0,0.75) to [out=135] (-0.75,1.75) to (-0.75,3);
\draw[arrow data={0.2}{<}, arrow data={0.9}{<}] (1.75,0) to (1.75,0.75) to  [looseness=1.1, in looseness=0.9] node[zxnode=\zxwhite, pos=0.36] {$P$} node[zxnode=\zxwhite, pos=0.64] {$P$}(-1.75,3);
\node[zxvertex=\zxblack, zxdown] at (0,0.75){};
\end{tz}
&
\begin{tz}[zx,xscale=1,every to/.style={out=up, in=down},scale=-1]
\draw (0,0) to (0,2.25);
\draw[arrow data={0.2}{<}, arrow data={0.8}{<}] (1,0) to [looseness=0.9] node[zxnode=\zxwhite, pos=0.5] {$P$} (-1,2.5) to (-1,3);
\node[zxvertex=\zxwhite] at (0,2.25){};
\end{tz}
=
\begin{tz}[zx,xscale=1,every to/.style={out=up, in=down},scale=-1]
\draw (0,0) to (0,0.75);
\draw[arrow data={0.2}{<}, arrow data={0.9}{<}] (1.,0) to (1.,0.75) to   (-1,3);
\node[zxvertex=\zxblack] at (0,0.75){};
\end{tz}
&
\begin{tz}[zx,xscale=1,every to/.style={out=up, in=down},scale=1]
\draw (0,0) to (0,3);
\draw[arrow data={0.2}{>}, arrow data={0.9}{>}] (1.75,0) to (1.75,0.75) to  [looseness=1.1, in looseness=0.9]  node[zxnode=\zxwhite, pos=0.5] {$P$}(-1.75,3);
\node[zxnode=\zxwhite] at (0,0.9){$\Gamma$};
\end{tz}
=
\begin{tz}[zx,xscale=1,every to/.style={out=up, in=down},scale=1]
\draw (0,0) to (0,3);
\draw[arrow data={0.2}{>}, arrow data={0.8}{>}] (1.75,0) to [looseness=0.9] node[zxnode=\zxwhite, pos=0.5] {$P$} (-1.75,2.5) to (-1.75,3);
\node[zxnode=\zxwhite] at (0,2.35) {$\Gamma'$};
\end{tz}
\end{calign}
\begin{calign}
\label{eq:quantumfunction2}
\begin{tz}[zx,xscale=1,every to/.style={out=up, in=down}]
\draw (0,0) to (0,2) to [out=45] (0.75,3);
\draw (0,2) to [out=135] (-0.75,3);
\draw[arrow data={0.2}{>}, arrow data={0.8}{>}] (1.75,0) to [looseness=0.9] node[zxnode=\zxwhite, pos=0.5] {$P$} (-1.75,2.5) to (-1.75,3);
\node[zxvertex=\zxblack, zxup] at (0,2){};
\end{tz}
=
\begin{tz}[zx,xscale=1,every to/.style={out=up, in=down}]
\draw (0,0) to (0,0.75) to [out=45] (0.75,1.75) to (0.75,3);
\draw (0,0.75) to [out=135] (-0.75,1.75) to (-0.75,3);
\draw[arrow data={0.2}{>}, arrow data={0.9}{>}] (1.75,0) to (1.75,0.75) to  [looseness=1.1, in looseness=0.9] node[zxnode=\zxwhite, pos=0.36] {$P$} node[zxnode=\zxwhite, pos=0.64] {$P$}(-1.75,3);
\node[zxvertex=\zxwhite, zxup] at (0,0.75){};
\end{tz}
&
\begin{tz}[zx,xscale=1,every to/.style={out=up, in=down}]
\draw (0,0) to (0,2.25);
\draw[arrow data={0.2}{>}, arrow data={0.8}{>}] (1,0) to [looseness=0.9] node[zxnode=\zxwhite, pos=0.5] {$P$} (-1,2.5) to (-1,3);
\node[zxvertex=\zxblack] at (0,2.25){};
\end{tz}
=
\begin{tz}[zx,xscale=1,every to/.style={out=up, in=down}]
\draw (0,0) to (0,0.75);
\draw[arrow data={0.2}{>}, arrow data={0.9}{>}] (1.,0) to (1.,0.75) to   (-1,3);
\node[zxvertex=\zxwhite, zxup] at (0,0.75){};
\end{tz}
\end{calign}
\begin{calign}\label{eq:quantumfunction} \begin{tz}[zx,xscale=1,every to/.style={out=up, in=down},xscale=0.8]
\draw [arrow data={0.2}{>},arrow data={0.8}{>}]  (0,0) to (2.25,3);
\draw (2.25,0) to node[zxnode=\zxwhite, pos=0.5] {$P^\dagger$} (0,3);
\end{tz}
=~~
\begin{tz}[zx,xscale=1,xscale=0.6,yscale=-1]
\draw[arrow data={0.5}{<}] (0.25,-0.5) to (0.25,0) to [out=up, in=-135] (1,1);
\draw (1,1) to [out=135, in=right] node[zxvertex=\zxwhite, pos=1]{} (-0.3, 1.7) to [out=left, in=up] (-1.25,1) to (-1.25,-0.5);
\draw[arrow data={0.5}{>}] (1.75,2.5) to (1.75,2) to [out=down, in=45] (1,1);
\draw (1,1) to [out= -45, in= left] node[zxvertex=\zxblack, pos=1] {} (2.3,0.3) to [out=right, in=down] (3.25,1) to (3.25,2.5);
\node [zxnode=\zxwhite] at (1,1) {$P$};
\end{tz} 
\end{calign}
An \textit{intertwiner} of quantum isomorphisms $(H,P) \to (H', P')$ is a linear map $f:H\to H'$ such that the following holds:%
\begin{calign}\label{eq:intertwiner}
\begin{tz}[zx,xscale=1,every to/.style={out=up, in=down}]
\draw (0,0) to (0,1) to (2,3);
\draw[arrow data={0.35}{>}] (2,0) to node[zxnode=\zxwhite, pos=0.9] {$f$} (2,1);
\draw[string,arrow data={0.9}{>},arrow data={0.26}{>}](2,1) to node[zxnode=\zxwhite, pos=0.5]{$P'$}  (0,3);
\end{tz}
\quad = \quad
\begin{tz}[zx,xscale=1,every to/.style={out=up, in=down},scale=-1]
\draw (0,0) to (0,1) to (2,3);
\draw[arrow data={0.35}{<}] (2,0) to node[zxnode=\zxwhite, pos=0.9] {$f$} (2,1);
\draw[string,arrow data={0.9}{<},arrow data={0.26}{<}](2,1) to node[zxnode=\zxwhite, pos=0.5]{$P$}  (0,3);
\end{tz}
\end{calign}
\end{definition}
\begin{remark}
This definition uses the opposite convention for the direction of the Hilbert space wire to the paper~\cite{Musto2018}. The theory of quantum isomorphisms can be developed analogously whichever convention is used. 
\end{remark}
\noindent
We now recall the definition of the quantum automorphism group of a quantum graph. 
\begin{definition}[{\cite[Def. 3.7]{Brannan2019}}]
Let $X=(A,\Gamma)$ be a quantum graph with $\dim(A)=n$, and let $\{\ket{i}\}_{i=1}^n$ be an orthonormal basis for $A$. Then the \emph{quantum automorphism group algebra} $O(G_X)$ is the universal unital $*$-algebra generated by the coefficients of a unitary matrix $[u_{ij}]_{i,j=1}^n \in M_n(O(G_X))$ subject to the relations making the map 
\begin{align*}
\rho: A \to A \otimes O(G_X) && \rho(\ket{i}) = \sum_j \ket{j} \otimes u_{ji}
\end{align*}
a unital $*$-homomorphism satisfying $\rho \circ \Gamma = (\Gamma \otimes \id_{O(G_X)}) \circ \rho$. Its Hopf-$*$-algebra structure is defined in~\cite[Prop. 3.8]{Brannan2019}.
\end{definition}
\begin{remark}
$O(G_X)$ can also be defined as the $*$-algebra of matrix coefficients of corepresentations of the $C^*$-algebra obtained by Woronowicz's Tannaka-Krein construction~\cite{Woronowicz1988} for a suitable concrete $W^*$-category (c.f.~\cite[Prop. 1.1]{Banica1999}). In particular, we have the following facts:
\begin{enumerate}
\item $G_X$ is a compact matrix quantum group with fundamental representation $A$.
\item The intertwiner spaces $\Hom_{\Rep(G_X)}(A^{\otimes m},A^{\otimes n})$ are generated by three morphisms $m:A \otimes A \to A$, $u:\mathbb{C} \to A$ and $\Gamma:A \to A$, satisfying the equations of a Frobenius monoid and of a quantum graph, under composition, monoidal product, dagger and linear combination.
\item The fundamental representation $A$ is self-dual in $\Rep(G_X)$ with cup and cap~\eqref{eq:cupcapfrob}.
\item The image of $((A,m,u),\Gamma)$ under the canonical fibre functor $F: \Rep(G_X) \to \Hilb$ is the quantum graph $X$.
\end{enumerate}
\end{remark}
\begin{definition}[{\cite[Def. 4.1]{Brannan2019}}]
Let $X=(A_X,\Gamma_X),Y=(A_Y,\Gamma_Y)$, be two quantum graphs, where $\dim(A_X) = n, \dim(A_Y) =m$, and let $\{\ket{i}\}_{i=1}^n$ and $\{\ket{j}\}_{j=1}^m$ be orthonormal bases for $A_X$ and $A_Y$ respectively. Define $O(G_Y,G_X)$ to be the universal $*$-algebra generated by the coefficients of a unitary matrix $p=[p_{ij}]_{ij} \in O(G_Y,G_X) \otimes B(A_X,A_Y)$ with relations ensuring that 
\begin{align*}
\rho_{Y,X}: A_X \to A_Y \otimes O(G_Y,G_X)&& \rho_{Y,X}(\ket{i})= \sum_j \ket{j} \otimes p_{ji}
\end{align*}
is a unital $*$-homomorphism satisfying $\rho \circ \Gamma_X = (\Gamma_Y \otimes \id_{O(G_Y,G_X)}) \circ \rho$.
\ignore{
If $O(G_Y,G_X)$ is nonzero we say that $G_X$ is \emph{algebraically quantum isomorphic} to $G_Y$ and write $G_X \cong_A G_Y$.}
\end{definition}
\begin{lemma}\label{lem:fctrtoqgraph}
Let $G_X$ be the automorphism group of a quantum graph $F(X)=((F(A),F(m),F(u)),F(\Gamma))$, where $F: \Rep(G_X) \to \Hilb$ is the canonical fibre functor, $A$ is the generating object of $\Rep(G_X)$ and $m,u,\Gamma$ are the generating morphisms. Let $F':\Rep(G) \to \Hilb$ be any other fibre functor accessible from $F$ by a UPT. Then $F'(X):=((F'(A),F'(m),F'(u)),F'(\Gamma))$ is a quantum graph f.d. quantum isomorphic to $X$:
\begin{calign}\label{eq:fctrtoqgraph}
\includegraphics[scale=1]{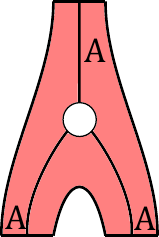}
&
\includegraphics[scale=1]{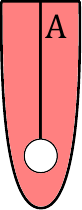}
&
\includegraphics[scale=1]{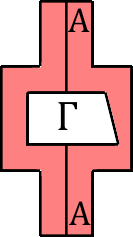}
\\
F'(m): F'(A) \otimes F'(A) \to F'(A)
&
F'(u): \mathbb{C} \to F'(A)
&
F'(\Gamma): F'(A) \to F'(A)
\end{calign}
All quantum graphs f.d. quantum isomorphic to $F(X)$ are obtained in this way.
\end{lemma}
\begin{proof}
All axioms of a quantum graph except symmetry follow straightforwardly from unitarity of the functor $F'$ (we need Lemma~\ref{lem:pushpast} for the Frobenius axiom). For the symmetry condition~\eqref{eq:frobsymmetry}, we recall that any accessible fibre functor has the form given in Theorem~\eqref{thm:splitting}. We can then show symmetry of $F'(X)$ as follows:
\begin{calign}
\includegraphics[scale=.8]{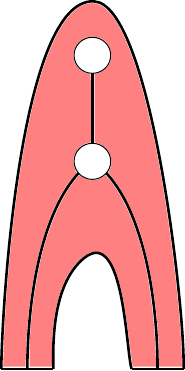}
~~=~~
\includegraphics[scale=.8]{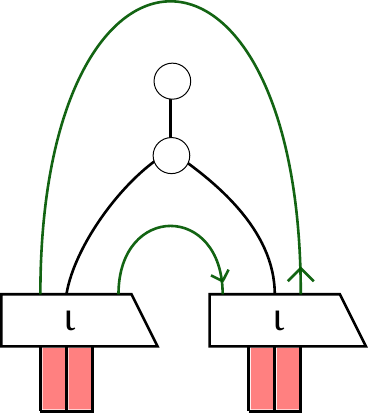}
~~=~~
\includegraphics[scale=.8]{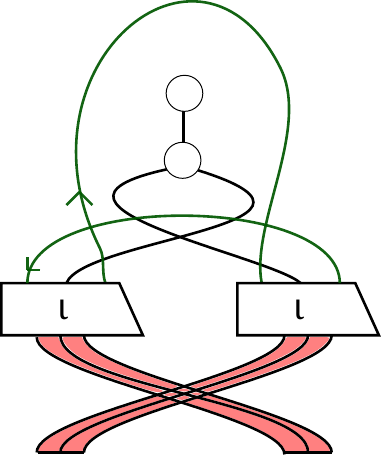}
~~=~~
\includegraphics[scale=.8]{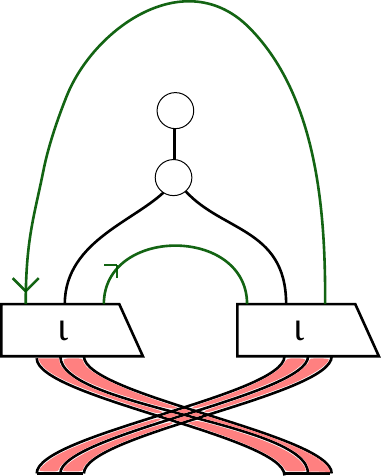}
~~=~~
\includegraphics[scale=.8]{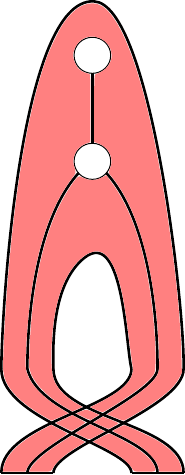}
\end{calign}
Here for the first equality we used the definition of the accessible fibre functor $F'$ (Theorem~\ref{thm:splitting}); for the second equality pulled one leg over the other and undid twists in the green wires; for the third equality we used symmetry of $F(X)$; and for the fourth equality we again used the definition of $F'$.

It follows that $F'(X)$ is a quantum graph. To see that $F'(X)$ is f.d. quantum isomorphic to $F(X)$, consider any UPT $\alpha:F \to F'$. Then its restriction $\alpha_A$ is a quantum isomorphism: the equations~\eqref{eq:quantumfunction2} follow from monoidality and naturality of $\alpha$, and the equation~\eqref{eq:quantumfunction} follows from unitarity of $\alpha$~\eqref{eq:dualpnt} and self-duality of $A$.
 
To see that all quantum graphs f.d. quantum isomorphic to $F(X)$ are thus obtained, we observe that a f.d. quantum isomorphism $P: Y \to X$ is an  f.d. $*$-representation of $O(G_X,G_Y)$, which is therefore nonzero. By~\cite[Thm. 4.5]{Brannan2019} $O(G_X,G_Y)$ is an $O(G_X)$-Hopf-Galois object and the associated fibre functor $\Rep(G_X)\to \Hilb$ takes $(A,m,u,\Gamma) \mapsto Y$. This fibre functor $F'$ is accessible from $F$ precisely because $O(G_X,G_Y)$ possesses a finite dimensional $*$-representation (Theorem~\ref{thm:higherhopfgal}).
\end{proof}
\begin{lemma}\label{lem:qisosredupts}
Let $X=(A,\Gamma)$ be a quantum graph, let $G_X$ be its quantum automorphism group and let $F_1,F_2: \Rep(G_X) \to \Hilb$ be two fibre functors accessible from the canonical fibre functor. Then there is an isomorphism of categories between:
\begin{itemize}
\item UPTs $F_1 \to F_2$, and modifications.
\item Finite-dimensional quantum isomorphisms $F_1(X) \to F_2(X)$, and intertwiners. 
\end{itemize}
\end{lemma}
\begin{proof}
We observe that a finite-dimensional quantum isomorphism $F_1(X) \to F_2(X)$ is precisely a reduced UPT $F_1 \to F_2$. Indeed, $A$ is the generating object for $\Rep(G_X)$ and the quantum isomorphism is a unitary linear map of the right type. We must therefore show the naturality equations~\eqref{eq:reduceduptnat}. But these are given on generators precisely by~\eqref{eq:quantumfunction2}, which is sufficient. The equation~\eqref{eq:quantumfunction} follows from self-duality of $A$. 

We also observe that an intertwiner of finite-dimensional quantum graph isomorphisms is precisely a modification of reduced UPTs. This follows from self-duality of $A$, which reduces the equations~\eqref{eq:uptredmods} to~\eqref{eq:intertwiner}.

The result then follows immediately from Theorem~\ref{thm:redupt}.
\end{proof}
\noindent
We now state the main result of this section.
\begin{definition}
Let $X$ be a quantum graph. The 2-category $\QGraph_X$ is defined as follows:
\begin{itemize}
\item \emph{Objects}: Quantum graphs f.d. quantum isomorphic to $X$.
\item \emph{1-morphisms}: Finite-dimensional quantum isomorphisms. 
\item \emph{2-morphisms}: Intertwiners. 
\end{itemize}
\end{definition}
\begin{definition}
Let $X$ be a quantum graph. The 2-category $\Fun(\Rep(G_X,\Hilb))_F$ is defined as follows:
\begin{itemize}
\item \emph{Objects}: Fibre functors accessible from the canonical fibre functor $F:\Rep(G_X) \to \Hilb$.
\item \emph{1-morphisms}: UPTs. 
\item \emph{2-morphisms}: Modifications. 
\end{itemize}
\end{definition}
\begin{theorem}\label{thm:qgraphequiv}
Let $X$ be a quantum graph. Then there is a 2-equivalence $\Fun(\Rep(G_X),\Hilb)_F \simeq \QGraph_{X}$. Moreover, this 2-equivalence is an isomorphism on $\Hom$-categories. 
\end{theorem}
\begin{proof}
We define a strict 2-functor $\Fun(\Rep(G_X,\Hilb)_F \to  \QGraph_X$ witnessing the equivalence as follows.
\begin{itemize}
\item \emph{On objects}: An accessible fibre functor $F': \Rep(G) \to \Hilb$ is taken to the quantum graph~\eqref{eq:fctrtoqgraph}. 
\item \emph{On 1-morphisms}: A UPT $\alpha: F' \to F''$ is taken to its component $\alpha_A$.
\item \emph{On 2-morphisms}: A modification $\alpha \to \beta$ is taken to an intertwiner $\alpha_A \to \beta_A$.
\end{itemize}
We first show that this is a well-defined strict 2-functor. That the quantum graph~\eqref{eq:fctrtoqgraph} is f.d. quantum isomorphic to $X$ was shown in Lemma~\ref{lem:fctrtoqgraph}, so the 2-functor is well-defined on objects. Well-definition on 1-morphisms and 2-morphisms follows from Lemma~\ref{lem:qisosredupts}. \ignore{We also saw in the proof of that proposition that the component $\alpha_A$ of a UPT obeys the equations (\ref{eq:quantumfunction2},~\ref{eq:quantumfunction}), so the pseudofunctor is well-defined on 1-morphisms. The intertwiner equation~\eqref{eq:intertwiner} is then just the equation for a modification~\eqref{eq:uptmod} on the component $\alpha_V$, so the pseudofunctor is well-defined on 2-morphisms.} Compositionality is clear by comparing the composition of quantum graph isomorphisms and intertwiners~\cite[Def. 3.18]{Musto2018} to that of UPTs and modifications. Essential surjectivity on objects follows immediately from the last statement of Lemma~\ref{lem:fctrtoqgraph}. 
\ignore{
For essential surjectivity on 1-morphisms we observe that a quantum graph isomorphism $(P,H):F'(X) \to F''(Y)$ is precisely a reduced UPT $F' \to F''$ in the sense of Definition~\ref{}. Indeed, the equation~\eqref{eq:quantumfunction} corresponds to self-duality of the object $A$, implying that the morphism $\eqref{}$ is precisely $(P,H)$ again and we only need to check the naturality condition~\eqref{} on morphisms $A^{\otimes m} \to A^{\otimes n}$; the equations~\eqref{eq:quantumfunction2} are precisely this naturality condition on the generators, which is sufficient.
By Theorem~\ref{} there is therefore a UPT $\alpha: F' \to F''$ which restricts to $(P,H)$ on $A$.
Finally, fullness and faithfulness on 2-morphisms follow from Proposition~\ref{}, since $A$ is self-dual and therefore only one of the modification conditions~\eqref{} is sufficient. 
}
That the equivalence is in fact an isomorphism on Hom-categories follows from Lemma~\ref{lem:qisosredupts}.
\end{proof}
\begin{remark}
By Theorem~\ref{thm:qgraphequiv}, we recover previous results about quantum graphs and finite-dimensional quantum isomorphisms. Indeed, one may immediately apply Theorem~\ref{thm:moritaclass} to obtain a classification of finite-dimensional quantum graph isomorphisms in terms of simple Frobenius monoids in the category of finite-dimensional $*$-representations of the quantum automorphism group algebra. This classification has already been shown in~\cite{Musto2019}. By Theorem~\ref{thm:higherhopfgal}, we also obtain a correspondence between quantum isomorphisms from a quantum graph and $*$-representations of Hopf-Galois objects for its quantum automorphism group, which has already been shown in~\cite{Brannan2019}.

Theorem~\ref{thm:qgraphequiv} additionally shows us that a finite-dimensional isomorphism from a quantum graph extends to a unitary pseudonatural transformation of fibre functors on the whole category of continuous unitary representations of the quantum automorphism group of the quantum graph. These transformations can be interpreted physically; this will be discussed in future work. This is a step towards finding an operational interpretation of quantum isomorphisms between quantum graphs, generalising the interpretation of quantum isomorphisms between classical graphs in terms of nonlocal games~\cite{Atserias2019}.
\end{remark}

\bibliographystyle{plainurl}

\begin{thebibliography}{10}

\bibitem{Atserias2019}
Albert Atserias, Laura Man{\v{c}}inska, David~E Roberson, Robert
  {\v{S}}{\'a}mal, Simone Severini, and Antonios Varvitsiotis.
\newblock Quantum and non-signalling graph isomorphisms.
\newblock {\em Journal of Combinatorial Theory, Series B}, 136:289--328, 2019.
\newblock \href {http://arxiv.org/abs/1611.09837} {\path{arXiv:1611.09837}}.

\bibitem{Banica1999}
Teodor Banica.
\newblock Symmetries of a generic coaction.
\newblock {\em Mathematische Annalen}, 314(4):763--780, 1999.
\newblock \href {http://arxiv.org/abs/math/9811060}
  {\path{arXiv:math/9811060}}.

\bibitem{Bichon1999}
Julien Bichon.
\newblock Galois extension for a compact quantum group.
\newblock 1999.
\newblock \href {http://arxiv.org/abs/math/9902031}
  {\path{arXiv:math/9902031}}.

\bibitem{Bichon2014}
Julien Bichon.
\newblock Hopf-{G}alois objects and cogroupoids.
\newblock {\em Rev. Un. Mat. Argentina}, 55(2):11--69, 2014.
\newblock \href {http://arxiv.org/abs/1006.3014} {\path{arXiv:1006.3014}}.

\bibitem{Brannan2019}
Michael Brannan, Alexandru Chirvasitu, Kari Eifler, Samuel Harris, Vern
  Paulsen, Xiaoyu Su, and Mateusz Wasilewski.
\newblock Bigalois extensions and the graph isomorphism game.
\newblock {\em Communications in Mathematical Physics}, pages 1--33, 2019.
\newblock \href {http://arxiv.org/abs/1812.11474} {\path{arXiv:1812.11474}}.

\bibitem{Duan2012}
Runyao Duan, Simone Severini, and Andreas Winter.
\newblock Zero-error communication via quantum channels, noncommutative graphs,
  and a quantum {L}ov{\'a}sz number.
\newblock {\em IEEE Transactions on Information Theory}, 59(2):1164--1174,
  2012.
\newblock \href {http://arxiv.org/abs/1002.2514} {\path{arXiv:1002.2514}}.

\bibitem{Heunen2016}
Chris Heunen and Martti Karvonen.
\newblock Monads on dagger categories.
\newblock {\em Theory and Applications of Categories}, 31(35):1016--1043, 2016.
\newblock \href {http://arxiv.org/abs/1602.04324} {\path{arXiv:1602.04324}}.

\bibitem{Heunen2019}
Chris Heunen and Jamie Vicary.
\newblock {\em Categories for Quantum Theory: An Introduction}.
\newblock Oxford Graduate Texts in Mathematics Series. Oxford University Press,
  2019.
\newblock \href {http://dx.doi.org/10.1093/oso/9780198739623.001.0001}
  {\path{doi:10.1093/oso/9780198739623.001.0001}}.

\bibitem{Joyal1991}
Andr{\'e} Joyal and Ross Street.
\newblock An introduction to {T}annaka duality and quantum groups.
\newblock In {\em Category theory}, pages 413--492. Springer, 1991.
\newblock URL: \url{http://web.science.mq.edu.au/~street/CT90Como.pdf}.

\bibitem{Leinster1998}
Tom Leinster.
\newblock Basic bicategories.
\newblock 1998.
\newblock \href {http://arxiv.org/abs/math/9810017}
  {\path{arXiv:math/9810017}}.

\bibitem{MacLane1963}
Saunders MacLane.
\newblock Natural associativity and commutativity.
\newblock {\em Rice Institute Pamphlet-Rice University Studies}, 49(4), 1963.

\bibitem{MacLane2013}
Saunders MacLane.
\newblock {\em Categories for the working mathematician}, volume~5.
\newblock Springer Science \& Business Media, 2013.

\bibitem{Mellies2006}
Paul-Andr{\'e} Melli{\`e}s.
\newblock Functorial boxes in string diagrams.
\newblock In {\em International Workshop on Computer Science Logic}, pages
  1--30. Springer, 2006.
\newblock URL:
  \url{https://www.irif.fr/~mellies/mpri/mpri-ens/articles/mellies-functorial-boxes.pdf}.

\bibitem{Musto2018}
Benjamin Musto, David Reutter, and Dominic Verdon.
\newblock A compositional approach to quantum functions.
\newblock {\em Journal of Mathematical Physics}, 59(8):081706, 2018.
\newblock \href {http://arxiv.org/abs/1711.07945} {\path{arXiv:1711.07945}}.

\bibitem{Musto2019}
Benjamin Musto, David Reutter, and Dominic Verdon.
\newblock The {M}orita theory of quantum graph isomorphisms.
\newblock {\em Communications in Mathematical Physics}, 365(2):797--845, 2019.
\newblock \href {http://arxiv.org/abs/1801.09705} {\path{arXiv:1801.09705}}.

\bibitem{Neshveyev2013}
Sergey Neshveyev and Lars Tuset.
\newblock {\em Compact Quantum Groups and Their Representation Categories}.
\newblock Collection SMF.: Cours sp{\'e}cialis{\'e}s. Soci{\'e}t{\'e}
  Math{\'e}matique de France, 2013.

\bibitem{Selinger2010}
Peter Selinger.
\newblock A survey of graphical languages for monoidal categories.
\newblock In {\em New {S}tructures for {P}hysics}, pages 289--355. Springer,
  2010.
\newblock \href {http://arxiv.org/abs/0908.3347} {\path{arXiv:0908.3347}}.

\bibitem{Verdon2020b}
Dominic Verdon.
\newblock Entanglement equivalences of covariant channels.
\newblock 2020.
\newblock \href {http://arxiv.org/abs/2012.05761} {\path{arXiv:2012.05761}}.

\bibitem{Verdon2020a}
Dominic Verdon.
\newblock Unitary pseudonatural transformations.
\newblock 2020.
\newblock \href {http://arxiv.org/abs/2004.12760} {\path{arXiv:2004.12760}}.

\bibitem{Woronowicz1988}
Stanis{\l}aw~L. Woronowicz.
\newblock Tannaka-{K}rein duality for compact matrix pseudogroups. twisted
  {SU(N)} groups.
\newblock {\em Inventiones mathematicae}, 93(1):35--76, 1988.
\newblock URL:
  \url{http://resolver.sub.uni-goettingen.de/purl?PPN356556735_0093}.

\end{thebibliography}


\section{Appendix 1: Rigid $C^*$-tensor categories are pivotal dagger categories}\label{sec:app}
Here we complete the proof of Theorem~\ref{thm:c*tenspivdag}, which states that a $C^*$-tensor category equipped with standard solutions to the conjugacy equations is a pivotal dagger category. We first recall the relevant results from~\cite{Neshveyev2013}.

\begin{definition}[{\cite[Theorem 2.2.16]{Neshveyev2013}}]
Let $\mathcal{C}$ be a $C^*$-tensor category with conjugates and let $X$ be some object. Then a \emph{standard solution of the conjugate equations} for $X$ is a choice of $[X^*,R,\bar{R}]$ as in Definition~\ref{def:conjeqns} such that 
\begin{align}
R^{\dagger} (\id_{X^*} \otimes f) R = \bar{R}^{\dagger} (f \otimes \id_{X^*}) \bar{R}
\end{align}
for every morphism $f: X \to X$.
\end{definition}
\ignore{
In a $C^*$-tensor category with conjugates a  standard solution of the conjugate equations may be picked for any object~\cite[Definition 2.2.14]{Neshveyev2013}.}
\begin{proposition}[{\cite[Proposition 2.2.15]{Neshveyev2013}}]\label{prop:standsolsunitary}
Let $[X^*,R,\bar{R}]$ and $[X^*{}',R',\bar{R}']$ be two standard solutions of the conjugate equations for an object $X$. Then there exists a unitary $T: X^* \to X^*{}'$ such that $R' = (T \otimes \id_R) R$ and $\bar{R}' = (\id_R \otimes T) \bar{R}$.
\end{proposition}
\begin{lemma}[{\cite[Proof of Theorem 2.2.18]{Neshveyev2013}}]\label{lem:tensorstandsols}
Let $(R_X,\bar{R}_X)$ be a standard solution of the conjugate equations for $X$, and $(R_Y,\bar{R}_Y)$ be a standard solution of the conjugate equations for $Y$. Then the following choice of $R$ and $\bar{R}$ are a standard solution of the conjugate equations for $X \otimes Y$:
\begin{align}
\label{eq:tensorprodconj}
R = (\id_{Y^*} \otimes R_X \otimes \id_Y) R_Y
&&
\bar{R} = (\id_{X} \otimes \bar{R}_Y \otimes \id_{X^*}) \bar{R}_X
\end{align}
\end{lemma}
\noindent
The following additional result is clear.
\begin{lemma}\label{lem:idstandsols}
The morphisms $R_{\mathbbm{1}}:= \id_{\mathbbm{1}}: \mathbbm{1} \to \mathbbm{1} \otimes \mathbbm{1}$ and $\bar{R}_{\mathbbm{1}}:= \id_{\mathbbm{1}}: \mathbbm{1} \to \mathbbm{1} \otimes \mathbbm{1}$ are standard solutions to the conjugacy equations for the object $\mathbbm{1}$.
\end{lemma}
\noindent
In a $C^*$-tensor category $\mathcal{C}$ where standard solutions to the conjugate equations have been chosen, one may consider the usual contravariant right duals functor. 
\begin{theorem}[{\cite[Theorem 2.2.21]{Neshveyev2013}}]\label{thm:natiso}
Where standard solutions to the conjugate equations have been chosen, the right duals functor is unitary and its square is naturally isomorphic to the identity functor.
\end{theorem}
\begin{proof}
The definition of the natural isomorphism is important for our considerations. For any object $X$, both $(\bar{R}_{X},R_{X})$ and $(R_{X^*},\bar{R}_{X^*})$ are standard solutions to the conjugate equations for $X^*$. Therefore, by Proposition~\ref{prop:standsolsunitary}, we obtain a unitary $\iota_X: X^{**} \to X$ such that:
\begin{align}
\label{eq:iota}
\bar{R}_X = (\iota_X \otimes \id_{X^*}) R_{X^*}
&&
R_{X}^{\dagger} = \bar{R}_{X^*}^{\dagger} (\id_{X^*} \otimes \iota_X^{\dagger})
\end{align}
These unitaries $\{\iota_X\}$ are the components of the natural isomorphism from the double duals functor to the identity functor.
\end{proof}
\noindent
It remains only to show that the unitary natural isomorphism $\iota$ is monoidal. The double duals functor may be given a monoidal structure as follows. First we observe that, since~\eqref{eq:tensorprodconj} are a standard solution to the conjugacy equations for $X \otimes Y$, by Proposition~\ref{prop:standsolsunitary} there exists a unitary $\mu_{X,Y}: Y^* \otimes X^* \to (X \otimes Y)^*$ such that:
\begin{calign}\label{eq:multiplicators}
\includegraphics[scale=.5]{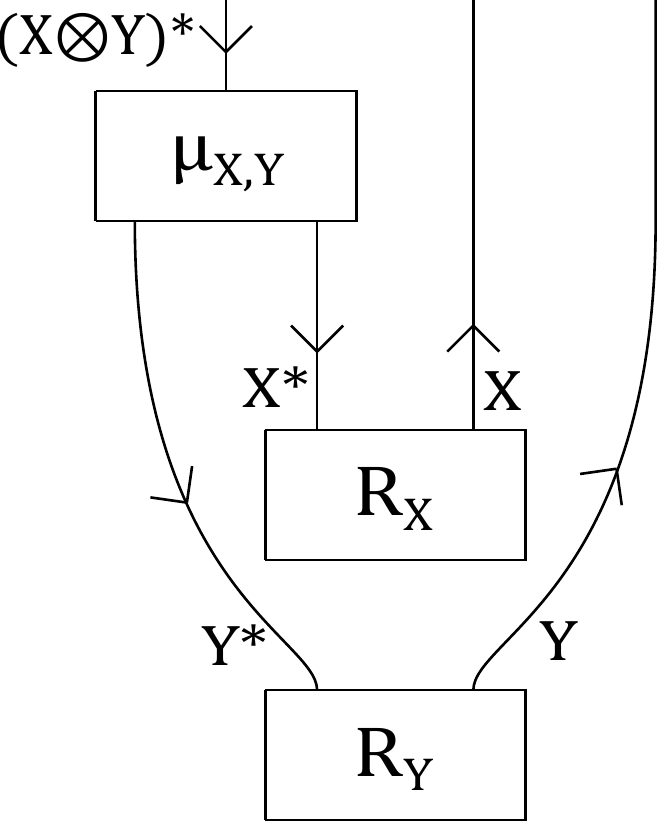}
=
\includegraphics[scale=.5]{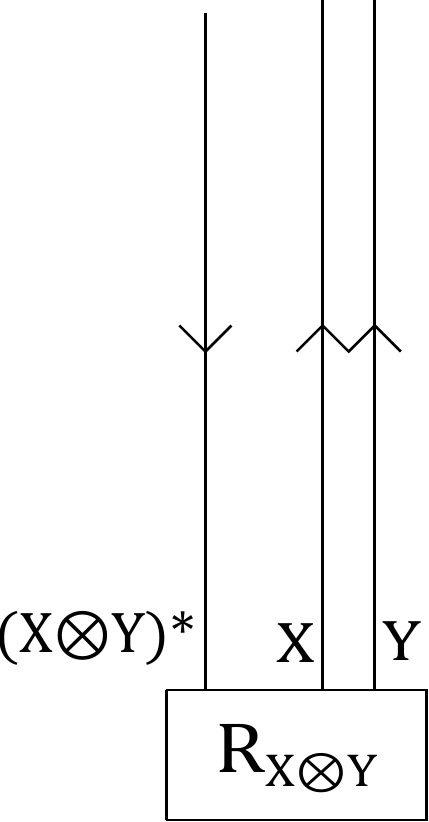}
&&
\includegraphics[scale=.5]{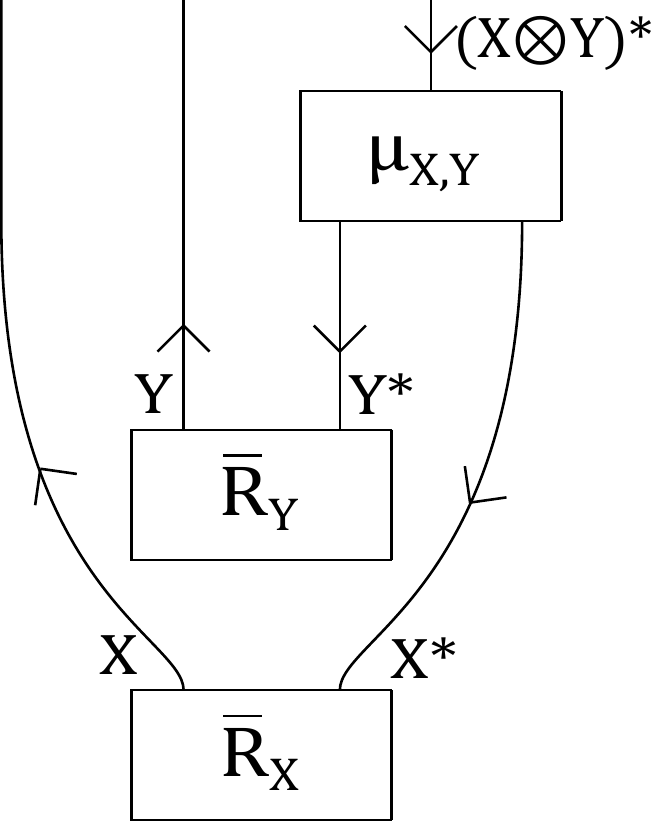}
=
\includegraphics[scale=.5]{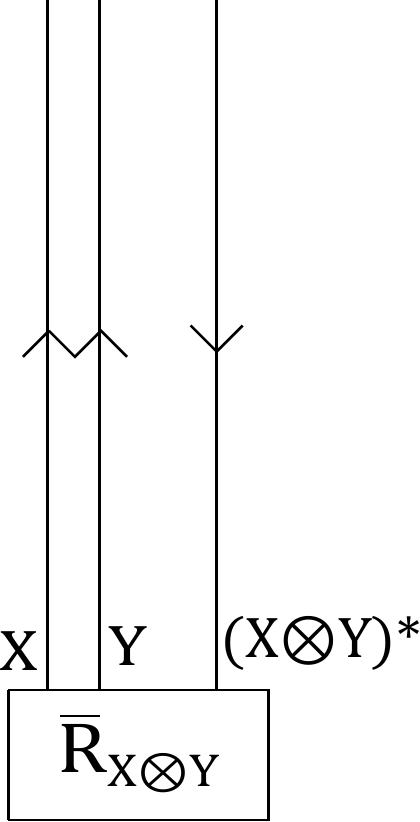}
\end{calign}
We then define the multiplicators of the double duals functor as follows, for any pair of objects $X,Y$ of $\mathcal{C}$ (using a double upwards arrow to denote a double dual object):
\begin{calign}
\includegraphics[scale=.5]{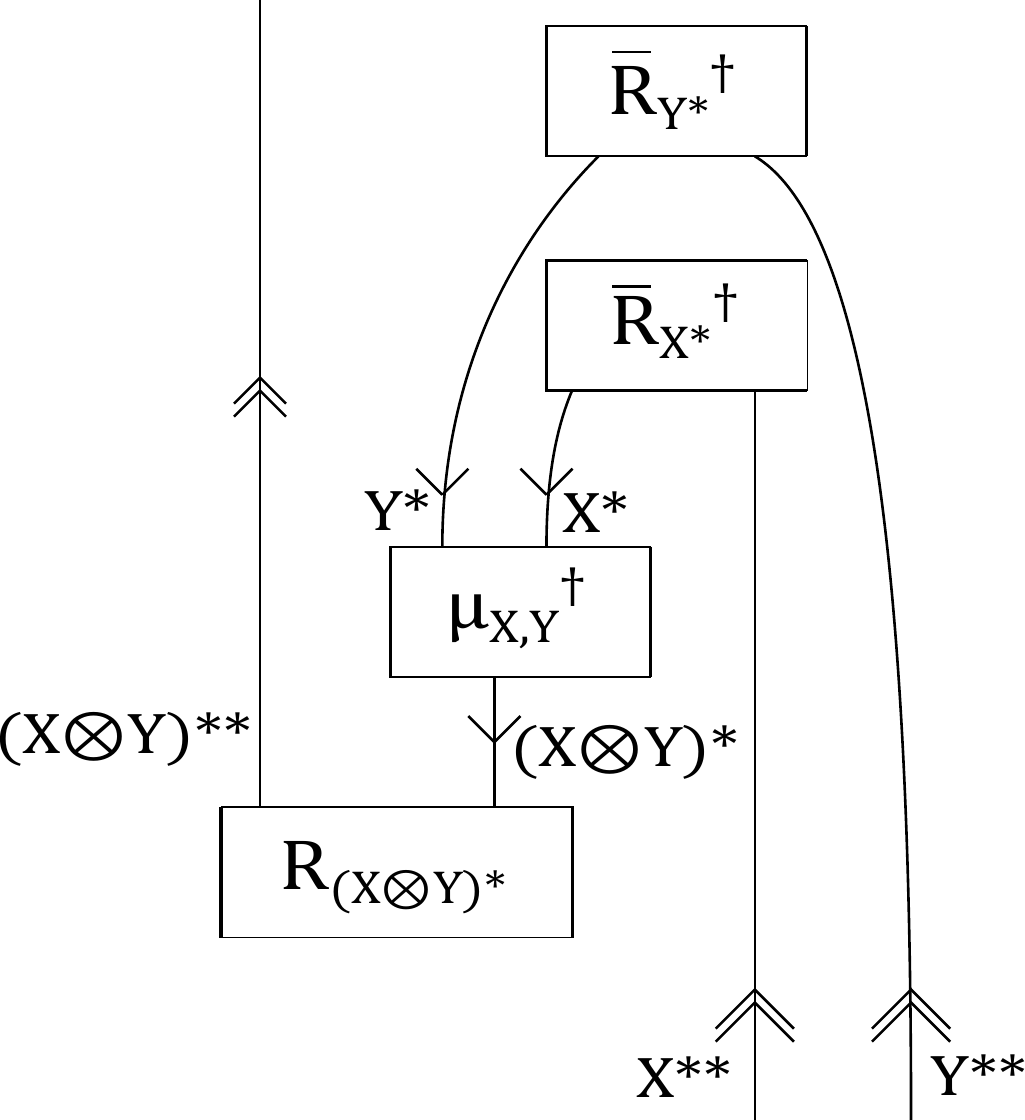}
\end{calign}
We also observe that, by Lemma~\ref{lem:idstandsols} and Proposition~\ref{prop:standsolsunitary}, there exists a unitary $\nu:  \mathbbm{1}^* \to \mathbbm{1}$ such that:
\begin{calign}
\nu \circ R_{\mathbbm{1}} = \id_{\mathbbm{1}} 
&&
\nu \circ \bar{R}_{\mathbbm{1}} = \id_{\mathbbm{1}}
\end{calign}
This simply implies that $\nu^{\dagger} =R_{\mathbbm{1}} = \bar{R}_{\mathbbm{1}}$. 
We then define the unitor of the double duals functor as
$$
(\id_{\mathbbm{1}^{**}} \otimes {R}_{\mathbbm{1}}^{\dagger})R_{\mathbbm{1}^*}.
$$
It is straightforward to check that the unitor and multiplicators just defined are a unitary monoidal structure for the double duals functor.

We also have the following lemma.
\begin{lemma}\label{lem:iota}
For every object $X$ we have equalities
\begin{align*}
\iota_X = (\id_{X} \otimes \bar{R}_{X^*}^{\dagger}) (\bar{R}_{X} \otimes \id_{X^{**}})
&&
\iota_X^{\dagger} = (\id_{X^{**}} \otimes R_X^{\dagger}) (R_{X^*} \otimes \id_X)
\end{align*} 
\end{lemma}
\begin{proof}
For the first equality:
\begin{align*}
\iota_X = (\iota_X \otimes R_X^{\dagger} \otimes \bar{R}_{X^*}^{\dagger})(R_{X^*} \otimes \bar{R}_X \otimes \id_{X^{**}}) &= (\id_{X} \otimes R_X^{\dagger} \otimes \bar{R}_{X^*}^{\dagger}) (\bar{R}_X \otimes \bar{R}_X \otimes \id_{X^{**}}) \\
&=(\id_{X} \otimes \bar{R}_{X^*}^{\dagger}) (\bar{R}_{X} \otimes \id_{X^{**}})
\end{align*}
The second equality is shown similarly.
\end{proof}
\begin{proposition}
The unitary natural isomorphism of Theorem~\ref{thm:natiso} is monoidal. 
\end{proposition} 
\begin{proof}
Compatibility with the multiplicators is seen as follows:
\begin{calign}
\includegraphics[scale=.5]{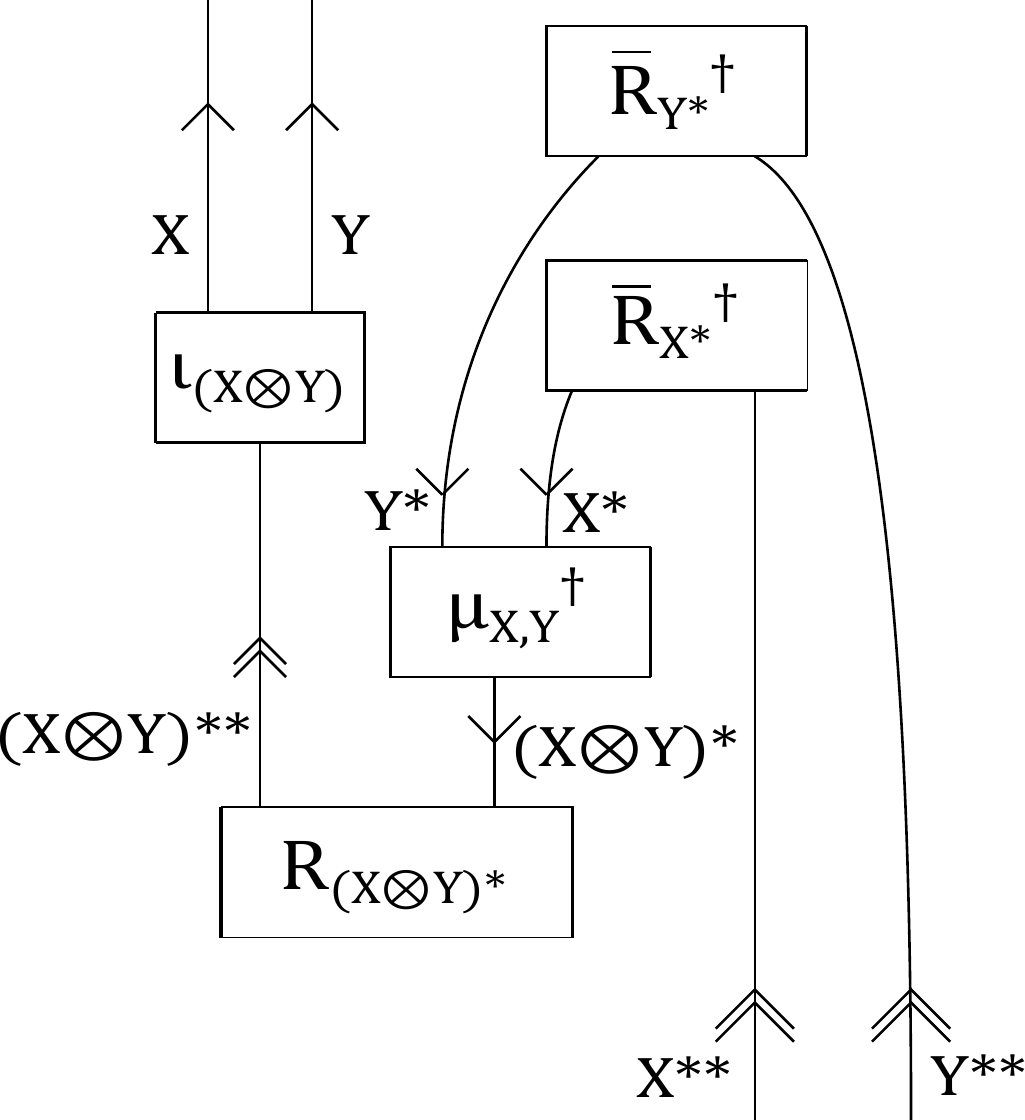}
~~=~~
\includegraphics[scale=.5]{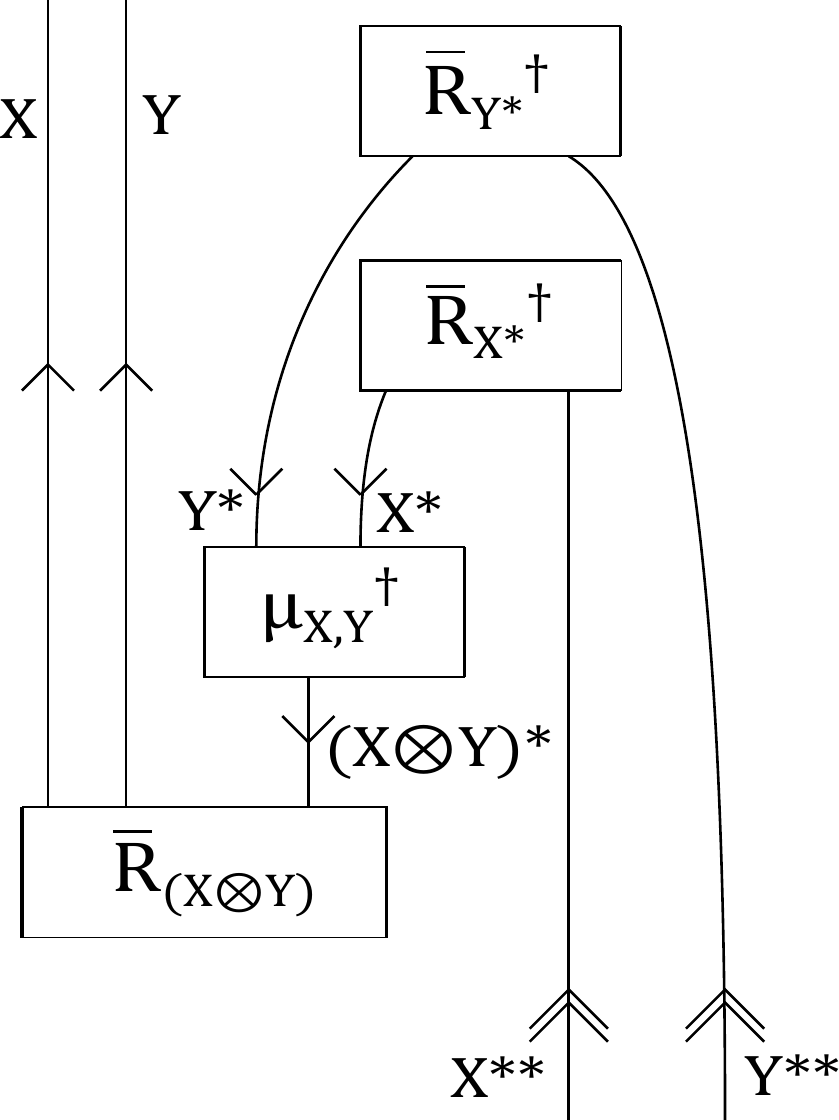}
\\=~~
\includegraphics[scale=.5]{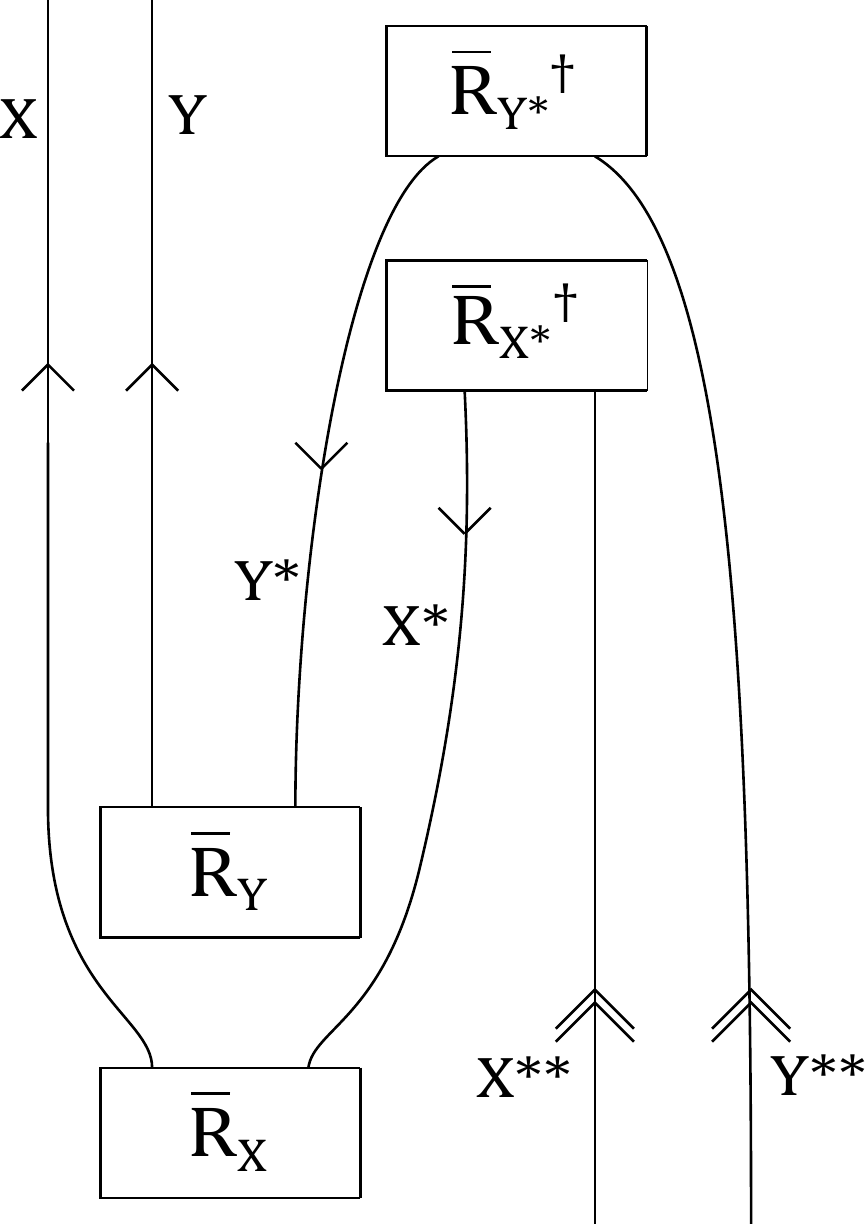}
~~=~~
\includegraphics[scale=.5]{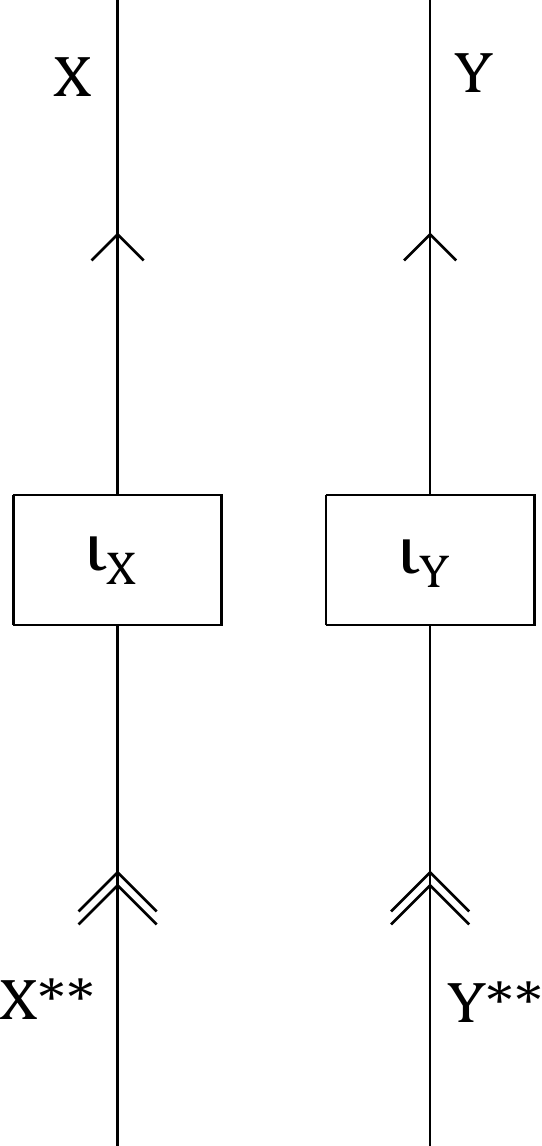}
\end{calign}
Here the first equality is by~\eqref{eq:iota}; the second equality is by~\eqref{eq:multiplicators}; and the third equality is by Lemma~\ref{lem:iota}.

For compatibility with the unitor we observe that, by Lemma~\ref{lem:iota}, the unitor is precisely $\iota_{\mathbbm{1}}^{\dagger}$. But then postcomposing by $\iota_{\mathbbm{1}}$ we have $\iota_{\mathbbm{1}} \iota_{\mathbbm{1}}^{\dagger} = \id_{\mathbbm{1}}$, as desired.
\end{proof}
\noindent
The proof that a rigid $C^*$-tensor category equipped with standard solutions to the conjugacy equations is a pivotal dagger category is therefore complete.

\end{document}